\documentclass[12pt]{article}
\usepackage{times}

\usepackage[utf8]{inputenc}
\usepackage{commath}
\usepackage{amsthm}
\usepackage{amsmath}
\usepackage{amssymb}
\usepackage{cite}
\usepackage{setspace}

\usepackage{bigints}
\usepackage{comment}
\usepackage{mathtools}
\usepackage{mathrsfs}
\usepackage{fancyhdr}

\allowdisplaybreaks[4]

\numberwithin{equation}{section}
\usepackage{etoolbox}
\patchcmd{\thebibliography}
  {\settowidth}
  {\setlength{\itemsep}{0pt plus -10pt}\settowidth}
  {}{}
\apptocmd{\thebibliography}
  {
  }
  {}{}
\makeatletter
\newtheorem*{rep@theorem}{\rep@title}
\newcommand{\newreptheorem}[2]{%
\newenvironment{rep#1}[1]{%
 \def\rep@title{#2 \ref{##1}}%
 \begin{rep@theorem}}%
 {\end{rep@theorem}}}
\makeatother

\theoremstyle{theorem}

\newreptheorem{theorem}{Theorem}
\newtheorem{thm}{Theorem}[section]
\newtheorem*{thm*}{Theorem}
\theoremstyle{definition}
\newtheorem{prop}[thm]{Proposition}
\newtheorem*{prop*}{Proposition}
\newtheorem{defn}[thm]{Definition}
\newtheorem{lem}[thm]{Lemma}

\newtheorem*{cor*}{Corollary}
\theoremstyle{remark}
\newtheorem{rem}[thm]{Remark}

\newcounter{tmp}
\title{Analytic torsion for surfaces with cusps I. \\
Compact perturbation theorem and anomaly formula.} 

\author
{Siarhei Finski
}

\date{}

\usepackage[%
    left=1in,%
    right=1in,%
    top=1.1in,%
    bottom=0.8in,%
    paperheight=11in,%
    paperwidth=8.5in%
]{geometry}

\newcommand{\imun} {\sqrt{-1}}
\newcommand{\vol}{v}

\newcommand{\comp}{\mathbb{C}}
\newcommand{\real}{\mathbb{R}}

\newcommand{\nat}{\mathbb{N}}
\newcommand{\integ}{\mathbb{Z}}
\newcommand{\cp}{\mathbb{C}{\rm{P}}^1}
\newcommand{\spec}{{\rm{Spec}}}
\newcommand{\dist}{{\rm{d}}}
\newcommand{\reg}{{\textbf{r}}}

\newcommand{\tinyint}{\begingroup\textstyle\int\endgroup}

\newcommand{\enmr}[1]{\text{End}{(#1)}}

\newcommand{\ccal}{\mathscr{C}}

\newcommand{\dbar}{ \overline{\partial} }

\newcommand{\laplcomp}{\Box}

\newcommand{\rk}[1]{{\rm{rk}} ( #1 )}
\newcommand{\tr}[1]{{\rm{Tr}} \big[ #1 \big]}

\renewcommand{\Re}{\operatorname{Re}}
\renewcommand{\Im}{\operatorname{Im}}
\newcommand{\scal}[2]{\big< #1, #2 \big>}

\newcommand{\td}{{\rm{Td}}}
\newcommand{\ch}{{\rm{ch}}}

\newcommand{\dd}{\mathbb{D}}
\newcommand{\hh}{\mathbb{H}}

\newcommand{\Addresses}{{
  \bigskip
  \footnotesize
  \noindent \textsc{Siarhei Finski, UFR de Mathématiques, Case 7012, Université Paris Diderot-Paris 7, France.}\par\nopagebreak
  \noindent  \textit{E-mail }: \texttt{siarhei.finski@imj-prg.fr}.
}}

\newenvironment{sciabstract}{}

\begin{document} 

\maketitle

\begin{sciabstract}
  \textbf{Abstract.}
   Let $\overline{M}$ be a compact Riemann surface and let $g^{TM}$ be a metric over $\overline{M} \setminus D_M$, where $D_M \subset \overline{M}$ is a finite set of points. We suppose that $g^{TM}$ is equal to the Poincaré metric over a punctured disks around the points of $D_M$. The metric $g^{TM}$ endows the twisted canonical line bundle $\omega_M(D)$ with the induced Hermitian norm $\norm{\cdot}_M$ over $\overline{M} \setminus D_M$. Let $(\xi, h^{\xi})$ be a holomorphic Hermitian vector bundle over $\overline{M}$. 
  \par
  \begin{sloppypar}
 In this article we define the analytic torsion $T(g^{TM}, h^{\xi} \otimes \norm{\cdot}_M^{2n})$ associated with $(M, g^{TM})$ and $(\xi \otimes \omega_M(D)^n, h^{\xi} \otimes \norm{\cdot}_M^{2n})$ for $n \leq 0$.
 We prove that $T(g^{TM}, h^{\xi} \otimes \norm{\cdot}_M^{2n})$ is related to the analytic torsion of non-cusped surfaces.
   Then we prove the anomaly formula for the associated Quillen norm.
  The results of this paper will be used in the sequel to study the regularity of the Quillen norm and its asymptotics in a degenerating family of Riemann surfaces with cusps and to prove the curvature theorem. We also prove that our definition of the analytic torsion for hyperbolic surfaces is compatible with the one obtained through Selberg trace formula by Takhtajan-Zograf.
     \end{sloppypar}
\end{sciabstract}

\tableofcontents

\section{Introduction}\label{sect_intro}
	Let $\overline{M}$ be a compact Riemann surface, $D_M = \{ P_1^{M}, \ldots, P_m^{M} \}$ be a finite set of distinct points in $\overline{M}$. Let $g^{TM}$ be a Kähler metric on the punctured Riemann surface $M := \overline{M} \setminus D_M$. 
	\par For $\epsilon \in ]0,1]$, $i = 1, \ldots, m$, let $z_i^{M} : \overline{M} \supset V_i^M(\epsilon) \to D(\epsilon) := \{ z \in \comp : |z| \leq \epsilon \}$ be a local holomorphic coordinate around $P_i^{M}$. We denote
	\begin{equation}\label{defn_v_i}
		V_i^{M}(\epsilon) := \{x \in M :  |z_i^{M}(x)| < \epsilon \}.
	\end{equation}	 
	We say that $g^{TM}$ is \textit{Poincaré-compatible} with coordinates $z_1^{M}, \ldots, z_m^{M}$ if for any $i = 1, \ldots, m$,  there is $
	\eta > 0$ such that $g^{TM}|_{V_i^{M}(\eta)}$ is induced by the Hermitian form
	\begin{equation}\label{reqr_poincare}
		\frac{\imun dz_i^{M} d\overline{z}_i^{M}}{ \big| z_i^{M}  \ln |z_i^{M}| \big|^2}.
	\end{equation}
	We say that $g^{TM}$ is a \textit{metric with cusps} if it is Poincaré-compatible with some holomorphic coordinates near $D_M$.
	A triple $(\overline{M}, D_M, g^{TM})$ of a Riemann surface $\overline{M}$, a set of punctures $D_M$ and a metric with cusps $g^{TM}$ is called a \textit{surface with cusps} (cf. \cite{MullerCusp}). 
	\par For example, if a pointed surface $(\overline{M}, D_M)$ is stable, i.e. the genus $g(\overline{M})$ of $\overline{M}$ satisfies 
	\begin{equation}\label{cond_stable}
		2 g(\overline{M}) - 2 + m > 0, 
	\end{equation}
	 then, by the uniformization theorem (cf. \cite[Chapter IV]{FarKra}, \cite[Lemma 6.2]{Auvr}), there is the \textit{canonical hyperbolic metric} $g^{TM}_{\rm{hyp}}$ of constant scalar curvature $-1$ on $M$. 
	 Once again, by the uniformization theorem, there are local holomorphic coordinates $z_i^{M}$ of $P_i^{M}$, $i = 1, \ldots, m$, such that $g^{TM}_{\rm{hyp}}$ is induced by (\ref{reqr_poincare}) in the neighbourhood of $D_M$.
	 Thus, $(\overline{M}, D_M, g^{TM}_{\rm{hyp}})$ is a surface with cusps.
	\par 	Let $\xi$ be a holomorphic vector bundle over a complex manifold $X$ with a Hermitian metric $h^{\xi}$ over $X$. A pair $(\xi, h^{\xi})$ is called a \textit{Hermitian vector bundle} over $X$. 	
	\par From now on, we fix a surface with cusps $(\overline{M}, D_M, g^{TM})$ and a Hermitian vector bundle $(\xi, h^{\xi})$ over it. 
	We denote by $\omega_{\overline{M}} := T^{*(1,0)}\overline{M}$ the \textit{canonical line bundle} over $\overline{M}$. 
	Let $\mathscr{O}_{\overline{M}}(D_M)$ be the line bundle associated to the divisor $D_M$. 	
	The \textit{twisted canonical line bundle} on $\overline{M}$ is defined as 
	\begin{equation}\label{eq_om_md_def}
			\omega_M(D) :=  \omega_{\overline{M}} \otimes  \mathscr{O}_{\overline{M}}(D_M).
	\end{equation}		
	The metric $g^{TM}$ endows the line bundle $\omega_M$ (resp. $\omega_M(D)$) with the induced Hermitian metric $\norm{\cdot}^{\omega}_{M}$ (resp. $\norm{\cdot}_{M}$ via the canonical isomorphism $\omega_M(D) \simeq \omega_M$) over $M$.
	In other worlds, there is $\epsilon > 0$, such that for the canonical section $s_{D_M}$ of $\mathscr{O}_{\overline{M}}(D_M)$, over $V_i^{M}(\epsilon)$, we have
	\begin{equation}\label{eqn_norms_local}
		\big\| dz_i^{M}\big\| _M^{\omega} = \big| z_i^{M} \ln |z_i^{M}| \big|, \qquad 
		\big\| dz_i^{M} \otimes s_{D_M} / z_i^{M} \big\|_M = \big| \ln |z_i^{M}| \big|.
	\end{equation}
	\begin{sloppypar}
	We denote by $\laplcomp^{\xi \otimes \omega_M(D)^n}$ the Kodaira Laplacian associated with $(M, g^{TM})$ and $(\xi \otimes \omega_M(D)^n, h^{\xi} \otimes \norm{\cdot}_M^{2n})$.
	\par
	In this article, apart from the discussion of the $L^2$-norm, we only consider the restriction of $\laplcomp^{\xi \otimes \omega_M(D)^n}$ on the sections of degree $0$.
	\par
	Assume first $m=0$, then the analytic torsion was defined by Ray-Singer \cite[Definition 1.2]{Ray73} as the regularized determinant of $\laplcomp^{\xi \otimes \omega_M(D)^n}$. More precisely, let $\lambda_i, i \in \nat$ be the non-zero eigenvalues of $\laplcomp^{\xi \otimes \omega_M(D)^n}$. By Weyl's law, the associated zeta-function 
	\begin{equation}\label{defn_zeta_comp}
		\zeta_M(s) := \sum \lambda_i^{s},
	\end{equation}
	 is defined for $\Re (s) > 1$ and it is holomorphic in this region. Moreover, as it can be seen by the small-time expansion of the heat kernel and the usual properties of the Mellin transform, it extends meromorphically to the entire $s$-plane. This extension is holomorphic at $0$, and the \textit{analytic torsion} is defined by
	\begin{equation}\label{defn_an_t_st}
		T(g^{TM}, h^{\xi} \otimes \norm{\cdot}_M^{2n}) := \exp(- \zeta'_M(0)).
	\end{equation}
	By (\ref{defn_zeta_comp}) and (\ref{defn_an_t_st}), we may interpret the analytic torsion as
	\begin{equation}\label{an_tors_intepr}
		T(g^{TM}, h^{\xi} \otimes \norm{\cdot}_M^{2n}) := \prod_{i=0}^{\infty} \lambda_i.
	\end{equation}
	Now, assume $m > 0$. Then $M$ is non-compact, and the heat operator associated to $\laplcomp^{\xi \otimes \omega_M(D)^n}$ is no longer of trace class. Also the spectrum of $\laplcomp^{\xi \otimes \omega_M(D)^n}$ is not discrete in general. Thus, neither the definition (\ref{defn_an_t_st}), nor the interpretation (\ref{an_tors_intepr}) are applicable, and another approach should be used.
	\end{sloppypar}
	\par Suppose for the moment that $(\overline{M}, D_M)$ satisfies (\ref{cond_stable}). Let $g^{TM}_{\rm{hyp}}$ be the canonical hyperbolic metric of constant scalar curvature $-1$. We denote by $Z_{(\overline{M}, D_M)}(s), s \in \comp$ the Selberg zeta-function, which is given for $\Re(s) > 1$ by the absolutely converging product:
	\begin{equation}\label{defn_zeta_sel}
		Z_{(\overline{M}, D_M)}(s) = \prod_{\gamma} \prod_{k=0}^{\infty} (1 - e^{-(s + k)l(\gamma)}),
	\end{equation}
	where $\gamma$ runs over the set of all simple closed geodesics on $M$ with respect to $g^{TM}_{\rm{hyp}}$, and $l(\gamma)$ is the length of $\gamma$.
	The function $Z_{(\overline{M}, D_M)}(s)$ admits a meromorphic extension to the whole complex $s$-plane with a simple zero at $s = 1$ (see for example \cite[(5.3)]{PhongHook}). We denote by $\norm{\cdot}_{M}^{\rm{hyp}}$ the norm induced by $g^{TM}_{\rm{hyp}}$ on $\omega_M(D)$ over $M$. 
	\par In this situation Takhtajan-Zograf in \cite[(6)]{TakZog} defined the analytic torsion by
	\begin{equation}\label{eqn_sel_norm}
		T_{TZ}(g^{TM}_{\rm{hyp}}, (\, \norm{\cdot}_{M}^{\rm{hyp}})^{2n}) := 
		\begin{cases}
			\exp(- c_{0} \chi(M)/2) \cdot Z_{(\overline{M}, D_M)}'(1), &\text{for } n=0, \\
			\exp(- c_{-n} \chi(M)/2) \cdot Z_{(\overline{M}, D_M)}(-n + 1), &\text{for } n<0,
		\end{cases}
	\end{equation}
	where for $k \in \nat^*$, we put
	\begin{equation}\label{defn_c_k_coeff}
	\begin{aligned}
		& \textstyle c_0 = 4 \zeta'(-1) - \frac{1}{2} + \ln(2 \pi),
		\\
		&
		\textstyle c_k = \sum_{l = 0}^{k - 1} (2k - 2l - 1) \big( \ln (2k + 2kl - l^2 - l) - \ln (2) \big) + 
		(\frac{1}{3} + k + k^2 ) \ln(2) 
		\\		
		& \qquad \qquad \qquad
		\textstyle + (2k + 1)\ln (2 \pi) + 4 \zeta'(-1) 
		- 2(k + \frac{1}{2})^2 
		- 4 \sum_{l = 1}^{k - 1} \ln(l!) - 2 \ln (k!).
	\end{aligned}
	\end{equation}
	\begin{rem}
		To explain the values $c_k$, $k \in \nat$, it was shown by Phong-D'Hoker \cite[(7.30)]{PhongHook}, \cite[(3.6)]{PhongHook2nd} (see also \cite{sarnakDet}, \cite[(50)]{BolStei} and \cite[(9)]{Oshima}), that the definition (\ref{eqn_sel_norm}) coincides with (\ref{defn_an_t_st})\footnote{It's easy to see that $T(g^{TM}_{\rm{hyp}}, (\, \norm{\cdot}_{M}^{\rm{hyp}})^{2n})$ corresponds to $\det' (\Delta_n^{-})$ in the notation of \cite[(1.1)]{PhongHook2nd} and to $\det' (\frac{1}{2} \Delta_n^{-})$ in the notation of \cite[(3)]{BolStei}. Since for $c > 0$, by \cite[\S 3]{PhongHook2nd}, we have $\det' (c \Delta_n^{-}) = \det' (c \Delta_{-n}^{+})$, coefficients (\ref{defn_c_k_coeff}) for $k \in \nat^*$ can be read of from \cite[(50)]{BolStei} for $c = 1/2$ and for $k = 0$ from \cite[(7.23), (7.30)]{PhongHook} or \cite[Corollary 1]{sarnakDet}.}. In other words, the two definitions are compatible for $M$ stable, $m=0$, $g^{TM} = g^{TM}_{\rm{hyp}}$ and  $n \leq 0$.
	\end{rem}
	The advantage of the definition (\ref{eqn_sel_norm}) is an explicit formula in terms of “simple" geometric objects and, thus, suitability for the variational-type arguments (see \cite{TakZog}, \cite{Fedosova}). However, it only works for the hyperbolic metrics $g^{TM}_{\rm{hyp}},\, \norm{\cdot}_{M}^{\rm{hyp}}$ on $M$ and trivial Hermitian vector bundle $(\xi, h^{\xi})$. 
	\par Our first goal of this article is to give a definition of the analytic torsion $T(g^{TM}, h^{\xi} \otimes \norm{\cdot}_{M}^{2n})$ for $n \leq 0$,\footnote{By Serre duality, if one prefers to work with positive line bundles, we can interpret it as the analytic torsion of the vector bundle $\xi^* \otimes \omega_M^{-n+1}(D_M)^{-n}$ associated to $(g^{TM}, (h^{\xi})^* \otimes \norm{\cdot}_M^{-2n} \otimes (\, \norm{\cdot}^{\omega}_{M})^2)$, for $n \leq 0$.} which generalizes both (\ref{defn_an_t_st}) and (\ref{eqn_sel_norm}), see Definition \ref{defn_rel_tor}. 
	Our definition is done in the spirit of (\ref{defn_an_t_st}), and in \cite{FinII3} we show that it actually coincides with (\ref{eqn_sel_norm}) for hyperbolic surfaces and $(\xi, h^{\xi})$ trivial (thus, extending the results of Phong-D'Hoker \cite[(7.30)]{PhongHook}, \cite[(3.6)]{PhongHook2nd}).
	We also give two results for computing the newly defined analytic torsion.
	The first one, Theorem \ref{thm_comp_appr}, which we also call the \textit{compact perturbation theorem}, expresses the quotient of two Quillen norms associated with surfaces with the same number of cusps through a quotient of two Quillen norms associated with surfaces without cusps. 
	The second one, Theorem \ref{thm_anomaly_cusp}, which we also call the \textit{anomaly formula}, explains how the Quillen norm changes under the change of the metrics $g^{TM}$, $h^{\xi}$.
	The study of the heat kernel associated to $h^{\xi} \otimes \norm{\cdot}_M^{2n}$ on a surface with cusps $(\overline{M}, D_M, g^{TM})$ plays the foremost role in our approach.
	\par 
	We also note that our definition is related to the definition of the relative analytic torsion due to Lundelius and Jorgenson-Lundelius, which was given for $(\xi, h^{\xi})$ trivial and $n = 0$ in \cite{Lun93}, \cite{JorLund}, \cite{JorLundMain}, see Remarks \ref{rem_anomaly_cusp}e), \ref{rem_def_an_tor}c), and the definition of Albin-Rochon, see Remark \ref{rem_def_an_tor}d), which was given for $(\xi, h^{\xi})$ trivial and $n = 0$ in \cite[\S 7.1]{AlbRoch}. 
	The $b$-trace of Melrose \cite{MelroseAPS}, used in the definition of Albin-Rochon, should also give the definition of the analytic torsion in our case, however we work in a relative setting, and $b$-trace does not appear explicitly. This gives us more flexibility to establish some estimates on the heat kernel which are used extensively in the proof of Theorem \ref{thm_comp_appr}.
	\par Now let's describe our results more precisely. For $n \leq 0$, in the end of Section \ref{sect_spec_gap}, we explain how to endow the complex line 
	\begin{multline}\label{defn_det_line}
		\big(\det H^{\bullet}(\overline{M}, \xi \otimes \omega_M(D)^n) \big)^{-1} \\
		:= \big( \Lambda^{\max} H^{0}(\overline{M}, \xi \otimes \omega_M(D)^n) \big)^{-1} \otimes  \Lambda^{\max} H^{1}(\overline{M}, \xi \otimes \omega_M(D)^n),
	\end{multline}
	with the $L^2$-norm $\norm{\cdot}_{L^2}(g^{TM}, h^{\xi} \otimes \norm{\cdot}_M^{2n})$. In the compact case it coincides with the $L^2$-norm induced on the harmonic forms.
	Then we define the Quillen norm on the complex line (\ref{defn_det_line}) by
	\begin{equation}\label{defn_quil}
		 \norm{\cdot}_{Q}(g^{TM}, h^{\xi} \otimes \, \norm{\cdot}_{M}^{2n}) 
		=
		 T(g^{TM}, h^{\xi} \otimes \, \norm{\cdot}_{M}^{2n})^{1/2}
		 \cdot 
		 \norm{\cdot}_{L^2}(g^{TM}, h^{\xi} \otimes \, \norm{\cdot}_{M}^{2n}).
	\end{equation}
	To motivate, when $m = 0$, this coincides with the usual definition of the Quillen norm from \cite[1.64]{BGS1} and \cite[Definition 1.5]{BGS3}.
	\noindent
	\begin{align}\label{data_rel_tors}
		\begin{split}
			& \qquad \text{The data for the \textit{compact perturbation theorem}  is: \textit{Riemann surfaces with cusps} }
			\\
			&\text{$(\overline{M}, D_M, g^{TM})$, $(\overline{N}, D_N, g^{TN})$ with $\# (D_M) = \# (D_N) =: m$; a Hermitian vector }\quad \, 
			\\
			&\text{bundle $(\xi, h^{\xi})$ over $\overline{M}$ of rank $\rk{\xi}$; the norms $\norm{\cdot}_{M}, \norm{\cdot}_{N}$ induced  by $g^{TM}$, $g^{TN}$}
			\\
			&\text{as in (\ref{eqn_norms_local}) on $\omega_M(D)$ and $\omega_N(D)$ over $M$ and $N$, and a number $n \in \integ, n \leq 0$.}
		\end{split}
	\end{align}
	\begin{defn}[Flattening of a metric]\label{defn_flat_metr}
		We say that a (smooth) metric $g^{TM}_{\rm{f}}$ over $\overline{M}$ is a \textit{flattening} of $g^{TM}$ if there is $\nu > 0$ such that $g^{TM}$ is induced by (\ref{reqr_poincare}) over $V_i^{M}(\nu)$, and 
	\begin{equation}\label{fl_exterior}
		g^{TM}_{\rm{f}}|_{M \setminus (\cup_i V_i^{M}(\nu))} = g^{TM}|_{M \setminus (\cup_i V_i^{M}(\nu))}.
	\end{equation}
	The supremum of all $\nu > 0$, satisfying (\ref{fl_exterior}) is called the \textit{tightness} of the flattening.
	\end{defn}
	\begin{figure}[h]
		\includegraphics[width=\textwidth]{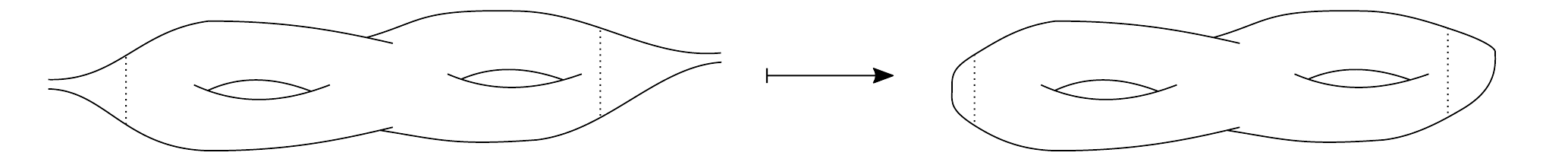}	
		\caption{An example of a flattening. The regions between the dashed lines are isometric.}
	\end{figure}
	\par Let $(\overline{N}, D_N, g^{TN})$  be another surface with cusps and let $g^{TN}_{\rm{f}}$ be a \textit{flattening} of $g^{TN}$. We say that the flattenings $g^{TM}_{\rm{f}}$ and $g^{TN}_{\rm{f}}$ are \textit{compatible}, if for any $i  = 1, \ldots, m$, we have
	\begin{equation}\label{eq_gtm_compat_cusp}
		((z_i^{N})^{-1} \circ z_i^{M})^* (g^{TM}_{\rm{f}}|_{V_i^{M}(\nu)}) = g^{TN}_{\rm{f}}|_{V_i^{N}(\nu)},
	\end{equation}
	for some $\nu > 0$, satisfying (\ref{fl_exterior}) and 
	\begin{equation}\label{eq_comp_gtn}
	g^{TN}_{\rm{f}}|_{N \setminus (\cup_i V_i^{N}(\nu))} = g^{TN}|_{N \setminus (\cup_i V_i^{N}(\nu))}.
	\end{equation} 
	Similarly, we define the notion of \textit{flattenings} $\norm{\cdot}_{M}^{\rm{f}}, \norm{\cdot}_{N}^{\rm{f}}$ for Hermitian norms $\norm{\cdot}_{M}, \norm{\cdot}_{N}$. We say that the flattenings $\norm{\cdot}_{M}^{\rm{f}}, \norm{\cdot}_{N}^{\rm{f}}$ are \textit{compatible} if they satisfy similar conditions to (\ref{fl_exterior}), (\ref{eq_comp_gtn}), and for any $i = 1, \ldots, m$, we have
	\begin{equation}\label{eq_normm_compat_cusp}
		((z_i^{N})^{-1} \circ z_i^{M})^* \big( \norm{\cdot}_{M} / \norm{\cdot}_{M}^{\rm{f}} \big)|_{V_i^{M}(\nu)} 
		= 
		\big( \norm{\cdot}_{N} / \norm{\cdot}_{N}^{\rm{f}} \big)|_{V_i^{N}(\nu)}.
	\end{equation}
	\begin{rem}
		The definitions of flattenings $g^{TM}_{\rm{f}}$ of $g^{TM}$ and $\norm{\cdot}_{M}^{\rm{f}}$ of $\norm{\cdot}_{M}$ are independent, and there is no relation between them as in (\ref{eqn_norms_local}).
	\end{rem}
	\begin{figure}[h]
		\includegraphics[width=\textwidth]{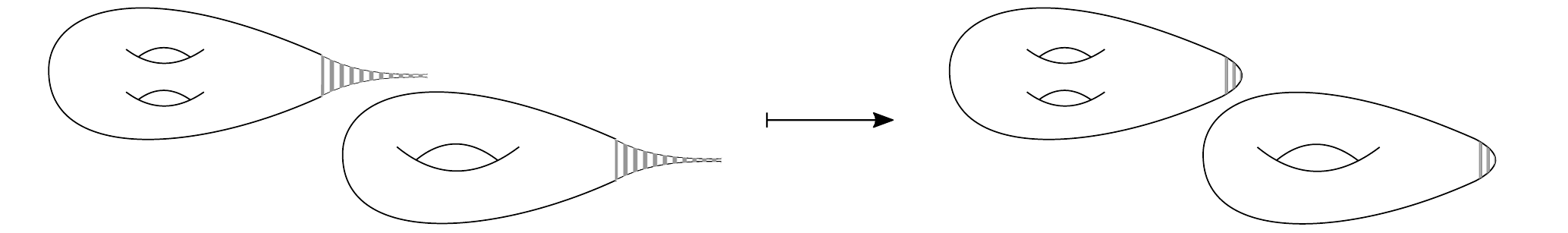}	
		\caption{An example of compatible flattenings. The striped regions are isometric.}
	\end{figure}
	\begin{sloppypar}
	\begingroup
	\setcounter{tmp}{\value{thm}}
	\setcounter{thm}{0} 
	\renewcommand\thethm{\Alph{thm}}
	\begin{thm}[Compact perturbation]\label{thm_comp_appr}
		Let $g^{TM}_{\rm{f}}, g^{TN}_{\rm{f}}, \norm{\cdot}_{M}^{\rm{f}}, \norm{\cdot}_{N}^{\rm{f}}$ be compatible flattenings of $g^{TM}, g^{TN}, \norm{\cdot}_{M}\norm{\cdot}_{N}$ respectively. Then
		\begin{multline}\label{eqn_of_quil_norms}
		2 \ln \Big( 
			\norm{\cdot}_{Q} \big(g^{TM}, h^{\xi} \otimes \, \norm{\cdot}_{M}^{2n}\big) 
			\big / 
			\norm{\cdot}_{Q} \big(g^{TM}_{\rm{f}}, h^{\xi} \otimes (\, \norm{\cdot}_{M}^{\rm{f}})^{2n})
			\Big)
			\\
			- 
		2 \rk{\xi} \ln \Big( 
			\norm{\cdot}_{Q} \big(g^{TN}, \norm{\cdot}_{N}^{2n} \big)
			\big /			
			\norm{\cdot}_{Q} \big(g^{TN}_{\rm{f}}, (\, \norm{\cdot}_{N}^{\rm{f}})^{2n})
			\Big)
			\\	
			 = 
			\int_M c_1(\xi, h^{\xi}) \Big(2n \ln (\, \norm{\cdot}_{M}^{\rm{f}}/ \norm{\cdot}_{M}) +  \ln (g^{TM}_{\rm{f}} / g^{TM}) \Big).
		\end{multline}
			In other words, the relative Quillen norm can be computed through a compact perturbation.
	\end{thm}
	\endgroup
	\setcounter{thm}{\thetmp}
	\end{sloppypar}
	\begin{rem}
	a) The term on the right-hand side of (\ref{eqn_of_quil_norms}) can be motivated by the “anomaly formula", Theorem \ref{thm_anomaly_cusp}. However, we stress out that we cannot apply Theorem \ref{thm_anomaly_cusp} to get Theorem \ref{thm_comp_appr}, since Theorem \ref{thm_comp_appr} “mixes" the compact metrics $g^{TM}_{\rm{f}}$, $g^{TN}_{\rm{f}}$ and the non-compact ones $g^{TM}$, $g^{TN}$. Moreover our proof of Theorem \ref{thm_anomaly_cusp} uses Theorem \ref{thm_comp_appr}.
	Nevertheless, philosophically Theorem \ref{thm_comp_appr}, should be interpreted as the anomaly formula, which permits “erasing" the cusps.
	\par
	To make this analogy even more apparent, we rewrite (\ref{eqn_of_quil_norms}) in the following form (cf. (\ref{ch_bc_0}))
			\begin{multline}\label{eq_thma_rewrite}
			2 \ln  \Big(  
				\norm{\cdot}_{Q} \big(g^{TM}, h^{\xi} \otimes \, \norm{\cdot}_{M}^{2n} \big)
				/
				\norm{\cdot}_{Q} \big(g^{TM}_{\rm{f}}, h^{\xi} \otimes (\, \norm{\cdot}_{M}^{\rm{f}})^{2n} \big)
			\Big)
			\\
			-
			2 \rk{\xi} \ln  \Big(  
				\norm{\cdot}_{Q} \big(g^{TN}, \norm{\cdot}_{N}^{2n} \big)
				/
				\norm{\cdot}_{Q} \big(g^{TN}_{\rm{f}}, (\, \norm{\cdot}_{N}^{\rm{f}})^{2n} \big)
			\Big)
				\\
				=
				\int_M \widetilde{\td} \big(\omega_M^{-1}, g^{TM}_{\rm{f}}, g^{TM} \big) c_1 \big(\xi, h^{\xi} \big)
				+ 
				\int_M 
				c_1 \big(\xi, h^{\xi} \big)
				\widetilde{\ch} \big(\omega_M(D)^n, (\, \norm{\cdot}_M^{\rm{f}})^{2n}, \norm{\cdot}_M^{2n} \big).
		\end{multline}	
	\par b) Suppose that $(\xi, h^{\xi})$ is trivial in the $\nu$-neighbourhood of the cusps, where $\nu > 0$ is the tightness of the flattenings  $g^{TM}_{\rm{f}}$ and $\norm{\cdot}_{M}^{\rm{f}}$.
	Then we simplify Theorem \ref{thm_comp_appr} to
	\begin{equation}\label{eqn_of_quil_norms_trivial_cups}
		\frac{\norm{\cdot}_{Q} \big(g^{TM}, h^{\xi} \otimes \, \norm{\cdot}_{M}^{2n}\big)}{\norm{\cdot}_{Q} \big(g^{TM}_{\rm{f}}, h^{\xi} \otimes (\, \norm{\cdot}_{M}^{\rm{f}})^{2n}\big)}
		=
		\frac{\norm{\cdot}_{Q} \big(g^{TN}, \norm{\cdot}_{N}^{2n} \big)^{\rk{\xi}}}{\norm{\cdot}_{Q} \big(g^{TN}_{\rm{f}}, (\, \norm{\cdot}_{N}^{\rm{f}})^{2n} \big)^{\rk{\xi}}}.
	\end{equation}
	In fact, in our proof of Theorem \ref{thm_comp_appr}, we reduce the main statement to (\ref{eqn_of_quil_norms_trivial_cups}).
	\par
	c) Generally, if we have a family of compatible flattenings $g^{TM}_{\rm{f}, \theta}$, $g^{TN}_{\rm{f}, \theta}$ and $\norm{\cdot}_{M}^{\rm{f}, \theta}$, $\norm{\cdot}_{N}^{\rm{f}, \theta}$, $\theta > 0$ such that they have tightness $\theta > 0$, and $\ln (g^{TM} / g^{TM}_{\rm{f}, \theta})$, $\ln (\, \norm{\cdot}_{M}/ \norm{\cdot}_{M}^{\rm{f}, \theta})$ are uniformly bounded by an integrable function (for an example of such flattenings, see Section \ref{sect_tight} and Appendix \ref{app_2}), then by (\ref{eqn_of_quil_norms}) and Lebesgue dominated convergence theorem, we have
	\begin{equation}\label{eqn_of_quil_norms_limit}
		\frac{\norm{\cdot}_{Q} \big(g^{TM}, h^{\xi} \otimes \, \norm{\cdot}_{M}^{2n}\big)}{\norm{\cdot}_{Q} \big(g^{TN}, \norm{\cdot}_{N}^{2n} \big)^{\rk{\xi}}}
		=			 
		\lim_{\theta \to 0} \, 
		\frac{\norm{\cdot}_{Q} \big(g^{TM}_{\rm{f}, \theta}, h^{\xi} \otimes (\, \norm{\cdot}_{M}^{\rm{f}, \theta})^{2n}\big)}{ \norm{\cdot}_{Q} \big(g^{TN}_{\rm{f}, \theta}, (\, \norm{\cdot}_{N}^{\rm{f}, \theta})^{2n} \big)^{\rk{\xi}}}.
	\end{equation}
	\par 
	d) It is possible to restate Theorem \ref{thm_comp_appr} in the way, which doesn't use the language of compatible flattenings. It says that the quantity
	\begin{multline}\label{eqn_of_quil_norms_no_flat}
		2 \rk{\xi}^{-1} \ln \Big( 
			\norm{\cdot}_{Q} \big(g^{TM}, h^{\xi} \otimes \, \norm{\cdot}_{M}^{2n}\big) 
			\big / 
			\norm{\cdot}_{Q} \big(g^{TM}_{\rm{f}}, h^{\xi} \otimes (\, \norm{\cdot}_{M}^{\rm{f}})^{2n})
			\Big)
			\\	
			-
			\rk{\xi}^{-1} \int_M c_1(\xi, h^{\xi}) \Big(2n \ln (\, \norm{\cdot}_{M}^{\rm{f}}/ \norm{\cdot}_{M}) +  \ln (g^{TM}_{\rm{f}} / g^{TM}) \Big)
	\end{multline}
	depends only on the integer $n \in \integ$, $n \leq 0$, the functions $( g^{TM}_{\rm{f}} / g^{TM} )|_{V_i^{M}(1)} \circ (z_i^{M})^{-1} : \dd^* \to \real $ and $(\, \norm{\cdot}_{M}^{\rm{f}}/ \norm{\cdot}_{M} )|_{V_i^{M}(1)}  \circ (z_i^{M})^{-1} : \dd^* \to \real$, for $i = 1, \ldots, m$. This reformulation is particularly useful when one studies the variations of Quillen norm in a family setting.
	\par
	e) For $n = 0$ and $(\xi, h^{\xi})$ trivial, Theorem \ref{thm_comp_appr} was proved by Jorgenson-Lundelius in \cite[Theorem 7.3]{JorLundMain}, where authors use substantially that the geometry near the cusps of $(M, g^{TM})$ and $(N, g^{TN})$ coincides. This doesn't hold in our case due to the presence of $(\xi, h^{\xi})$, and the techniques we use are different even in the case when $(\xi, h^{\xi})$ is trivial. 
	\par
	The main feature of our techniques is that they are implicit, and unlike \cite{JorLundMain}, we avoid studying the precise contribution of the continuous spectrum to the heat kernel - a problem, which seems to be open for $n < 0$ (see Müller \cite{MullerCusp} for $n = 0$).
	We treat large-time asymptotics of the heat kernel by functional analysis and spectral gap theorem, and small-time asymptotics by analytic localization techniques of Bismut-Lebeau \cite[\S 11]{BisLeb91} and by the parametrix construction. Along the way we obtain some estimates (e.g. Moser-type estimates from Theorem \ref{thm_hk_est}) on the heat kernel associated with $(M, g^{TM})$, $(\xi, h^{\xi})$ and $n \in \integ$, which are of independent interest.
	Particularly, those estimates should be helpful if one wishes to extend the sup-estimates of the Bergman kernel on the Riemann surface with hyperbolic cusps due to Auvray-Ma-Marinescu \cite[Corollary 1.3]{Auvr}, \cite{AuvrMaArx} (cf. also \cite{FJK}) to the case when $(\xi, h^{\xi})$ is not necessarily trivial around the cusps.
	\end{rem}
	\par Our next result explains how the Quillen norm changes under the conformal change of the metric with cusps.
	Let's recall that by \cite[Theorem 1.27]{BGS1}, the Bott-Chern forms of a vector bundle $\xi$ with Hermitian metrics $h^{\xi}_{1}$,  $h^{\xi}_{2}$ are natural differential forms (strictly speaking, those are classes of differential forms, see Remark \ref{rem_anomaly_cusp}b)) defined so that they satisfy
		\begin{equation}\label{eq_der_tilde_tdch}
		\begin{aligned}
			& \frac{\partial \dbar}{2 \pi \imun} \widetilde{\td} (\xi, h^{\xi}_{1}, h^{\xi}_{2}) 
			&& = 
			\td (\xi, h^{\xi}_{1}) - 
			\td (\xi, h^{\xi}_{2}), \\
			& \frac{\partial \dbar}{2 \pi \imun} \widetilde{\ch} (\xi, h^{\xi}_{1}, h^{\xi}_{2}) 
			&& = 
			\ch (\xi, h^{\xi}_{1})  -
			\ch (\xi, h^{\xi}_{2}),
		\end{aligned}
		\end{equation}
		where $\td$, $\ch$ are Todd and Chern forms.
		By \cite[Theorem 1.27]{BGS1}, we have the following identities
		\begin{equation}
			\widetilde{\ch}(\xi, h^{\xi}_{1},  h^{\xi}_{2})^{[0]} = 2 \widetilde{\td}(\xi, h^{\xi}_{1},  h^{\xi}_{2})^{[0]} = \ln \big( \det( h_1^{\xi} /h_2^{\xi}) \big). \label{ch_bc_0}
		\end{equation}
		If, moreover, $\xi := L$ is a line bundle, we have
		\begin{equation}
			\widetilde{\ch}(L, h^L_{1}, h^L_{2})^{[2]} = 6 \widetilde{\td}(L, h^L_{1},  h^L_{2})^{[2]} 
			 =  \ln ( h^L_{1}/h^L_{2} ) \Big(c_1(L, h^L_{1})  + c_1(L, h^L_{2}) \Big) / 2, \label{ch_bc_2}
		\end{equation}
		where $c_1$ is the first Chern form.
	\par 
	\begin{defn}\label{defn_wolpert_norm}
		For a surface with cusps $(\overline{M}, D_M, g^{TM})$, the \textit{Wolpert norms} $\norm{\cdot}^W_{i}$ on the complex lines $\omega_{\overline{M}}|_{P_i^{M}}$, $i=1, \ldots, m$, is defined by $\| dz_i^{M} \|^W_{i} = 1$. It induces the Wolpert norm $\norm{\cdot}^W$ on the complex line $\otimes_i \omega_{\overline{M}}|_{P_i^{M}}$. 
	\end{defn}
	\begin{rem}
		The norms $\norm{\cdot}^W_{i}$ are well-defined since the Poincaré-compatible coordinates are well-defined up to a multiplication by a unitary constant.
		The norm $\norm{\cdot}^W$ was defined by Wolpert for hyperbolic surfaces in \cite[Definition 1]{Wol07}.
	\end{rem}
	\begin{sloppypar}
	\begingroup
	\setcounter{tmp}{\value{thm}}
	\setcounter{thm}{1} 
	\renewcommand\thethm{\Alph{thm}}
	\begin{thm}[Anomaly formula for metrics with cusps]\label{thm_anomaly_cusp}
		Let $\phi : M \to \real$ be a smooth function such that for the metric
		\begin{equation}\label{anomaly_rel_metrics}
			g^{TM}_{0} = e^{2 \phi} g^{TM},
		\end{equation}
		the triple $(\overline{M}, D_M, g^{TM}_{0})$ is a surface with cusps.
		We denote by $\norm{\cdot}_M, \norm{\cdot}_{M}^{0}$ the norms induced by $g^{TM}, g^{TM}_{0}$ on $\omega_M(D)$, and by $\norm{\cdot}^W$, $\norm{\cdot}^W_{0}$ the associated Wolpert norms. 
		Let $ h^{\xi}_0$ be a Hermitian metric on $\xi$ over $\overline{M}$.
		Then the right-hand side of the following equation is finite, and
		\begin{equation}\label{eq_anomaly_cusp}
			\begin{aligned}
				2 \ln & \Big(  
				\norm{\cdot}_{Q}  \big(g^{TM}_{0},  h^{\xi}_{0} \otimes (\, \norm{\cdot}_{M}^{0})^{2n} \big) 
				\big/				 
				 \norm{\cdot}_{Q} \big(g^{TM}, h^{\xi} \otimes \, \norm{\cdot}_{M}^{2n} \big) 
				 \Big) 
				\\
				&  = 
		 			  \int_{M} 
		 			\Big[ 
	 					\widetilde{\td} \big(\omega_M(D)^{-1}, \, \norm{\cdot}^{-2}_{M}, (\, \norm{\cdot}^{0}_{M})^{-2} \big) \ch \big(\xi, h^{\xi} \big)  \ch \big(\omega_M(D)^n, \norm{\cdot}_M^{2n} \big)  \\
						&  \phantom{= \int_{M} 
		 			\Big[ } +	
		 			\td \big(\omega_M(D)^{-1}, (\, \norm{\cdot}^{0}_{M})^{-2} \big) \widetilde{\ch} \big(\xi, h^{\xi}, h^{\xi}_{0} \big)  \ch \big(\omega_M(D)^n, \norm{\cdot}_M^{2n} \big)  \\
						&  \phantom{= \int_{M} 
		 			\Big[ } +			 		 					
			 			\td \big(\omega_M(D)^{-1}, (\, \norm{\cdot}^{0}_{M})^{-2} \big) \ch \big(\xi, h^{\xi}_{0} \big) \widetilde{\ch} \big(\omega_M(D)^n, \norm{\cdot}_{M}^{2n}, (\, \norm{\cdot}_{M}^{0})^{2n} \big) 									 				
			 		\Big]^{[2]} \\
			 		& \phantom{ = }  - \frac{\rk{\xi}}{6} \ln \Big( \norm{\cdot}^W /  \norm{\cdot}^W_{0} \Big)	+ \frac{1}{2} \sum \ln \Big(\det (h^{\xi} / h^{\xi}_{0})|_{P_i^{M}} \Big).
			 \end{aligned}
		\end{equation}
	\end{thm}
	\endgroup
	\setcounter{thm}{\thetmp}
	\end{sloppypar}
	\begin{rem}\label{rem_anomaly_cusp}
	\par a) The anomaly formula was firstly proved by Polyakov in \cite{PolyakBos} for $m = 0$, $n = 0$ and $(\xi, h^{\xi})$ trivial, who used it to compute some integrals over random surfaces which arise in mathematical physics.
	It was generalized by Bismut-Gillet-Soulé \cite[Theorem 1.23]{BGS3} for $m = 0$, but in any arbitrary dimension.
	For $m = 0$, in \cite{Fay1992}, Fay gave an alternative proof of (\ref{eq_anomaly_cusp}), which doesn't use the formalism of heat kernels. Our proof relies on the anomaly formula for $m = 0$.
	\par b) Strictly speaking, the integral in (\ref{eq_anomaly_cusp}) is not well-defined, since $\widetilde{\ch}$, $\widetilde{\td}$ as \textit{classes} are only well-defined up to an element of the form $\partial \alpha + \dbar \beta$. Since a priori nothing is known about the growth of $\alpha$, $\beta$ near $D_M$, the integrals of $\partial \alpha$ and $\dbar \beta$ over $M$ might not converge (leave alone being equal to $0$ by “Stokes" theorem). For the purposes of this article, however, it is enough to think of $\widetilde{\ch}$, $\widetilde{\td}$ as \textit{forms}, defined by (\ref{ch_bc_0}) and (\ref{ch_bc_2}). An alternative way to interpret those classes is through the Bott-Chern theory for pre-log-log Hermitian vector bundles, introduced by Burgos Gil-Kramer-K\"uhn in \cite{BurKrKun} (cf. \cite{FinII2}).
		\begin{sloppypar}
		 c) Experts will notice the difference between the terms under the integral in the right-hand side of (\ref{eq_anomaly_cusp}) and the terms, which appear in the right-hand side of the anomaly formula of Bismut-Gillet-Soulé \cite[Theorem 1.23]{BGS3} (see (\ref{eq_anomaly})), where in the arguments of Todd class and secondary Todd class we have $\omega_M$ in place of $\omega_M(D)$. 
		However, this difference is not a real issue, since for the current of integration $\delta_{D_M}$ along $D_M$, we have the following identities over $\overline{M}$:
		\begin{equation}
		\begin{aligned}\label{eq_ch_similarity}
			& \widetilde{\td} \big(\omega_M^{-1}, (\, \norm{\cdot}^{\omega}_{M})^{-2}, (\, \norm{\cdot}^{\omega, 0}_{M})^{-2} \big)
			= \widetilde{\td} \big(\omega_M(D)^{-1}, \, \norm{\cdot}^{-2}_{M}, (\, \norm{\cdot}^{0}_{M})^{-2} \big),
			\\
			& \big[ \td \big(\omega_M^{-1}, (\, \norm{\cdot}^{\omega, 0}_{M})^{-2} \big)  \big]^{[2]}
			= \big[ \td \big(\omega_M(D)^{-1}, (\, \norm{\cdot}^{0}_{M})^{-2} \big) \big]^{[2]} + \frac{1}{2} \delta_{D_M},
			\\
			& \big[ \widetilde{\ch} \big(\omega_M(D)^n, \norm{\cdot}_{M}^{2n}, (\, \norm{\cdot}_{M}^{0})^{2n} \big) \big]^{[0]}|_{D_M} = 0,
		\end{aligned}
		\end{equation}
		where $[0]$, $[2]$ stand for the part of degree $0$ and $2$, and in the second identity we used Poincaré–Lelong equation.
		Nevertheless, we prefer to state Theorem \ref{thm_anomaly_cusp} in the given form, since in the sequel we will use that the Hermitian line bundles $(\omega_M(D), \norm{\cdot}_M)$, $(\omega_M(D), \norm{\cdot}_M^{0})$ are pre-log-log with singularities along $D_M$ in the terminology of Burgos Gil-Kramer-K\"uhn \cite{BurKrKun}, and the Hermitian line bundles $(\omega_M, \norm{\cdot}_M^{\omega})$, $(\omega_M, \norm{\cdot}_M^{\omega, 0})$ do not satisfy those properties.
		\end{sloppypar}
		\par d) In the case when $\phi$ has compact support in $M$, Theorem \ref{thm_anomaly_cusp} follows from the anomaly formula of Bismut-Gillet-Soulé (see Theorem \ref{thm_anomaly_BGS}), Theorem \ref{thm_comp_appr} and (\ref{eq_ch_similarity}). \par 
		The difference between Theorem \ref{thm_anomaly_cusp} and Theorem \ref{thm_anomaly_BGS} is in the last two terms of (\ref{eq_anomaly_cusp}):
		\begin{equation}\label{eq_anomal_terms_suppl}
			- \frac{\rk{\xi}}{6} \ln \Big( \norm{\cdot}^W /  \norm{\cdot}^W_{0} \Big) + \frac{1}{2} \sum \ln \Big(\det (h^{\xi} / h^{\xi}_{0})|_{P_i^{M}} \Big).
		\end{equation}
		In our applications we use extensively Theorem \ref{thm_anomaly_cusp} for $\phi$ of non-compact support.
		Thus, the appearance of terms (\ref{eq_anomal_terms_suppl}) is of fundamental importance in what follows.
		\par
		e) Similar theorem appeared in the paper of Lundelius \cite[Theorem 1.1]{Lun93}. However, we disagree with his result, as it differs from ours in the last two terms of (\ref{eq_anomaly_cusp}). From \cite[p. 226, line 4]{Lun93}, his proof should only work for $\phi$ of compact support in $M$, as it is cited in \cite[Proposition 7.2]{JorLundMain}.
	\end{rem}
	To motivate this paper, we discuss several applications of Theorems \ref{thm_comp_appr}, \ref{thm_anomaly_cusp}, which will be proved in the sequel \cite{FinII2}, \cite{FinII3}. All those results are done in a family setting, i.e. we fix a holomorphic, proper map $\pi: X \to S$ of complex analytic manifolds such that for every $t \in S$, the space $X_t := \pi^{-1}(t)$ is a complex curve, which singularities are at most ordinary double points. We also fix disjoint sections $\sigma_1, \ldots, \sigma_m : S \to X$, which avoid double points of the fibres, and we denote by $D_{X/S}$ the divisor, given by $\Im(\sigma_1) + \ldots + \Im(\sigma_m)$.
	\begin{sloppypar}
		\textit{1. Regularity and asymptotics of the Quillen norm in a degenerating family of surfaces, \cite[Theorem C]{FinII2}.}
	We consider the determinant line bundle $\lambda(j^*( \xi \otimes \omega_{X/S}(D)^n )) := (\det R^{\bullet} \pi_* (\xi \otimes \omega_{X/S}(D)^n))^{-1}$, $n \leq 0$, where $\xi$ is a holomorphic vector bundle over $X$ and $\omega_{X/S}(D) := \omega_{X/S} \otimes \mathscr{O}_X(D_{X/S})$ is the twisted relative canonical line bundle. We endow the vector bundles $\xi$, $\omega_{X/S}(D)$ with Hermitian metrics $h^{\xi}, \norm{\cdot}_{X/S}$ satisfying some mild hypothesises. 
	Let $|\Delta|$ be the locus of singular curves of $\pi$.	
	We define the Quillen norm $\norm{\cdot}_{Q}$ on $\lambda(j^*(\xi \otimes \omega_{X/S}(D)^n))$ over $S \setminus |\Delta|$, as a family version of (\ref{defn_quil}). Then we study its regularity and singularities near $| \Delta |$. We also explicit some conditions under which the renormalized Quillen norm becomes continuous at the singular fibers. 
	\par The hypotheses, which we put on $\norm{\cdot}_{X/S}$ are mild enough to include the case of hyperbolic surfaces. In this particular case, the asymptotics of the associated analytic torsion was studied before by Wolpert \cite{Wol87}, Lundelius \cite{Lun93}, Jorgenson-Lundelius \cite{JorLund97}, and many others.
	\end{sloppypar}
	 \textit{2. Curvature theorem for surfaces with cusps, \cite[Theorem D]{FinII2}.}
	We will show that the metric $\norm{\cdot}_Q$ from previous paragraph is good enough, so that one can define its Chern form as a current. Then we give an explicit formula for this current. 
	\par In particular, if we consider the family of hyperbolic surfaces, this generalizes the curvature theorem of Takhtajan-Zograf \cite[Theorem 1]{TakZog}. If we consider the case when there is no cusps, we get a generalization of Bismut-Bost \cite[Théorème 2.1]{BisBost} to the case of degenerating metrics. 
	\\ \hspace*{0.5cm} \textit{3. Immersion and compatibility theorems, \cite{FinII3}.}
	We will relate the restriction of the renormalized Quillen norm $\norm{\cdot}_Q$ at the locus of singular fibers $|\Delta|$ with the Quillen norm of the normalization of singular fibers. By combination of this result with the analogical statement for Takhtajan-Zograf analytic torsion (see (\ref{eqn_sel_norm})), we deduce the compatibility between our definition and the one of Takhtajan-Zograf. This generalizes the result of Jorgenson-Lundelius \cite[Corollary 4.3]{JorLund97}, where authors did it for $(\xi, h^{\xi})$ trivial, and $n=0$.
	\par 
	Finally, due to recent interest in orbifold Riemann surfaces (see \cite{JorgGarb}, \cite{FreixPip}, \cite{TakhZogrOrbi}), let's discuss how the theory developed here can be adapted to the orbifold setting.
	By combining the definition of the analytic torsion here and of the orbifold analytic torsion due to Ma \cite{MaOrbif2005}, for an orbisurface $(M, g^{TM})$ with cusps $D_M \subset \overline{M}$ and orbifold singularities $D_M' \subset \overline{M}$, we may define the analytic torsion $T(g^{TM}, h^{\xi} \otimes \norm{\cdot}_M^{2n})$, where $n \leq 0$ and $\norm{\cdot}_M$ is the induced norm on the orbifold twisted line bundle $\omega_{\overline{M}} \otimes \mathscr{O}_{\overline{M}}(D_M) \otimes ( \otimes_{P_i' \in D_M'} \mathscr{O}_{\overline{M}}((1-1/m_i) P_i'))$. Here $m_i$ is the order of elliptic fixed point $P_i'$. This definition should generalize both definitions of the analytic torsion due to Takhtajan-Zograf \cite{TakhZogrOrbi}, which is done for stable hyperbolic orbisurfaces and $(\xi, h^{\xi})$ trivial, and of Freixas-von Pippich \cite{FreixPip}, which is done for stable hyperbolic orbisurfaces, $(\xi, h^{\xi})$ trivial and $n = 0$. 	
	Since our methods in the proof of Theorem \ref{thm_comp_appr} are purely local, the analogue of Theorem \ref{thm_comp_appr} would still hold. Since we got Theorem \ref{thm_anomaly_cusp} by combining Theorem \ref{thm_comp_appr} and the anomaly formula of Bismut-Gillet-Soulé \cite[Theorem 1.23]{BGS1}, by replacing the last reference by its orbifold analogue of Ma \cite[Theorem 0.1]{MaOrbif2005}, it is possible to get an analogue of Theorem \ref{thm_anomaly_cusp}. Then by combining the Deligne-Mumford isomorphism in the orbifold setting due to Freixas-von Pippich \cite{FreixPip} and the anomaly formula, in the same way as we proceed in \cite{FinII2}, it should be possible to get the orbifold analogue of Deligne-Mumford isometry for any orbisurface with metric with cusps and a Hermitian vector bundle over it. We hope to return to this question very soon.
	\par Now, let's describe the structure of this paper. 
	In Section 2, we develop spectral theory for surfaces with cusps. We introduce the notion of the analytic torsion and Quillen norm, which are used throughout the article.
	In Section 3, we prove Theorem \ref{thm_comp_appr}. For this, we study the families of metrics which “converge" to the metric with cusps in such a way that certain spectral properties are preserved.
	In Section 4, we prove Theorem \ref{thm_anomaly_cusp}. The main idea is to use Theorem \ref{thm_comp_appr} and to obtain Theorem \ref{thm_anomaly_cusp} as a limit of the anomaly formula of Bismut-Gillet-Soulé \cite[Theorem 1.23]{BGS1}.
	Finally, in Appendix we prove elliptic estimates for surfaces with cusps. Also we prove the existence of the families of metrics described in Section \ref{sect_compact_pert}.
	\par \textbf{Notation.} For $\epsilon > 0$ and $(\overline{M}, D_M), (\overline{N}, D_N), \xi$ as in (\ref{data_rel_tors}), we denote
	\begin{equation}
	\begin{aligned}
		& D(\epsilon) = \{ u \in \comp : |u| < \epsilon \},  \quad D^*(\epsilon) = \{ u \in \comp : 0 < |u| < \epsilon \}, \\
		& \dd := D(1), \qquad \qquad \qquad \quad \dd^* = D^*(1), \\
		& \omega_M(D) :=  \omega_{\overline{M}} \otimes  \mathscr{O}_{\overline{M}}(D_M),  \\
		& E_M^{\xi, n} := \xi \otimes \omega_M(D)^n,  \qquad   \, \, E_N^{n} := \omega_N(D)^n. \\
	\end{aligned}
	\end{equation}
	By $g^{T \dd^*}$ we denote the metric on $\dd^*$, induced by (\ref{reqr_poincare}), and by $d v_{\dd^*}$ the associated Riemannian volume form. By $\spec(A)$ we denote the spectrum of a self-adjoint operator $A$, acting on some Hilbert space.
	We denote by $B^M(x, r)$ the geodesic ball of radius $r > 0$ around $x \in M$ in a Riemannian surface $M$ with Riemannian metric $g^{TM}$.
	\par {\bf{Acknowledgements.}} This work is part of our PhD thesis, which was done at Université Paris Diderot. We would like to express our deep gratitude to our PhD advisor Xiaonan Ma for his teaching, overall guidance, constant support and invaluable comments on the preliminary version of this article.

	\section{Spectral theory of surfaces with cusps}\label{sect_spec_th}
	In this section we study spectral properties of surfaces with cusps and define the analytic torsion. 
	\par More precisely, in Section 2.1 we set up the notation and state the spectral gap theorem.
	In Section 2.2 we state several estimations of the heat kernel associated with a hyperbolic surface, we define the regularized heat trace and the analytic torsion.
	Section 2.3 is the most technical one. Here we prove the estimations on the heat kernel of the hyperbolic punctured disc.
	Finally, in Section 2.4 we prove the statements from Sections 2.1, 2.2.
	
	\subsection{The setting of the problem and the spectral gap theorem}\label{sect_spec_gap}
	\par Let $(\overline{M}, D_M, g^{TM})$ be a Riemann surface with cusps and let $(\xi, h^{\xi})$ be a Hermitian vector bundle over $\overline{M}$. We denote by $\norm{\cdot}_{M}$ the Hermitian norm induced by $g^{TM}$ on $\omega_M(D)$ (see (\ref{eq_om_md_def})) over $M$. 
	\par Let $\alpha, \alpha' \in \ccal^{\infty}_{c}(M, E_M^{\xi, n})$. The $L^2$-scalar product is defined by
	\begin{equation}\label{defn_L_2}
		\scal{\alpha}{\alpha'}_{L^2} := \tinyint_{M} \scal{\alpha(x)}{\alpha'(x)}_h d \vol_M(x), 
	\end{equation}
	where $d \vol_M$ is the Riemannian volume form on $(M, g^{TM})$, and $\langle \cdot, \cdot \rangle_h$ is the pointwise Hermitian product induced by $h^{\xi}$, $\norm{\cdot}_M$. By (\ref{reqr_poincare}), the right-hand side of (\ref{defn_L_2}) is finite for $n \leq 0$.
	\begin{sloppypar}
	We define the Hilbert space $(L^2(E_M^{\xi, n}), \scal{\cdot}{\cdot}_{L^2})$, as the $L^2$-completion of the space $\ccal^{\infty}_{c}(M, E_M^{\xi, n})$ with respect to $\langle \cdot, \cdot \rangle_{L^2}$.
	Sometimes when we want to insist on the choice of $g^{TM}$, $h^{\xi}$ and $\norm{\cdot}_M$, we denote this space by  $L^2(g^{TM}, h^{\xi} \otimes (\, \norm{\cdot}_M)^{2n})$.
	\end{sloppypar}
	\par
	We denote by $\laplcomp^{E_M^{\xi, n}}$ the Kodaira Laplacian on $\ccal^{\infty}_{c}(M, E_M^{\xi, n})$, given by
	\begin{equation}\label{eq_kodaira_laplacian}
		\laplcomp^{E_M^{\xi, n}} := (\dbar^{E_M^{\xi, n}})^*  \dbar^{E_M^{\xi, n}},
	\end{equation}
	where $(\dbar^{E_M^{\xi, n}})^*$ is the formal adjoint of $\dbar^{E_M^{\xi, n}}$ with respect to $\langle \cdot, \cdot \rangle_{L^2}$.
	 Since $(M, g^{TM})$ is complete, the operator $\laplcomp^{E_M^{\xi, n}}$ is essentially self-adjoint on $L^2(E_M^{\xi, n})$ (cf. \cite[Corollary 3.3.4]{MaHol}). We denote its closure by the same symbol.
	 \par In this article we are mostly interested in the heat operator $\exp(-t \laplcomp^{E_M^{\xi, n}}), t > 0$. 
	We denote
	\begin{equation}\label{p_perp_defn}
		\exp^{\perp}(-t \laplcomp^{E_M^{\xi, n}}) := \exp(-t \laplcomp^{E_M^{\xi, n}}) - P_{M},
	\end{equation}
	where $P_{M}$ is the orthogonal projection onto $\ker ( \laplcomp^{E_M^{\xi, n}} )$. We denote by
	\begin{equation}\label{defn_hk_notat}
		\exp(-t \laplcomp^{E_M^{\xi, n}})(x, y),
		\exp^{\perp}(-t \laplcomp^{E_M^{\xi, n}})(x, y) 
		\in (E_M^{\xi, n})_{x} \boxtimes (E_M^{\xi, n})^{*}_{y},
		 \quad \text{for} \quad x,y \in M,
	\end{equation}
	the smooth kernels of  $\exp(-t \laplcomp^{E_M^{\xi, n}}), \exp^{\perp}(-t \laplcomp^{E_M^{\xi, n}})$ with respect to the volume form $d \vol_M$.
	In particular, we see that
	\begin{equation}
		\exp(-t \laplcomp^{E_M^{\xi, n}})(x, x),
		\exp^{\perp}(-t \laplcomp^{E_M^{\xi, n}})(x, x) 
		\in \enmr{\xi}_x,
		 \quad \text{for} \quad x \in M.
	\end{equation}
	In Section \ref{sect_spec_th}, we fix $g^{TM}$, $h^{\xi}$, $\norm{\cdot}_M$ and remove them from some notation: by $| \cdot |_{h \times h}$ we mean the pointwise norm on $(\omega_{\overline{M}}^{k} \otimes E_M^{\xi, n})^* \boxtimes ( \omega_{\overline{M}}^{l} \otimes E_M^{\xi, n} )$, $k,l \in \integ$ induced by $h^{\xi}$, $\norm{\cdot}_M$, $g^{TM}$; by $| \cdot |$ we mean either the modulus of a complex number, or the pointwise norm on the vector bundle ${\rm{End}}(\xi)$ induced by $h^{\xi}$. We defer the proof of the next theorem until Section \ref{sect_aux}.
	\begin{thm}\label{spec_gap_thm}
		For $n \leq 0$, the operator $\laplcomp^{E_M^{\xi, n}}$ has a spectral gap near $0$. More precisely, we have
		\begin{equation}\label{eqn_ker_lapl}
			H^{0}(\overline{M}, E_M^{\xi, n}) = \ker (\laplcomp^{E_M^{\xi, n}}),
		\end{equation}
		and there is  $\mu > 0$ such that
		\begin{equation}\label{eqn_spec_gap}
			\spec \big( \laplcomp^{E_M^{\xi, n}} \big) \cap \, ]0, \mu] = \emptyset.
		\end{equation}
	\end{thm}
	\begin{rem}\label{rem_spec_gap_mul}
		As it would follow from our proof, there are $c_1, c_1 > 0$ such that the set
		\begin{equation}\label{eqn_spec_gap_strong}
			\spec \big( \laplcomp^{E_M^{\xi, n}} \big) \cap \, [0, -c_1n + c_2] \qquad \text{is discrete}
		\end{equation}
		for any $(M, g^{TM})$, $(\xi, h^{\xi})$, $\norm{\cdot}_M$ and $n \leq 0$.  We leave the verification of the details to the interested reader. 
		For $n = 0$, $(\xi, h^{\xi})$ trivial, and $c_2 = 1/4$, this was proved by Müller in \cite[\S 6]{MullerCusp}.
	\end{rem}
	Our proof relies on the result of Müller \cite[\S 6, Proposition 6.9]{MullerCusp}, who proves Theorem \ref{spec_gap_thm} for $n = 0$ and $(\xi, h^{\xi})$ trivial. 
	To motivate, let's prove one special case of Theorem \ref{spec_gap_thm} : \textit{let $(\xi, h^{\xi})$ be trivial, $n \leq -1$, the genus $g(\overline{M})$ satisfies (\ref{cond_stable}) and $g^{TM} = g^{TM}_{\rm{hyp}}$ be the canonical hyperbolic metric}. Then from Nakano's inequality (cf. \cite[Theorem 1.4.14]{MaHol}), we get  $\ker (\laplcomp^{E_M^{\xi, n}}) = \{ 0 \}$ and (\ref{eqn_spec_gap}). By Kodaira vanishing theorem (cf. \cite[Theorem 1.5.4]{MaHol}), we have $H^0(\overline{M}, E_M^{\xi, n}) = \{ 0 \}$. Thus, Theorem \ref{spec_gap_thm} holds for $(\overline{M}, D_M, g^{TM}_{\rm{hyp}})$, $n \leq -1$ and $(\xi, h^{\xi})$ trivial.
	\par Finally, let's discuss the construction of the $L^2$-norm $\norm{\cdot}_{L^2}(g^{TM}, h^{\xi} \otimes \norm{\cdot}_M^{2n})$ on the line bundle (\ref{defn_det_line}). By the isomorphism (\ref{eqn_ker_lapl}), we may endow $H^{0}(\overline{M}, E_M^{\xi, n})$ with the $L^2$-scalar product induced by (\ref{defn_L_2}). 
	Similarly to the analysis in the proof of (\ref{eqn_ker_lapl}), we have a natural isomorphism
	\begin{equation}\label{eqn_ker_lapl_h1}
		\ker (\laplcomp^{E_M^{\xi, n}}_{1}) = 
		\begin{cases} 
      			\hfill H^{1}(\overline{M}, E_M^{\xi, n}), & \text{ for } n=0, \\
      			\hfill H^{1}(\overline{M}, E_M^{\xi, n} \otimes \mathscr{O}_{\overline{M}}(D_M)),  & \text{ for }  n \leq -1,
 			\end{cases}
	\end{equation}
	where $\laplcomp^{E_M^{\xi, n}}_{1} = \dbar^{E_M^{\xi, n}} (\dbar^{E_M^{\xi, n}})^*$ is the Kodaira Laplacian associated with $1$-forms with values in $E_M^{\xi, n}$. We induce the $L^2$-scalar product on $H^{1}(\overline{M}, E_M^{\xi, n})$ by the natural inclusion
	\begin{equation}
		H^{1}(\overline{M}, E_M^{\xi, n}) 
		\xhookrightarrow{} 
		H^{1}(\overline{M}, E_M^{\xi, n} \otimes \mathscr{O}_{\overline{M}}(D_M)), \qquad \alpha \mapsto \alpha \otimes s_{D_M},
	\end{equation}
	where $s_{D_M}$ is the canonical holomorphic section of $\mathscr{O}_{\overline{M}}(D_M)$.
	Those scalar products induce the natural $L^2$-norm $\norm{\cdot}_{L^2}(g^{TM}, h^{\xi} \otimes \norm{\cdot}_M^{2n})$ on the line bundle (\ref{defn_det_line}).

\subsection{Relative spectral theory for surfaces with cusps}
	\begin{sloppypar}
		The main goal of this section is to define the \textit{analytic torsion} for any $(\xi, h^{\xi})$, $n \leq 0$, $m \in \nat$.
		This extends the relative definition due to Jorgenson-Lundelius \cite[Definition 1.9]{JorLundMain}, which they gave in the case $n = 0$ and $(\xi, h^{\xi})$ trivial. 
		The challenge here is that unlike in \cite{JorLundMain}, the precise contribution of the continuous spectrum to the heat kernel is unknown, moreover the local geometry near the cusp depends on $(\xi, h^{\xi})$. 
		We circumvent this difficulty by the analytic localization techniques of Bismut-Lebeau \cite[\S 11]{BisLeb91} and by the parametrix construction for the heat kernel (cf.  \cite[\S 2.4, 2.5]{BGV}).
		The parametrix construction will be particularly useful when we would estimate the effect of non-triviality of $(\xi, h^{\xi})$ (see Theorem \ref{thm_rel_hk_infty} and (\ref{eq_small_time_coeff_rel})).
	\end{sloppypar}
	\par We fix $n \in \integ$. Let the function $\rho_M : M \to [1, + \infty[$ be given by
		\begin{equation}\label{defn_rho}
			\rho_M(x) = 
			\begin{cases} 
      			\hfill 1 & \text{ for } x \in  M \setminus (\cup_i V_i^{M}(1/2)), \\
      			\hfill \sqrt{| \ln |z_i(x)||}  & \text{ for }  x \in  V_i^{M}(1/2), i = 1, \ldots, m. \\
 			\end{cases}
		\end{equation}	
	\begin{rem}\label{rem_inj_rad}
		The function $(\rho_M(x))^{-2}$ is proportional to the injectivity radius at point $x$ of $(M, g^{TM})$.
	\end{rem}
	We denote by $\dist(\cdot, \cdot)$ the distance function on $(M, g^{TM})$.
	Now we can state the main theorems of this section. Their proofs are delayed until Section \ref{sect_aux}.
	\begin{thm}\label{thm_hk_est}
		For any $l,l' \in \nat$, there are $c, c', C > 0$ such that for any $t > 0$, $x, x' \in M$, we have
		\begin{multline}\label{thm_est_exp2} 
			\textstyle \big| (\nabla_x)^{l} (\nabla_{x'})^{l'} \exp(-t \laplcomp^{E_M^{\xi, n}})(x, x') \big|_{h \times h} \textstyle \leq C \rho_M(x) \rho_M(x') t^{-1 - (l + l')/2} 
			\cdot \\
			\cdot  \exp(ct - c' \cdot \dist(x, x')^2/t) ,
		\end{multline}
		where $\nabla$ is induced by the Levi-Civita connection and the Chern connections of $(\xi, h^{\xi})$ and $(\omega_{M}(D), \norm{\cdot}_M)$.
		Also, if $n \leq 0$, then there are $c, C > 0$ such that for any $t > 0$, we have
		\begin{equation}\label{thm_est_exp_perp2}
			\textstyle \big| (\nabla_x)^{l} (\nabla_{x'})^{l'} \exp^{\perp}(-t \laplcomp^{E_M^{\xi, n}})(x, x') \big|_{h \times h}  \textstyle\leq C  \rho_M(x) \rho_M(x') t^{-4 -l - l'} \exp (-ct ). 
		\end{equation}
	\end{thm}
	\begin{rem}
		a) By Remark \ref{rem_inj_rad}, we see that for $n=0$, $(\xi, h^{\xi})$ trivial and $k,l=0$, (\ref{thm_est_exp2}) is exactly the Moser's estimate \cite[p. 115-117]{MoserHarn} (cf. \cite[Theorem VIII.8]{Chav}) for a hyperbolic surface. The proof of (\ref{thm_est_exp2}) is different from \cite[p. 115-117]{MoserHarn} and it uses an explicit construction of the parametrix of the heat kernel. 
		\par b) By using the same techniques as in the proof of (\ref{thm_est_exp_perp2}), we may deduce that for any $l, l' \in \nat$, there is $C>0$ such that
			\begin{equation}\label{thm_est_exp_est_old}
				\textstyle \big| (\nabla_x)^{l} (\nabla_{x'})^{l'} \exp(-t \laplcomp^{E_M^{\xi, n}})(x, x') \big|_{h \times h}  \textstyle \leq C  \rho_M(x) \rho_M(x') t^{-4 -l - l'}. 
			\end{equation}
			The estimate (\ref{thm_est_exp_est_old}) is unfortunately not enough for our needs, since we use (\ref{thm_est_exp2}) in the proof of (\ref{hk_infty}) inside Duhamel's formula, thus, the precise power of $t$ is important. However, by Remark \ref{hk_infty_sharp}, we note that if one considers only $(\xi, h^{\xi})$ which are trivial around the cusps, then the estimate (\ref{thm_est_exp_est_old}) is enough to prove (\ref{hk_infty_perp}), and all the analysis associated with the parametrix construction is not necessary. 
	\end{rem}
	Now, let $M, N$ and all related notions be as in (\ref{data_rel_tors}).
	\begin{thm}\label{thm_rel_hk_infty}
		For any $k \in \nat$, there are $\epsilon, c, c', C > 0$ such that for any $t > 0$, $u \in \comp, |u| \leq \epsilon$:
		\begin{align}
			 \textstyle
			\Big|
				\exp(-t \laplcomp^{E_M^{\xi, n}}) \big( (z_i^{M})^{-1}(u), (z_i^{M})^{-1}(u) \big)    - 
				{\rm{Id}_{\xi}} \cdot \exp(-t \laplcomp^{E_N^{n}}) \big( (z_i^{N})^{-1}(u), (z_i^{N})^{-1}(u) \big)
			\Big|  \nonumber \\  
				\leq
				C | \ln |u| | \exp(ct) \cdot \min \Big\{
				| \ln |u| |^{-k} + \exp(-c' (\ln |\ln |u||)^2/t );  \label{hk_infty}   \\
      			 |u|^{1/3} + \exp(-c'/t) \label{hk_infty2}
 			\Big\}.
		\end{align}
		Moreover, if $n \leq 0$, then there are $\varsigma  < 1$ and $c, C > 0$ such that
		\begin{multline}\label{hk_infty_perp}
			\textstyle
			\Big|
				\exp^{\perp}(-t \laplcomp^{E_M^{\xi, n}}) \big( (z_i^{M})^{-1}(u), (z_i^{M})^{-1}(u) \big)    	
				- 
				{\rm{Id}_{\xi}} \cdot \exp^{\perp}(-t \laplcomp^{E_N^{n}}) \big( (z_i^{N})^{-1}(u), (z_i^{N})^{-1}(u) \big)
			\Big|  \\    
				\textstyle \leq C |\ln |u||^{\varsigma} t^{-4}  \exp (-ct). 
		\end{multline}
	\end{thm}
	\begin{rem}
		As we explain in the course of the proof of Theorem \ref{thm_rel_hk_infty}, if $(\xi, h^{\xi})$ is trivial around the cusps, then. since the geometry around the cusps of $(M, g^{TM})$, $(N, g^{TN})$ coincides, the estimates (\ref{hk_infty}), (\ref{hk_infty2}) could be easily improved. In this case, we have
		\begin{align}
			 \textstyle
			\Big|
				\exp(-t \laplcomp^{E_M^{\xi, n}}) \big( (z_i^{M})^{-1}(u), (z_i^{M})^{-1}(u) \big)    - 
				{\rm{Id}_{\xi}} \cdot \exp(-t \laplcomp^{E_N^{n}}) \big( (z_i^{N})^{-1}(u), (z_i^{N})^{-1}(u) \big)
			\Big|  \nonumber \\  
				\leq
				C | \ln |u| |  \exp(-c' (\ln |\ln |u||)^2/t ). \label{hk_infty_sharp} 
		\end{align}
		To prove (\ref{hk_infty}), (\ref{hk_infty2}) in full generality, we use Duhamel's formula and estimates from (\ref{thm_est_exp2}).
	\end{rem}
	\begin{thm}\label{thm_small_time_exp}
		There are smooth bounded functions $a_{\xi, j}^{M, n} : M \to \enmr{\xi}$, $j \geq -1$ such that for any $x \in M$, $t_0 > 0$, $k \in \nat$, there is $C > 0$ such that for any $t \in ]0, t_0]$, we have
		\begin{equation}\label{eq_small_time_exp_111}
			\Big| \exp(-t \laplcomp^{E_M^{\xi, n}}) \big( x, x \big) - \sum_{j = - 1}^{k} a_{\xi, j}^{M, n}(x) t^j \Big| \leq C t^k.
		\end{equation}
		Moreover, if $x \in M \setminus (\cup_i V_i^{M}(e^{-t^{-1/3}}))$, then $C$ can be chosen independently of $t \in ]0, t_0]$ and $x$.
		\par 
		Also, there is $\epsilon > 0$, such that for any $l \in \nat$, $j \geq -1$, there is $C > 0$ such that for any $u \in \comp$, $0 < |u| \leq \epsilon$, $i = 1, \ldots, m$, we have
		\begin{equation}\label{eq_small_time_coeff_rel}
			\Big| (\nabla_u)^l \Big( a_{\xi, j}^{M, n} \big( (z_i^{M})^{-1}(u) \big) - {\rm{Id}_{\xi}} a_{j}^{N, n} \big( (z_i^{N})^{-1}(u) \big) \Big) \Big|_{h} \leq C |u|^{1/3},
		\end{equation}
		where $\nabla$ is induced by the Levi-Civita connection and Chern connections associated with $(\xi, h^{\xi})$ and $(\omega_{\dd}(0), \norm{\cdot}_{\dd})$.
	\end{thm}
	\begin{rem}\label{rem_integ_small_time}
		Since the functions $a_j^{E_M^{\xi, n}}$, $j \geq -1$ are bounded, the functions ${\rm{Tr}}^{\xi}\big[ a_j^{E_M^{\xi, n}} \big]$ are integrable over $M$, which is already suggested by an analogue of \cite[Theorem 4.1.7]{MaHol} and the fact that the scalar curvature is constant.
	\end{rem}
	From now on till the end of this section, we denote by 
	\begin{equation}
		P: = \cp \setminus \{0, 1, \infty\},
	\end{equation}
	 and by $g^{TP}$ the unique hyperbolic metric of constant scalar curvature $-1$ over $P$ with cusps at $D_P = \{0, 1, \infty \}$. We use the notations $\norm{\cdot}_P$, $V_i^{P}(\epsilon)$, $E_P^{n}$, $\ldots$, and denote by $z^P$ the Poincaré-compatible coordinate of $0 \in \cp$.
	 \begin{defn}\label{defn_rel_HT}
	 	We define the \textit{regularized heat trace} by
	 	\begin{align}
	 	{\rm{Tr}}^{\reg} \big[ \exp^{\perp} (& -t  \laplcomp^{E_M^{\xi, n}}) \big] := 
		\nonumber
	 	  \int_{M \setminus (\cup_i V_i^{M}(\eta))} \tr{\exp^{\perp}(-t  \laplcomp^{E_M^{\xi, n}})(x, x)} d \vol_M(x) 	 	
	 	\\ \label{eqn_rel_HT}
	 	- \frac{m \cdot \rk{\xi}}{3}
	 	&  
	 	\int_{P \setminus ( \cup_i V_i^{P}(\eta))} \tr{\exp^{\perp}(-t  \laplcomp^{E_P^{n}})(x, x)} d \vol_P(x) 
	 	\\ \nonumber
	 	+ 
	 	& 
	 	 \sum_i \int_{D^*(\eta)} \Big( {\rm{Tr}} \Big[ \exp^{\perp}(-t  \laplcomp^{E_M^{\xi, n}}) \big((z_{i}^{M})^{-1}(u), (z_{i}^{M})^{-1}(u)\big) \Big] 
	 	\\ \nonumber
	 	& 
	 	\qquad \qquad
	 	- \rk{\xi} {\rm{Tr}} \Big[ \exp^{\perp}(-t  \laplcomp^{E_P^{n}})\big((z^P)^{-1}(u), (z^P)^{-1}(u)\big) \Big] \Big) 
	 	d \vol_{\dd^*}(u),
	 	\end{align}
	 	where $\eta > 0$ is such that Theorem \ref{thm_rel_hk_infty} and (\ref{reqr_poincare}) hold.
	 \end{defn}
	 \begin{rem}\label{rem_ht}
	 	a) From the fact that there is a holomorphic automorphism of $\cp$ inducing the isometry on $(P, g^{TP})$, which permutes $D_P$ as we wish, the coordinate $z^P$ in the definition might be changed to a Poincaré-compatible coordinate associated with $1$ or $\infty$, and this would result in the same definition.
	 	\par 
	 	b) Essentially, in our definition of the regularized heat trace, we take out the diverging part of the usual heat trace. This idea is very similar to the famous $b$-trace, defined by Melrose in \cite[Lemma 4.62]{MelroseAPS}, which was used in the context of Riemann surfaces with cusps by Albin-Rochon \cite{AlbRoch}. It has also appeared in the paper \cite{JorLund97} of Jorgenson-Lundelius, where they defined the regularized heat trace of the hyperbolic surface with cusps for $n=0$ and $(\xi, h^{\xi})$ trivial.
	 \end{rem}
	 \begin{prop}
	 	The Definition \ref{defn_rel_HT} makes sense and it is independent of $\epsilon > 0$. We also have
	 	\begin{multline}\label{eqn_rel_HT_alt}
	 	{\rm{Tr}}^{\reg}  \big[ \exp^{\perp} ( -t  \laplcomp^{E_M^{\xi, n}}) \big] :=  	
		\lim_{r \to 0}	
		\bigg( 	
	 	\int_{M \setminus (\cup_i V_i^{M}(r))} \tr{\exp^{\perp}(-t  \laplcomp^{E_M^{\xi, n}})(x, x)} d \vol_M(x) 	 	
	 	\\ 
	 	- \frac{\rk{\xi}}{3}
	 	\int_{P \setminus ( \cup_i V_i^{P}(r))} \tr{\exp^{\perp}(-t  \laplcomp^{E_P^{n}})(x, x)} d \vol_P(x) \bigg).
	 	\end{multline}
	 \end{prop}
	 \begin{proof}
	 	The first two integrals in the right-hand side of (\ref{eqn_rel_HT}) are bounded by (\ref{thm_est_exp_perp2}). The last one is bounded by (\ref{hk_infty_perp}) and the fact that for any $\varsigma < 1$, we have
		\begin{equation}\label{eq_int_mu_fin}
			\int_{D(\epsilon)} \frac{\imun du d \overline{u}}{|u|^2 |\ln |u| |^{2 - \varsigma}} \leq + \infty.
		\end{equation}
		The independence on $\epsilon > 0$ is trivial.
		The formula (\ref{eqn_rel_HT_alt}) follows trivially from (\ref{hk_infty_perp}).
	 \end{proof}
	 A similar quantity ${\rm{Tr}}^{\reg}  \big[ \exp(-t  \laplcomp^{E_M^{\xi, n}}) \big]$ (see also \cite[Definition 1.1]{JorLundMain} for the relative version) is defined by the same formulas as in (\ref{eqn_rel_HT}), where we put $\exp$ in place of $\exp^{\perp}$. It is well-defined due to (\ref{hk_infty}), (\ref{eq_int_mu_fin}) and the fact that for any $c' > 0$ and $\epsilon > 0$ small enough, there is $C > 0$ such that for any $t > 0$, we have
		\begin{equation}\label{eqn_small_t_as25}
			\int_{D(\epsilon)} \exp \big(-c' (\ln |\ln |u||)^2/t \big) \frac{ \imun du d \overline{u} }{|u|^2 |\ln |u||}
			\leq C t^{1/2} \exp \big(-(c'/2) (\ln |\ln \epsilon| )^2/t \big).
		\end{equation}
	 By (\ref{eqn_ker_lapl}), the relation between Definition \ref{defn_rel_HT} and ${\rm{Tr}} \big[ \exp(-t  \laplcomp^{E_M^{\xi, n}}) \big]$ is given by
	\begin{multline}\label{eq_rel_nonperp}
		{\rm{Tr}}^{\reg}  \big[ \exp^{\perp}(-t  \laplcomp^{E_M^{\xi, n}}) \big] 
		\\
		= 
		{\rm{Tr}}^{\reg}  \big[ \exp(-t  \laplcomp^{E_M^{\xi, n}}) \big] 
		-  \dim H^0(\overline{M}, E_M^{\xi, n})
		+ \frac{\rk{\xi}}{3}  \dim H^0(\overline{P}, E_P^{n}).
	\end{multline}
	\begin{rem}
		In \cite[\S 3]{JorLundMain}, Jorgenson-Lundelius defined the relative heat trace 
		\begin{equation}
			{\rm{Tr}}^{{\rm{rel}}} \big[ \exp^{\perp}(-t  \laplcomp^{E_M^{\xi, n}}); \exp^{\perp}(-t  \laplcomp^{E_N^{n}}) \big] 
		\end{equation}				
		for $(\xi, h^{\xi})$ trivial and $n=0$. Directly from the definition, in this case we have
		\begin{align}
			& {\rm{Tr}}^{{\rm{rel}}} \big[ \exp^{\perp}(-t  \laplcomp^{E_M^{\xi, n}}); \exp^{\perp}(-t  \laplcomp^{E_N^{n}}) \big]  = 
			{\rm{Tr}}^{\reg}  \big[ \exp^{\perp}(-t  \laplcomp^{E_M^{\xi, n}}) \big]  - 
			\rk{\xi} {\rm{Tr}}^{\reg}  \big[ \exp^{\perp}(-t  \laplcomp^{E_N^{n}}) \big],
			\nonumber
			\\
			& {\rm{Tr}}^{\reg}  \big[ \exp^{\perp}(-t  \laplcomp^{E_M^{\xi, n}}) \big] = 
			\frac{1}{3}
			{\rm{Tr}}^{{\rm{rel}}} \big[ 3 \exp^{\perp}(-t  \laplcomp^{E_M^{\xi, n}}); m \exp^{\perp}(-t  \laplcomp^{E_P^{n}}) \big],
		\end{align}
		where $3 \exp^{\perp}(-t  \laplcomp^{E_M^{\xi, n}})$ (resp. $m \exp^{\perp}(-t  \laplcomp^{E_P^{n}})$) means the heat operator on $M \sqcup M \sqcup M$ (resp. on $P \sqcup \cdots \sqcup P$) with the induced geometry.
	\end{rem}
	 For $j \geq -1$, we denote
		\begin{equation}\label{defn_ajMN}
		\begin{aligned}
			& A_{\xi, j, 0}^{M, n} := \int_M {\rm{Tr}} \big[ a_{\xi, j}^{M, n}(x) \big] dv_M(x) - \frac{\rk{\xi}}{3}  \int_P  a_{j}^{P, n}(x) dv_P(x),
			\\
			& A_{\xi, j}^{M, n} = A_{\xi, j, 0}^{M, n} 
			-  \dim H^0(\overline{M}, E_M^{\xi, n}) + \frac{\rk{\xi}}{3}  \dim H^0(\overline{P}, E_P^{n}).
		\end{aligned}
		\end{equation}
	This makes sense due to Theorem \ref{thm_small_time_exp} or Remark \ref{rem_integ_small_time}.
	\begin{prop}\label{rel_tr_small_time}
		For any $t_0 > 0$, $k \in \nat$, there is $C > 0$ such that for any $t \in ]0, t_0]$, we have
		\begin{equation}
			\Big| 
			{\rm{Tr}}^{\reg}  \big[ \exp^{\perp}(-t  \laplcomp^{E_M^{\xi, n}})\big] 
			- 
			\sum_{j=-1}^{k} A_{\xi, j}^{M, n} t^j
			\Big|  \leq C t^k.
		\end{equation}
	\end{prop}
	\begin{proof}
		First of all, by (\ref{eq_rel_nonperp}), it is enough to prove that for any $t_0 > 0$, $k \in \nat$, there is $C > 0$ such that for any $t \in ]0, t_0]$, we have
		\begin{equation}\label{eq_tr_rel_small_asy}
			\Big| 
			{\rm{Tr}}^{\reg}  \big[ \exp(-t  \laplcomp^{E_M^{\xi, n}}) \big] 
			- 
			\sum_{j=-1}^{k} A_{\xi, j, 0}^{M, n} t^j
			\Big|  \leq C t^k.
		\end{equation}
		By Theorem \ref{thm_small_time_exp}, for any $t_0 > 0$, $k \in \nat$, there is $C > 0$ such that for any $t \in ]0, t_0]$, we have
		\begin{equation}\label{eqn_small_t_as1}
		\begin{aligned}
			& \bigg| \int_{M \setminus (\cup_i V_i^{M}(e^{-t^{-1/3}}))} \Big[ \tr{\exp(-t  \laplcomp^{E_M^{\xi, n}})(x, x)} - \sum_{j=-1}^{k}  {\rm{Tr}} \big[ a_{\xi, j}^{M, n}(x) \big] t^j \Big] dv_M(x)  \bigg| \leq C t^k, \\
			& \bigg| \int_{P \setminus (\cup_i V_i^{P}(e^{-t^{-1/3}}))} \Big[  \exp(-t  \laplcomp^{E_P^{n}})(x, x) - \sum_{j=-1}^{k}   a_{j}^{P, n}(x) t^j \Big] dv_P(x)  \bigg| \leq C t^k.
		\end{aligned}
		\end{equation}
		Since for $u \in \comp$, $0 < |u| \leq e^{-t^{-1/3}}$, we have $t^{-1/3} \leq |\ln |u||$, by (\ref{hk_infty}), (\ref{eq_int_mu_fin}) and (\ref{eqn_small_t_as25}), for any $k \in \nat$ there are $c, C > 0$ such that for any $t \in ]0, t_0]$, $i = 1,\ldots, m$, we have
		\begin{multline}\label{eqn_small_t_as2}
			 \int_{D(e^{-t^{-1/3}})}
			\bigg|
				{\rm{Tr}} \Big[ \exp(-t \laplcomp^{E_M^{\xi, n}}) \big( (z_i^{M})^{-1}(u), (z_i^{M})^{-1}(u) \big) \Big]    
				\\
				- 
				\rk{\xi} \exp(-t \laplcomp^{E_P^{n}}) \big( (z^P)^{-1}(u), (z^P)^{-1}(u) \big)
			\bigg| d v_{\dd^*}(u)  \leq C t^k + C \exp(-ct^{-1/2}).
		\end{multline}
		Also, by (\ref{eq_small_time_coeff_rel}), for any $j \in \nat$, $i = 1,\ldots, m$ there are $c, C > 0$, such that we have
		\begin{multline}\label{eqn_small_t_as3}
			 \int_{D(e^{-t^{-1/3}})}
			\Big|
				{\rm{Tr}} \Big[ a_{\xi, j}^{M, n}\big( (z_i^{M})^{-1}(u) \big) \Big]    
				\\
				- 
				 \rk{\xi} a_{j}^{P, n} \big( (z^P)^{-1}(u) \big)
			\Big|  d v_{\dd^*}(u)  \leq C \exp(-ct^{-1/3}).
		\end{multline}
		We see that (\ref{eq_tr_rel_small_asy}) holds by (\ref{eqn_small_t_as1}), (\ref{eqn_small_t_as2}) and  (\ref{eqn_small_t_as3}).
	\end{proof}
	\begin{prop}\label{rel_tr_large_time}
		For any $t_0 > 0$, there are $c, C> 0$ such that for any $t \geq t_0$, we have
		\begin{equation}\label{eqn_rel_tr_large_time}
			\Big| {\rm{Tr}}^{\reg}   \big[ \exp^{\perp}(-t  \laplcomp^{E_M^{\xi, n}}) \big] \Big| \leq C \exp(-ct). 
		\end{equation}
	\end{prop}
	\begin{proof}
		By Theorem \ref{thm_hk_est} since $\rho_M$ is bounded over $M \setminus (\cup_i V_i^{M}(\eta))$, $\eta > 0$ for some $c, C > 0$ and for any $t \geq t_0$, we get
		\begin{equation}\label{tr_rel_perp_bound1}
		\begin{aligned}
	 	& 
	 	\bigg|  
	 		\int_{M \setminus (\cup_i V_i^{M}(\eta))} \tr{\exp^{\perp}(-t  \laplcomp^{E_M^{\xi, n}})(x, x)} d \vol_M(x)  
	 	\bigg| && \leq C \exp(- c t),
	 	\\ 
	 	& 
	 	\bigg|  
	 		\int_{P \setminus (\cup_i V_i^{P}(\eta))} \tr{\exp^{\perp}(-t  \laplcomp^{E_P^{n}})(x, x)} d \vol_P(x) 
	 	\bigg| && \leq C \exp(- c t).
	 	\end{aligned}
		\end{equation}
		By (\ref{reqr_poincare}), (\ref{hk_infty_perp}) and (\ref{eq_int_mu_fin}), we deduce that there are $c, C > 0$ such that for any $t \geq t_0$, we have
		\begin{multline}\label{tr_rel_perp_bound12}
			\Big| \int_{D(\eta)} \Big( {\rm{Tr}} \Big[ \exp^{\perp}(-t  \laplcomp^{E_M^{\xi, n}}) \big((z_{i}^{M})^{-1}(u), (z_{i}^{M})^{-1}(u)\big) \Big]
	 		\\
	 		- \rk{\xi} {\rm{Tr}} \Big[ \exp^{\perp}(-t  \laplcomp^{E_P^{n}})\big((z^P)^{-1}(u), (z^P)^{-1}(u)\big) \Big] \Big) 
	 		d \vol_{\dd^*}(u) \Big| \leq C \exp(-c t).
		\end{multline}
		We conclude by (\ref{tr_rel_perp_bound1}) and (\ref{tr_rel_perp_bound12}).
	\end{proof}
	\begin{defn}\label{defn_rel_zeta}
		We define the \textit{regularized spectral zeta function} $\zeta_{M}(s)$ for $s \in \comp$, $\Re(s) > 1$ by
		\begin{equation}\label{eq_defn_rel_zeta}
			\zeta_{M}(s) = \frac{1}{\Gamma(s)} \int_{0}^{+ \infty}
			{\rm{Tr}}^{\reg}  \big[ \exp^{\perp}(-t  \laplcomp^{E_M^{\xi, n}}) \big] t^{s} \frac{dt}{t}.
		\end{equation}
	\end{defn}
	By Propositions \ref{rel_tr_small_time} and \ref{rel_tr_large_time}, the function $\zeta_{M}(s)$ is holomorphic for $\Re(s) > 1$ and has a meromorphic extension to the entire $s$-plane. Classically, this extension, which we also denote by $\zeta_{M}(s)$, is holomorphic at $s = 0$.
	\begin{defn}\label{defn_rel_tor}
		We define the \textit{analytic torsion} by
		\begin{equation}\label{eq_defn_rel_tor}
			T (g^{TM}, h^{\xi} \otimes \norm{\cdot}_{M}^{2n}) := \exp(- \zeta_{M}'(0)/2) \cdot T_{TZ}(g^{TP}, \norm{\cdot}_{P}^{2n})^{m \cdot \rk{\xi} / 3}.
		\end{equation}
	\end{defn}
	\begin{rem}\label{rem_def_an_tor}
		a) In the forthcoming paper we show that under the conditions (\ref{cond_stable}), we have
		\begin{equation}\label{eqn_compat_thmg}
			T (g^{TM}_{{\rm{hyp}}}, (\, \norm{\cdot}_{M}^{{\rm{hyp}}})^{2n}) = T_{TZ}(g^{TM}_{{\rm{hyp}}}, (\, \norm{\cdot}_{M}^{{\rm{hyp}}})^{2n}).
		\end{equation}
		For the moment we note that (\ref{eqn_compat_thmg}) holds for $M = P$ by the choice of the last multiplicand in (\ref{eq_defn_rel_tor}).
		\par b) Explicitly, we have the following identity (see Proposition \ref{rel_tr_small_time} for the definition of $a^{M}_{-1}$):
		\begin{multline}\label{eq_der_zeta}
			\zeta_{M}'(0)  = 
			\int_{0}^{1}
			\Big(
			{\rm{Tr}}^{\reg}  \big[ \exp^{\perp}(-t  \laplcomp^{E_M^{\xi, n}}) \big]  - \frac{A_{\xi, -1}^{M, n}}{t} - A_{\xi, 0}^{M, n} \Big) \frac{dt}{t}
			\\
			+
			\int_{1}^{+\infty}
			{\rm{Tr}}^{\reg}  \big[ \exp^{\perp}(-t  \laplcomp^{E_M^{\xi, n}}) \big] \frac{dt}{t}
			+ A_{\xi, -1}^{M, n}	
			- \Gamma'(-1) A_{\xi, 0}^{M, n}.
		\end{multline}
		\par c) By (\ref{eq_defn_rel_zeta}), the relation between the relative analytic torsion, defined by Jorgenson-Lundelius \cite{JorLundMain} for $(\xi, h^{\xi})$ trivial and $n=0$, and our definition is
		\begin{equation}
			T^{{\rm{rel}}}(g^{TM}, h^{\xi} \otimes \norm{\cdot}_{M}^{2n}; g^{TN}, \norm{\cdot}_{N}^{2n}) = \frac{T (g^{TM}, h^{\xi} \otimes \norm{\cdot}_{M}^{2n})}{T (g^{TN}, \norm{\cdot}_{N}^{2n})^{\rk{\xi}}}.
		\end{equation}
		\par d) In the article of Albin-Rochon \cite{AlbRoch}, authors gave an alternative definition of the analytic torsion $T_{AR}(g^{TM}, (\, \norm{\cdot}_{M}^{\rm{hyp}})^{2n})$, $n = 0$.
		By \cite[(1.24)]{AlbRoch_gen}, \cite[\S 7]{AlbRoch} and (\ref{hk_infty}), the relation between their definition and ours is given by
		\begin{equation}
			\frac{T (g^{TM}, (\, \norm{\cdot}_{M})^{2n})}{T (g^{TN}, (\, \norm{\cdot}_{N})^{2n})} 
			=
			\frac{T_{AR} (g^{TM}, (\, \norm{\cdot}_{M})^{2n})}{T_{AR} (g^{TN}, (\, \norm{\cdot}_{N})^{2n})},
		\end{equation}
		for $M, N$ as in (\ref{data_rel_tors}). Their definition is based on $b$-trace of Melrose \cite{MelroseAPS}, see Remark \ref{rem_ht}b).
		\par e) In his thesis \cite[Corollary 8.2.2]{FreixTh}, Freixas explicitly evaluated (see (\ref{defn_zeta_sel}))
		\begin{equation}\label{eq_value_zeta_explicit}
			\log Z'_{(\overline{P}, D_P)}(1) = 4 \zeta'(-1) + \log 2 \pi + \frac{10}{9} \log 2.
		\end{equation}
		By combining (\ref{eqn_sel_norm}), (\ref{eq_value_zeta_explicit}), we may give an explicit formula for $T_{TZ}(g^{TP}, \norm{\cdot}_{P}^{2n})$, $n = 0$ in (\ref{eq_defn_rel_tor}).
	\end{rem} 
	
\subsection{The parametrix for the heat kernel on the punctured hyperbolic disc}\label{sect_pametr}
	In this section we recall the well-known construction \cite[\S 2.4, 2.5]{BGV} of the parametrix, applied for the heat kernel on the punctured hyperbolic disc, endowed with a Hermitian vector bundle. The results of this section will be applied in Section \ref{sect_aux} to prove (\ref{thm_est_exp2}) and Theorem \ref{thm_small_time_exp}.
	\par
	Let's explain the setting of the problem. Let the holomorphic line bundle $\omega_{\dd}(0)$ over $\dd$ be defined as in (\ref{eq_om_md_def}), and let $\norm{\cdot}_{\dd}$ be the norm on $\omega_{\dd}(0)$ over $\dd^*$, induced by $g^{T \dd^*}$ as in (\ref{eqn_norms_local}). Let $(\xi, h^{\xi})$ be Hermitian vector bundle over $\dd$. We denote by the same symbol its restriction to $\dd^*$. By Cartan's Theorem A, we may and we do fix a holomorphic trivialization $e_1, \ldots, e_{\rk{\xi}}$ of $\xi$ over $\dd$.
	We suppose that 	it is a normal trivialization (cf. \cite[Proposition V.12.10]{DemCompl}), i.e. we have 
	\begin{equation}\label{defn_normal_triv}
		h^{\xi}(e_i, e_j)(u) = \delta_{ij} + O(|u|^2).
	\end{equation}
	We denote by $\laplcomp^{\xi \otimes \omega_{\dd}(0)^{n}}$, $n \in \integ$ the Kodaira Laplacian associated with $h^{\xi}$ and $\norm{\cdot}_{\dd}$ on $(\dd^*, g^{T \dd^*})$. The smooth kernel
	\begin{equation}\label{eq_hk_dd}
		\exp(-t \laplcomp^{\xi \otimes \omega_{\dd}(0)^{n}})(z_1, z_1)
		\in (\xi \otimes \omega_{\dd}(0)^{n})^{*}_{z_1} \boxtimes (\xi \otimes \omega_{\dd}(0)^{n})_{z_2},
		 \quad \text{for} \quad z_1, z_1 \in \dd^{*},
	\end{equation}
	of the heat operator $\exp(-t \laplcomp^{\xi \otimes \omega_{\dd}(0)^{n}})$ with respect to the volume form $d \vol_{\dd^*}$ on $\dd^{*}$ is the main object of study in this section.
	\par 
	We consider the covering $\rho : \hh \to \dd^*$, $z \mapsto e^{\imun z}$. 
	The metric $g^{T \hh} := \rho^* (g^{T \dd^*})$ is equal to the standard hyperbolic metric on the upper half-plane.
	The Deck transformations of $\rho$ are generated by the isometry 
	\begin{equation}
		U : \hh \to \hh, \qquad z \mapsto z+ 2 \pi.
	\end{equation}
	Let $\norm{\cdot}_{\hh}$ be the norm on $\omega_{\hh}$, given by $\rho^* (\, \norm{\cdot}_{\dd})$. 
	For $z = (x, y) := x + \imun y$, we have
	\begin{equation}
		g^{T \hh}_z = \frac{dx^2 + dy^2}{y^2}, \qquad \norm{dz}_{\hh}(z) = y.
	\end{equation}
	Let $\laplcomp^{\xi \otimes \omega_{\hh}^{n}}$ be the Kodaira Laplacian associated with $g^{T \hh}$, $\rho^* ( h^{\xi})$, $\norm{\cdot}_{\hh}$ on $(\hh, g^{T \hh})$, and let 
	\begin{equation}\label{eq_hk_hh}
		\exp(-t \laplcomp^{\xi \otimes \omega_{\hh}^{n}})(z_1, z_2)
		\in (\rho^*(\xi) \otimes \omega_{\hh}^{n})^{*}_{z_1} \boxtimes (\rho^*(\xi) \otimes \omega_{\hh}^{n})_{z_2},
		 \quad \text{for} \quad z_1, z_2 \in \hh,
	\end{equation}
	be the smooth kernel of the heat operator $\exp(-t \laplcomp^{\xi \otimes \omega_{\hh}^{n}})$ with respect to the Riemannian volume form $d \vol_{\hh}$ on $\hh$, induced by $g^{T \hh}$. For $z_1, z_2 \in \dd$, the relation between (\ref{eq_hk_dd}) and (\ref{eq_hk_hh}) is given by
	\begin{equation}\label{eq_rel_hk_dd_hh}
		\exp(-t \laplcomp^{\xi \otimes \omega_{\dd}(0)^{n}})(z_1, z_2) = \sum_{i \in \integ} \exp(-t \laplcomp^{\xi \otimes \omega_{\hh}^{n}})(\tilde{z}_1, U^i \tilde{z}_2),
	\end{equation}
	where $\tilde{z}_i \in \hh$, $\rho(\tilde{z}_i) = z_i$ for $i = 1,2$.
	\par Since  $(\hh, g^{T \hh})$ is a compete manifold, we may use the framework of \cite[\S 2.4, 2.5]{BGV} to construct the parametrix of $\exp(-t \laplcomp^{\xi \otimes \omega_{\hh}^{n}})$. Let us briefly recall the main steps of this construction. By doing so, we also provide some uniform estimates on the heat kernels.
	\par We denote by $\dist(z_1,z_2)$, $z_1, z_2 \in \hh$ the Riemannian distance associated with $g^{T \hh}$, we have
	\begin{equation}\label{eq_dist_hyper}
		\dist \big((x_1, y_1), (x_2, y_2) \big) 
		= 2 \ln \bigg(
			\frac{\sqrt{(x_1 - x_2)^2 + (y_1 - y_2)^2}
			+
			\sqrt{(x_1 - x_2)^2 + (y_1 + y_2)^2}}{2 \sqrt{y_1 y_2}}
		 \bigg).
	\end{equation}
	\par Let $\psi : \real \to [0,1]$ be a smooth even function such that
	\begin{equation}\label{eq_defn_psi}
		\psi(u) = 
		\begin{cases} 
      		\hfill 1 & \text{ for } |u| < 1/2, \\
      		\hfill 0  & \text{ for } |u| > 1. \\
 		\end{cases}
	\end{equation}
	For $k \in \nat$, $z_1, z_2 \in \hh$, $t > 0$, let $k^{\xi \otimes \omega_{\hh}^n}_{t,k} \in \ccal^{\infty}(\hh \times \hh, (\rho^*(\xi) \otimes \omega_{\hh}^{n}) \boxtimes (\rho^*(\xi) \otimes \omega_{\hh}^{n})^{*})$ be given by (cf. \cite[(2.7)]{BGV})
	\begin{equation}\label{eq_par_defn_k}
		k^{\xi \otimes \omega_{\hh}^n}_{t,k}(z_1, z_2) := \frac{\psi(\dist(z_1,z_2)^2)}{t}		 \exp \Big(- \frac{\dist(z_1,z_2)^2}{4t} \Big) 
		\Big( 
			\sum_{i=0}^{k} t^i \Phi_{i, n}^{\xi}(z_1,z_2)		
		\Big),
	\end{equation}
	where $\Phi_{i, n}^{\xi} \in \ccal^{\infty}(\hh \times \hh, (\rho^*(\xi) \otimes \omega_{\hh}^{n}) \boxtimes (\rho^*(\xi) \otimes \omega_{\hh}^{n})^{*})$, $i \geq 0$ are symmetric (i.e. $\Phi_{i, n}^{\xi}(z_1, z_2) = (\Phi_{i, n}^{\xi}(z_2, z_1) )^*$) and given by the procedure, described in \cite[Theorem 2.26]{BGV}.
	We denote by $\Phi_{i, n}$, $i \geq 0$ the sections associated to $(\xi, h^{\xi})$ trivial.
	Now let's state the main result of this section.
	\begin{thm}\label{thm_bound_phi}
		The sections $\Phi_{i, n}^{\xi}$ are uniformly $\ccal^{\infty}$-bounded in the following sense: for any $i, l, l' \in \nat$, there is $C > 0$ such that for any $z_1, z_2 \in \hh$, we have
		\begin{equation}\label{eq_bound_phi1}
			\big| (\nabla_{z_1})^l (\nabla_{z_2})^{l'} \Phi_{i, n}^{\xi}(z_1,z_2) \big|_{h \times h} \leq C,
		\end{equation}
		where $\nabla$ is induced by the Levi-Civita connection and Chern connections associated with $(\xi, h^{\xi})$, $(\omega_{\hh}(0), \norm{\cdot}_{\hh})$, and $|\cdot|_{h \times h}$ is the associated pointwise norm.
		\par 
		Moreover, for any $i, l, l' \in \nat$, there is $C > 0$ such that for any $z_1, z_2 \in \hh$, we have
		\begin{equation}\label{eq_bound_phi2}
			\Big| (\nabla_{z_1})^l (\nabla_{z_2})^{l'} \big( \Phi_{i, n}^{\xi} - {\rm{Id}}_{\xi} \cdot \Phi_{i, n} \big)(z_1,z_2)  \Big|_{h \times h} \leq C \exp(-(\Im{z_1}+\Im{z_2})/6).
		\end{equation}
	\end{thm}
	\begin{proof}
		Let's fix $z_0 \in \hh$, $z_0 = (x_0, y_0)$. For $z \in \hh$, $r > 0$ we denote by $B^{\hh}(z, r) \subset \hh$ the hyperbolic disc of radius $r$ around $z$. We consider the isometry $g_{z_0} : (\hh, g^{T \hh}) \to (\hh, g^{T \hh})$, $(x, y) \mapsto ((x - x_0)/y_0, y/y_0 )$. Then since $g_{z_0}(z_0) = (0, 1) := \imun$, we have $g_{z_0}(B^{\hh}(z_0, 1)) = B^{\hh}(\imun, 1)$. We recall that by the procedure, described in \cite[Theorem 2.26]{BGV}, the sections $\Phi_{i, n}^{\xi}(z,\cdot)$ are defined locally, i.e. they depend only on the restriction of $(\hh, g^{T \hh})$, $(\xi, h^{\xi})$ over $B^{\hh}(z, 1)$, and if $\dist(z, z_2) > 1$, then $\Phi_{i, n}^{\xi}(z,z_2) = 0$. Moreover, if one changes “smoothly" the parameters $g^{T \hh}$, $h^{\xi}$, then the sections $\Phi_{i, n}^{\xi}(z,\cdot)$ also change smoothly “at the same rate". Let's make the last point more precise, and adapt it for our situation.
		\par Let $h^{\xi}_{z}$, $z \in \hh$ be a family of Hermitian metrics on $\xi$ over $B^{\hh}(\imun, 1)$, and let $\Phi_{i, n, z}^{\xi}(\imun, \cdot)$ be the corresponding sections from (\ref{eq_par_defn_k}). Suppose that there is $f: \real_+ \to \real_+$ and a Hermitian metric $h^{\xi}_{0}$ on $\xi$ over $B^{\hh}(\imun, 1)$ such that for any $l \in \nat$, there is $C > 0$ such that for any $z_2 \in B^{\hh}(\imun, 1)$:
		\begin{equation}\label{eq_bound_hxi}
			\begin{aligned}
				& \big| \nabla^{l} (h^{\xi}_{z})(z_2) \big|_{h} && \leq C,
				\\
				& \big| \nabla^{l} (h^{\xi}_{z} - h^{\xi}_{0})(z_2) \big|_{h} && \leq C f(\Im z).
			\end{aligned}
		\end{equation}
		\par From the procedure, described in \cite[Theorem 2.26]{BGV}, the sections $\Phi_{i, n, z}^{\xi}(\imun,\cdot)$ are obtained by applying the associated Laplacian and integration over the geodesics of length $\leq 1$, emanating from $\imun$. Thus, for any $l \in \nat$ there is $C > 0$ such that for any $z_2 \in B^{\hh}(\imun, 1)$, we have
		\begin{equation}\label{eq_bound_phiepsilon}
			\begin{aligned}
				& \big| (\nabla_{z_2})^{l} \Phi_{i, n, z}^{\xi}(\imun, z_2) \big|_{h} &&\leq C, \\
				& \big| (\nabla_{z_2})^{l} \big( \Phi_{i, n, z}^{\xi} - \Phi_{i, n, 0}^{\xi} \big)(\imun, z_2)  \big|_{h} &&\leq C f(\Im z),
			\end{aligned}
		\end{equation}
		or, as we stated before the sections $\Phi_{i, n, z}^{\xi}(\imun,\cdot)$, $i \geq 0$ change “at the same rate" as $h^{\xi}_{z}$.
		\begin{sloppypar}		 
		Now, let $h^{\xi}_{z}$, $z \in \hh$ be defined by 
		\begin{equation}\label{defn_z_metr}
			h^{\xi}_{z} := ((g_z^{-1} \rho)^* h^{\xi})|_{B^{\hh}(\imun, 1)}.
		\end{equation}
		Let $\Phi_{i, n, z}^{\xi}(\imun,\cdot)$ be the sections from (\ref{eq_par_defn_k}), associated with $h^{\xi}_{z}$, $\norm{\cdot}_{\hh}|_{B^{\hh}(\imun, 1)}$ and $g^{T \hh}|_{B^{\hh}(\imun, 1)}$. Then by the locality of $\Phi_{i, n, z}^{\xi}(\imun,\cdot)$, for any $z_2 \in B^{\hh}(\imun, 1)$, we have 
		\begin{equation}
			\Phi_{i, n, z}^{\xi}(\imun, z_2) = \Phi_{i, n}^{\xi}(z, g_{z}^{-1} (z_2)).
		\end{equation}
		By symmetry of $\Phi_{i, n, z}^{\xi}$ and (\ref{eq_bound_phiepsilon}), to complete the proof of Theorem \ref{thm_bound_phi}, it is enough to prove the analogue of (\ref{eq_bound_hxi}), i.e. for some $C > 0$, we have for any $z_2 = (x,y) \in B^{\hh}(\imun, 1)$:
		\begin{equation}\label{eq_h_xi_appr_1}
			\begin{aligned}
				& \big| (\partial_x)^l (\partial_y)^{l'} (h^{\xi}_{z})(z_2) \big|_{h} && \leq C, \\
				& \big| (\partial_x)^l (\partial_y)^{l'} (h^{\xi}_{z} - {\rm{Id}}_{\rk{\xi}})(z_2) \big|_{h} && \leq C e^{- \Im z / 3}.
			\end{aligned}
		\end{equation}
		Let the frame $e_1, \ldots, e_{\rk{\xi}}$ be chosen as (\ref{defn_normal_triv}). Then for $z_2 = (x,y) \in B^{\hh}(\imun, 1)$:
		\begin{equation}\label{eq_h_xi_appr_2}
			h^{\xi}_{z}((g_{z}^{-1} \rho)^* e_i, (g_{z}^{-1} \rho)^* e_j)(z_2) = h^{\xi}(e_i, e_j)(e^{ - y y_0 + \imun (x y_0 + x_0)}).
		\end{equation}
		Now, by the formula (\ref{eq_dist_hyper}), we have 
		\end{sloppypar}
		\begin{equation}
		\min \{ \Im z : z \in B^{\hh}(\imun, 1) \} \geq 1/6, 
		\end{equation}
		by (\ref{defn_normal_triv}) and (\ref{eq_h_xi_appr_2}), we have (\ref{eq_h_xi_appr_1}), which finishes the proof.
	\end{proof}
	To compare $k^{\xi \otimes \omega_{\hh}^n}_{t,k}(x, y)$ with the heat kernel, we recall the definition of the “defect":
	\begin{equation}\label{eq_par_defn_r}
		r^{\xi \otimes \omega_{\hh}^n}_{t,k}(z_1,z_2) := \big( \partial_t + \laplcomp_{\dd, x}^{\xi \otimes \omega_{\dd}(0)^{n}} \big) k^{\xi \otimes \omega_{\hh}^n}_{t,k}(z_1,z_2).
	\end{equation}
	The following theorem says, in particular, that as one increases $k \in \nat$, the kernel $k^{\xi \otimes \omega_{\hh}^n}_{t,k}(z_1,z_2)$ more and more accurately “satisfies" the properties defined by the heat kernel.
	\begin{thm}\label{thm_bound_r}
		For any $t_0 > 0$, the family of kernels $k^{\xi \otimes \omega_{\hh}^n}_{t,k}(z_1,z_2)$, $t \in ]0, t_0]$, $z_1,z_2 \in \hh$ defines a uniformly bounded family of operators $K^{\xi \otimes \omega_{\hh}^n}_{t,k}$ on $\ccal^{\infty}_c(\hh, \rho^*(\xi) \otimes \omega_{\hh}^{n})$ such that for any $s \in \ccal^{\infty}_c(\hh, \rho^*(\xi) \otimes \omega_{\hh}^{n})$, the sections $K^{\xi \otimes \omega_{\hh}^n}_{t,k} (s)$ converge, as $t \to 0$, to $s$ over any compact subset of $\hh$ with all it's derivatives.
		\par Moreover, for any $l, l', l'' \in \nat$, there are $c', C > 0$ such that for any $t \in ]0, t_0]$, $z_1,z_2 \in \hh$:
		\begin{multline}\label{eq_bound_k}
			\big| (\nabla_{z_1})^l (\nabla_{z_2})^{l'} (\partial_t)^{l''} k^{\xi \otimes \omega_{\hh}^n}_{t,k}(z_1,z_2) \big|_{h \times h} 
			\leq C t^{- 1 - (l + l')/2 - l''} \cdot \psi(\dist(z_1,z_2)^2 / 2) 
			\\
			\cdot			
			\exp(-c' \cdot \dist(z_1,z_2)^2/t),
		\end{multline}
		\vspace*{-30pt}
		\begin{multline}\label{eq_bound_r}
			\big| (\nabla_{z_1})^l (\nabla_{z_2})^{l'} (\partial_t)^{l''} r^{\xi \otimes \omega_{\hh}^n}_{t,k}(z_1,z_2) \big|_{h \times h} 
			\leq C  t^{k - (l + l')/2 - l''} \cdot \psi(\dist(z_1,z_2)^2 / 2)
			\\
			\cdot \exp(-c' \cdot \dist(z_1,z_2)^2/t).
		\end{multline}
	\end{thm}
	\begin{proof}
		The first statement is done as in \cite[Theorem 2.29]{BGV}. The estimate (\ref{eq_bound_k}) follows directly from (\ref{eq_par_defn_k}) and (\ref{eq_bound_phi1}).  The proof of (\ref{eq_bound_r}) uses (\ref{eq_bound_phi1}), but otherwise it is done in the same way as \cite[Theorem 2.29]{BGV}.
	\end{proof}
	This theorem means that $k^{\xi \otimes \omega_{\hh}^n}_{t,k}(z_1,z_2)$ is the parametrix of the heat equation in the sense of \cite[p.77]{BGV}. Thus, we may construct the heat kernel by the procedure, which follows. Let $k, k' \in \nat$, $z,z' \in \hh$, we denote
	\begin{multline}\label{eq_defn_q_appr}
		q^{\xi \otimes \omega_{\hh}^n}_{t,k,k'}(z,z') := \int_{t \Delta_{k'}} \int_{\hh^{k'}} k^{\xi \otimes \omega_{\hh}^n}_{t-t_{k'},k}(z, z_{k'}) r^{\xi \otimes \omega_{\hh}^n}_{t_{k'} - t_{k' - 1},k}(z_{k'},z_{k' - 1}) \cdots
		\\
		\cdots
		 r^{\xi \otimes \omega_{\hh}^n}_{t_{1},k}(z_{1}, z') \, dv_{\hh}(z_{k'}) \otimes \cdots \otimes dv_{\hh}(z_{1}) dv_{t \Delta_{k'}}(t_1, \ldots, t_{k'}),
	\end{multline}
	where $\Delta_{k'}$ is the standard $k'$-simplex, and $dv_{t \Delta_{k'}}(t_1, \ldots, t_{k'})$ is the standard volume form over $t \Delta_{k'}$. 
	Now let's explain why (\ref{eq_defn_q_appr}) is well-defined.	
	The integration over $\hh^{k'}$ in (\ref{eq_defn_q_appr}) is well-defined since by (\ref{eq_defn_psi}), (\ref{eq_par_defn_k}) and (\ref{eq_par_defn_r}), the functions under the integral vanish if the arguments are too distant, so all the integrations are done in a compact subset. The integration over $t \Delta_{k'}$ is well-defined for $k \geq 1$ by (\ref{eq_par_defn_k}) and (\ref{eq_bound_r}). By the same reasons, it is easy to see that if $k \geq (l + l')/2 + l'' + 1$, then the partial derivatives $(\partial_{z_1})^l (\partial_{z_2})^{l'} (\partial_t)^{l''} q^{\xi \otimes \omega_{\hh}^n}_{t,k, k'}(z_1,z_2)$ exists.
	\begin{thm}
		For any $t_0 > 0$, $k \in \nat^*$ and $l, l', l'' \in \nat$, there are $c', C > 0$ such that for any $t \in ]0, t_0]$, $z_1, z_2 \in \hh$ and $k' \in \nat$ satisfying $k \geq (l + l')/2 + l'' + 1$, we have
		\begin{equation}\label{eq_bound_q}
			\Big| (\nabla_{z_1})^l (\nabla_{z_2})^{l'} (\partial_t)^{l''} q^{\xi \otimes \omega_{\hh}^n}_{t,k, k'}(z_1,z_2) \Big|_{h \times h} \leq  \frac{C^{k'} t^{k k' - l - l' - 2l''}}{(k'-1)!} \exp(-c' \cdot \dist(z_1,z_2)^2/t).
		\end{equation}
		\par Moreover, for any $t \in ]0, t_0]$, $z_1, z_2 \in \hh$, $k \in \nat^*$, the series 
		\begin{equation}
			\sum_{k' = 0}^{\infty} (-1)^{k'} q^{\xi \otimes \omega_{\hh}^n}_{t,k, k'}(z_1, z_2),
		\end{equation}		 
		converges to $\exp(-t \laplcomp_{\hh}^{\xi \otimes \omega_{\hh}^{n}})(z_1, z_2)$ in $\ccal^{2k - 2}(\hh \times \hh)$, and for any $l, l', l'' \in \nat$, satisfying $k \geq (l + l')/2 + l'' + 1$, there is $C > 0$ such for any $t \in ]0, t_0]$, $z_1, z_2 \in \hh$, we have
		\begin{multline}\label{eq_exp_appr_k}
			\Big| (\nabla_{z_1})^l (\nabla_{z_2})^{l'} (\partial_t)^{l''} \Big( \exp(-t \laplcomp^{\xi \otimes \omega_{\hh}^{n}}) - k^{\xi \otimes \omega_{\hh}^n}_{t,k}\Big)(z_1,z_2)  \Big|_{h \times h} 
			\\
			\leq C t^{k - l - l' - 2l''}  \exp(-c' \cdot \dist(z_1,z_2)^2/t).
		\end{multline}
	\end{thm}
	\begin{proof}
		First of all, we note that by the weighted mean inequality and the  triangle inequality, for $k' \in \nat$, $t > t_{k'} > \ldots > t_{1} > 0$, and $z, z', z_1, \ldots, z_{k'} \in \hh$, we have
		\begin{multline}\label{eq_triangle_exp}
			\exp \Big(-c' \frac{\dist(z, z_{k'})^2}{t-t_{k'}} \Big) \exp \Big(-c' \frac{\dist(z_{k'}, z_{k'-1})^2}{t_{k'} - t_{k' - 1}} \Big) \cdots \\
			\cdots \exp \Big(-c' \frac{\dist(z_1, z')^2}{t_1} \Big) \leq \exp \Big(- \frac{c'}{t} \dist(z, z')^2 \Big).
		\end{multline}
		We also note that the integration over each variable $z_1, \ldots, z_{k'}$ is done over a hyperbolic ball of radius $1$, which has a constant volume, independently of the choice of its center.
		From now on, the proof remains verbatim with \cite[Lemma 2.22, Theorem 2.23]{BGV}, where one has to replace the appropriate estimates by (\ref{eq_bound_k}), (\ref{eq_bound_r}) and use (\ref{eq_triangle_exp}) to bound the exponentials.
	\end{proof}
	Now let's apply all this theory to the study of the heat kernel on the hyperbolic disc. We summarize all the important results, which will be used in Section \ref{sect_aux}, in the following theorem, which is a local analogue of (\ref{thm_est_exp2}) and Theorem \ref{thm_small_time_exp}.
	\begin{thm}\label{thm_local_summary}
		For any $l, l', l'' \in \nat$, there are $t_0 > 0$ $c, c', C > 0$ such that for any $t \in ]0, t_0]$, $u, v \in \dd^*$, we have
		\begin{multline}\label{eq_bound_exp_sm_t}
			\Big| (\nabla_{u})^l (\nabla_{v})^{l'} (\partial_t)^{l''} \exp(-t \laplcomp^{\xi \otimes \omega_{\dd}(0)^{n}})(u, v) \Big|_{h \times h} \leq C t^{- 1 - (l + l')/2 - l''}  
			\\
			\cdot  \big( 1 + |\ln |u|| \big)^{1/2} \big(1 + |\ln |v|| \big)^{1/2} \exp \big(- c' \cdot \dist(u, v)^2/t \big).
		\end{multline}
		Moreover, there are bounded sections $a_{\xi, j}^{\dd^*, n} \in \ccal^{\infty}(\dd^*, \enmr{\xi})$, $j \geq -1$ such that there are $c', C > 0$ such that for any $u \in \dd^*$, $k \in \nat$ and $t \in ]0, t_0]$, we have
		\begin{multline}\label{eq_local_sm_time}
			\Big| \exp(-t \laplcomp^{\xi \otimes \omega_{\dd}(0)^{n}})(u, u) - \sum_{j = - 1}^{k} a_{\xi, j}^{\dd^*, n} (u) t^j \Big| 
			\\
			\leq  \Big(1 + |\ln |u|| \Big) \Big( C t^k  + \frac{C}{t} \exp \Big(- \frac{c'}{t |\ln |u||^2} \Big) \Big).
		\end{multline}
		\par 
		Moreover, for any $j \geq -1$, there is $C > 0$ such that for any $u \in \dd^*$, we have
		\begin{equation}\label{eq_small_time_coeff_rel_local}
			\Big| (\nabla_{u})^l \Big(  a_{\xi, j}^{\dd^*, n} - {\rm{Id}}_{\rk{\xi}} a_{j}^{\dd^*, n}  \Big)(u)  \Big|_{h \times h} \leq C  |u|^{1/3},
		\end{equation}
		where we trivialized $\xi$ as in the beginning of this section.
	\end{thm}
	Before proving this theorem, let's prove the following
	\begin{lem}\label{lem_dist_hyp_exp}
		There is $t_0 > 0$ such that for any $z_1, z_2 \in \hh$, $t \in ]0,t_0]$, satisfying $\dist(z_1, U^i z_2) \leq \dist(z_1, z_2)$ for any $i \in \integ$, we have
		\begin{equation}\label{eq_lem_dist_hyp_exp}
			\sum \exp \big(- \dist(z_1, U^i z_2)^2 / t \big) \leq C \Big( \big(\Im(z_1) + 1\big) \big(\Im(z_2) + 1 \big)\Big)^{1/2} \exp\Big(- \dist(z_1, z_2)^2 / (2t)\Big).
		\end{equation}
	\end{lem}
	\begin{proof}
		We decompose the sum in (\ref{eq_lem_dist_hyp_exp}) into two parts: for $i^2 \leq 4 \Im(z_1) \Im(z_2) $ and the complementary.
		Trivially, by the assumption, the first part is bounded by
		\begin{equation}\label{eq_bound_frstprt}
			4 \Big( \big(\Im(z_1) + 1 \big) \big(\Im(z_2) + 1 \big)\Big)^{1/2} \exp\Big(- \dist(z_1, z_2)^2 / (2t)\Big).
		\end{equation}
		Now, by (\ref{eq_dist_hyper}), for any $i \neq 0$, we have 
		\begin{equation}\label{eq_lem_dist_hyp_exp_1}
			\dist(z_1, U^i z_2) \geq \ln \big(i^2 / ( \Im(z_1) \Im(z_2) )\big).
		\end{equation}
		Now, by choosing $t_0$ small enough, we see that 
		\begin{equation}\label{eq_lem_dist_hyp_exp_1111}
			\exp \Big(- \Big( \ln \frac{i^2}{\Im(z_1) \Im(z_2)} \Big)^2/t \Big) \leq \frac{\Im(z_1) \Im(z_2)}{i^2}.
		\end{equation}
		By (\ref{eq_lem_dist_hyp_exp_1}) and (\ref{eq_lem_dist_hyp_exp_1111}), we see that 
		\begin{multline}\label{eq_lem_dist_hyp_exp_2}
			\sum_{i^2 > 4 \Im(z_1) \Im(z_2) } \exp(- \dist(z_1, U^i z_2)^2 / t) 
			\\			
			\leq \Im(z_1) \Im(z_2) \exp(- \dist(z_1, z_2)^2 / (2t)) \sum_{i^2 > 4 \Im(z_1) \Im(z_2) } i^{-2} 
			\\
			\leq (\Im(z_1) \Im(z_2) )^{1/2}  \exp(- \dist(z_1, z_2)^2 / (2t)).
		\end{multline}
		Thus, we conclude by (\ref{eq_bound_frstprt}) and (\ref{eq_lem_dist_hyp_exp_2}).
	\end{proof}
	\begin{proof}[Proof of Theorem \ref{thm_local_summary}]
		By (\ref{eq_rel_hk_dd_hh}), (\ref{eq_exp_appr_k}) and Lemma \ref{lem_dist_hyp_exp}, we have
		\begin{multline}\label{eq_local_summ_1111}
		\Big| (\nabla_{u})^l (\nabla_{v})^{l'} (\partial_t)^{l''} \Big( \exp(-t \laplcomp^{\xi \otimes \omega_{\dd}(0)^{n}})(u, v) - \sum_i k^{\xi \otimes \omega_{\hh}^n}_{t,k} (\tilde{u},U^i \tilde{v}) \Big)   \Big|_{h \times h}
		\\
		 \leq C t^{k - (l + l')/2 - l''}  \big(1 + |\ln |u|| \big)^{1/2} \big(1 + |\ln |v|| \big)^{1/2} \exp \big(-c' \cdot \dist(u, v)^2/t \big).
		\end{multline}
		Now, by (\ref{eq_lem_dist_hyp_exp_1}), there is $C > 0$ such that for any $z_1, z_2 \in \hh$, we have 
		\begin{equation}\label{eq_bound_ball}
			\# \Big\{ i \in \integ : \dist(z_1, U^i z_2) < 2 \Big\} \leq C \big( (\Im(z_1) + 1)(\Im(z_2) + 1) \big)^{1/2}.
		\end{equation}
		Thus, the number of non-zero terms in the sum under the module in (\ref{eq_local_summ_1111}) is bounded by the right-hand side of (\ref{eq_bound_ball}).
		So, by (\ref{eq_bound_k}) and (\ref{eq_local_summ_1111}), we get (\ref{eq_bound_exp_sm_t}).
		\par Now, by (\ref{eq_dist_hyper}), there is $C > 0$ such that for any $z \in \dd^*$ and $i \in \integ^*$, we have
		\begin{equation}\label{eq_boud_orbit}
			\dist(z, U^i z) \geq \frac{C}{ | \ln |z| |}.
		\end{equation}
		Thus, from Lemma \ref{lem_dist_hyp_exp}, (\ref{eq_local_summ_1111}) and (\ref{eq_boud_orbit}), we get (\ref{eq_local_sm_time}) for 
		\begin{equation}\label{def_a_j_local}
			a_{\xi, j}^{\dd^*, n} (z) := \Phi_{j + 1}^{\xi \otimes \omega_{\hh}^n}(\tilde{z}, \tilde{z}), \quad \text{where} \quad \tilde{z} \in \hh, \rho(\tilde{z}) = z, \quad \text{and} \quad j \geq -1.
		\end{equation}
		Now, (\ref{eq_small_time_coeff_rel_local}) follows from (\ref{eq_bound_phi2}) and (\ref{def_a_j_local}).
	\end{proof}
	\begin{rem}\label{rem_xiesp}
		Instead of $h^{\xi}$ we choose a family of Hermitian metrics $h^{\xi}_{\eta}$, $\eta \in ]0,1]$:
	\begin{equation}\label{eq_hxi_flat_disc}
			h^{\xi}_{\eta}(e_i, e_j)(u) := \big( 1 - \psi(|u|^2/\eta) \big) h^{\xi}(e_i, e_j)(u) +\psi(|u|^2/\eta) \delta_{ij},
	\end{equation}
	where $\psi$ is defined in (\ref{eq_defn_psi}), $e_i$, $i = 1,\ldots, \rk{\xi}$ is as in (\ref{defn_normal_triv}), and $\delta_{ij}$ is the Kronecker delta symbol.
	Then all the estimates of this chapter would continue to hold uniformly over $\eta \in ]0,1]$. This is due to the fact that for the Hermitian metrics $h^{\xi}_{z, \eta}$, defined in the notation of Theorem \ref{thm_bound_phi} by (compare with (\ref{defn_z_metr}))
	\begin{equation}
			h^{\xi}_{z, \eta} := ((g_z^{-1} \rho)^* h^{\xi}_{\eta})|_{B^{\hh}(\imun, 1)},
	\end{equation}
	we have (compare with (\ref{eq_h_xi_appr_2}))
	\begin{equation}\label{iden_exponen}
			h^{\xi}_{z, \eta}((g_{z}^{-1} \rho)^* e_i, (g_{z}^{-1} \rho)^* e_j)(z_2) = h^{\xi}_{\eta}(e_i, e_j)(e^{ - y y_0 + \imun (x y_0 + x_0)}).
	\end{equation}
	But now, if a function $f(y)$, $y \in \real$ has bounded derivative, then the function $f(e^{-y}/\eta)$ has bounded derivatives uniformly on $\eta > 0$. 
	From this observation, (\ref{iden_exponen}) implies (compare with (\ref{eq_h_xi_appr_1}))
	\begin{equation}\label{eq_boud_hxiepsi}
			\begin{aligned}
				& \big| (\partial_x)^l (\partial_y)^{l'} (h^{\xi}_{z, \eta})(z_2) \big|_{h} && \leq C, \\
				& \big| (\partial_x)^l (\partial_y)^{l'} (h^{\xi}_{z, \eta} - {\rm{Id}}_{\rk{\xi}})(z_2) \big|_{h} && \leq C e^{- \Im z / 3}.
			\end{aligned}
	\end{equation}
	We will use this in Section \ref{sect_spec_case}.
	\end{rem}

\subsection{Proofs of Theorems \ref{spec_gap_thm}, \ref{thm_hk_est}, \ref{thm_rel_hk_infty}, \ref{thm_small_time_exp}}\label{sect_aux}
	In this section we finally present the proofs of Theorems \ref{spec_gap_thm}, \ref{thm_hk_est}, \ref{thm_rel_hk_infty} and \ref{thm_small_time_exp}.
	\begin{proof}[Proof of Theorem \ref{spec_gap_thm}.]
		First of all, for $n \leq 0$, there is $C > 0$ such that for any $z \in D^*(1/2)$, we have
		\begin{equation}\label{ineq_n_leq_0}
			C < \big( |z|^2 (\ln |z|)^{2 - 2n} \big)^{-1}.
		\end{equation}
		Let $g^{TM}_{\rm{sm}}$ be some smooth metric over $\overline{M}$, and let $\norm{\cdot}_{M}^{\rm{sm}}$ be some smooth Hermitian norm on $\omega_M(D)$ over $\overline{M}$. 
		By (\ref{ineq_n_leq_0}), we conclude that there is $C > 0$ such that $g^{TM}_{\rm{sm}} \otimes (\, \norm{\cdot}_M^{\rm{sm}})^{2n} \leq C g^{TM} \otimes \norm{\cdot}_M^{2n}$, thus
		\begin{equation}\label{ker_subset}
			\ker (\laplcomp^{E_M^{\xi, n}}) \subset L^2 \big( g^{TM}_{\rm{sm}}, h^{\xi} \otimes ( \, \norm{\cdot}_{M}^{\rm{sm}})^{2n} \big).
		\end{equation}
		Let $s \in \ker (\laplcomp^{E_M^{\xi, n}})$. By (\ref{ker_subset}) and the classical $L^2$-extension theorem (cf. \cite[Lemma 2.3.22]{MaHol}), $s$ extends holomorphically to $V_i^{M}(\epsilon)$. In other words
		\begin{equation}\label{fin_sp_gp25_1}
			\ker (\laplcomp^{E_M^{\xi, n}}) \subset H^0 (\overline{M}, E_M^{\xi, n}).
		\end{equation}
		By the finiteness of the volume of $(M, g^{TM})$, see (\ref{eq_int_mu_fin}), we see that that each holomorphic section lies in $L^2(g^{TM}, h^{\xi} \otimes \, \norm{\cdot}_{M}^{2n})$, i.e.
		\begin{equation}\label{fin_sp_gp25_2}
			H^0 (\overline{M}, E_M^{\xi, n}) \subset \ker (\laplcomp^{E_M^{\xi, n}}) .
		\end{equation}
		We deduce (\ref{eqn_ker_lapl}) by (\ref{fin_sp_gp25_1}) and (\ref{fin_sp_gp25_2}).
		\par In \cite[\S 6]{MullerCusp}, Müller proved (\ref{eqn_spec_gap_strong}) for $(\xi, h^{\xi})$ trivial, $n = 0$ and $c_2 = 1/4$, see Remark \ref{rem_spec_gap_mul}. 
		This implies, in particular, that (\ref{eqn_spec_gap}) holds for $(\xi, h^{\xi})$ trivial and $n = 0$ (see \cite[Proposition 6.9]{MullerCusp}).
		He proved (\ref{eqn_spec_gap_strong}) in this case by studying explicitly the spectrum of Kodaira Laplacian of the von Neumann problem in the cusp and using the scattering matrix to relate the continuous spectrum of the manifolds.
		If the vector bundle $(\xi, h^{\xi})$ is trivial around the cusps, the presence of it doesn't change the Hermitian structure around $D_M$. Thus, the result of Müller extends line by line to the case $n = 0$ and $(\xi, h^{\xi})$ trivial around the cusps, which we summarize in
		\begin{equation}\label{eq_precise_gap}
			\spec (\laplcomp^{E_M^{\xi, n}}) \cap [0, 1/4] \qquad \text{is discrete.}
		\end{equation}
		\par Now, let $h^{\xi}$ be any Hermitian metric on $\xi$. 
		We will prove that there is $k \in \nat$ and $F \subset L^2(E_M^{\xi, n})$, ${\rm{codim}} F = k$, such that we have
		\begin{equation}\label{lem_comp_23}
			 \inf_{s \in F}
			 \bigg\{
				\frac{\langle \laplcomp^{E_M^{\xi, n}} s, \laplcomp^{E_M^{\xi, n}} s \rangle_{L^2}}{\langle s, s \rangle_{L^2}}
			\bigg\} > 0.
		\end{equation}
		Then, by the min-max theorem (cf. \cite[(C.3.3)]{MaHol}), (\ref{eqn_ker_lapl}) and (\ref{lem_comp_23}), we get (\ref{eqn_spec_gap}).
	\par
		We choose $\eta \in ]0,1/2]$ small enough, so that (\ref{reqr_poincare}) is satisfied for any $i = 1,\ldots, m$. For each $i = 1, \ldots, m$, we fix a normal trivialization of $\xi$ over $V_i^{M}(\eta)$, i.e. a local holomorphic frame $e_1, \ldots, e_{\rk{\xi}}$ of $\xi$ over $V_i^{M}(\eta)$ as in (\ref{defn_normal_triv}).
		Let $h^{\xi}_{\eta}$ be a Hermitian metric on $\xi$ such that it coincides with $h^{\xi}$ over $M \setminus (\cup_i V_i^{M}(\eta))$ and over $V_i^{M}(\eta)$ it is given by (compare with (\ref{eq_hxi_flat_disc}))
		\begin{equation}\label{eq_hxi_flat}
			h^{\xi}_{\eta}((z_i^{M})^{-1}(u))(e_i, e_j) = (1 - \psi(|u|^2/\eta)) h^{\xi}((z_i^{M})^{-1}(u))(e_i, e_j) +\psi(|u|^2/\eta) \delta_{ij},
		\end{equation}
		where $\psi$ is defined in (\ref{eq_defn_psi}), $e_i$, $i = 1,\ldots, \rk{\xi}$ is as in (\ref{defn_normal_triv}), and $\delta_{ij}$ is the Kronecker delta symbol.
		Then $(\xi, h^{\xi}_{\eta})$ is trivial around the cusps, and there is $C > 0$ such that for any $\eta \in ]0,1/2]$, we have
		\begin{equation}\label{eq_bound_cusped_der}
			(h^{\xi}_{\eta})^{-1}  \frac{\partial h^{\xi}_{\eta}}{\partial z_i^{M}}\Big((z_i^M)^{-1}(u)\Big) \leq C |u|.
		\end{equation}
		We denote by $\laplcomp^{E_M^{\xi, n}}_{\eta}$ the Kodaira Laplacian on $(M, g^{TM})$, associated with $h^{\xi}_{\eta}$ and $\norm{\cdot}_M$. Then over $V_i^{M}(\eta)$, we have
		\begin{equation}\label{eq_dbar_dual}
			(\dbar^{\xi})^* = \Big( \|dz_i^{M}\|_{M}^{\omega} \Big)^2 \Big( \frac{\partial}{\partial z_i^{M}}  +  (h^{\xi}_{\eta})^{-1} \frac{\partial h^{\xi}_{\eta}}{\partial z_i^{M}} \Big)
			\cdot
			\iota_{\partial/ \partial \overline{z}_i^{M}},
		\end{equation}
		where $\iota$ is the contraction and $*$ is the adjoint with respect to the $L^2$ product induced by $h^{\xi}_{\eta}$. By (\ref{eq_kodaira_laplacian}) and (\ref{eq_dbar_dual}), we deduce
		\begin{equation}\label{eq_comp_e_lapl}
			\laplcomp^{E_M^{\xi, n}}_{\eta} - \laplcomp^{E_M^{\xi, n}} 
			=
			\sum_i |z_i^{M}|^{2} (\ln |z_i^{M}|)^2 
			\Big( 
				(h^{\xi}_{\eta})^{-1}  \frac{\partial h^{\xi}_{\eta}}{\partial z_i^{M}} 
				-
				(h^{\xi})^{-1}  \frac{\partial h^{\xi}}{\partial z_i^{M}}
			\Big) 
			\frac{\partial}{\partial \overline{z}_i^{M}}.
		\end{equation}
		\par 
		We denote by $\langle \cdot, \cdot \rangle_{L^2_{\eta}}$ the $L^2$-scalar product induced by $g^{TM}$, $\norm{\cdot}$ and $h^{\xi}_{\eta}$. We fix $\eta > 0$ small enough so that $2 h^{\xi} > h^{\xi}_{\eta} > h^{\xi}/2$. Then we have $2 \langle \cdot, \cdot \rangle_{L^2} > \langle \cdot, \cdot \rangle_{L^2_{\eta}} > \langle \cdot, \cdot \rangle_{L^2}/2$.
		Now, by (\ref{eq_comp_e_lapl}) and Cauchy inequality, for $s \in \ccal^{\infty}_{c}(M, E_M^{\xi, n})$, as the support of (\ref{eq_comp_e_lapl}) lies in $\cup V_i^M(\eta)$, by (\ref{eq_bound_cusped_der}):
		\begin{equation}\label{eq_lapl_comp}
			\textstyle			
			\langle \laplcomp^{E_M^{\xi, n}} s, s \rangle_{L^2}
			\geq 
			\frac{1}{2}
			\langle \laplcomp^{E_M^{\xi, n}}_{\eta} s, s \rangle_{L^2_{\eta}}
			-
			2 m |\eta^2 \ln |\eta| | \big( \langle s, s \rangle_{L^2_{\eta}} \cdot \langle \laplcomp^{E_M^{\xi, n}}_{\eta} s, s \rangle_{L^2_{\eta}} \big)^{1/2}.
		\end{equation}
		We fix $\eta > 0$ small enough so that $4 m | \eta^2 \ln |\eta|| \leq 1/16$, and put 
		\begin{equation}\label{eq_lapl_comp222}
			F := \Big\langle \big\{ s \in {\rm{dom}} \big( \laplcomp^{E_M^{\xi, n}}_{\eta} \big) : \laplcomp^{E_M^{\xi, n}}_{\eta} s = \lambda s, \quad \text{for} \quad \lambda < 1/4 \big\} \Big\rangle^{\perp},
		\end{equation}
		where the orthogonal complement is taken with respect to $\langle \cdot, \cdot \rangle_{L^2_{\eta}}$. 
		Since $(\xi, h^{\xi}_{\eta})$ is trivial around the cusps, by (\ref{eq_precise_gap}), the space $F$ is of finite codimension. By (\ref{eq_lapl_comp}) and (\ref{eq_lapl_comp222}), for $s \in F$:
		\begin{equation}	\label{eq_lapl_comp3}
			\frac{\langle \laplcomp^{E_M^{\xi, n}} s, s \rangle_{L^2}}{\langle s, s \rangle_{L^2}}	
			\geq 
			\frac{1}{4} \bigg( \frac{\langle \laplcomp^{E_M^{\xi, n}}_{\eta} s, s \rangle_{L^2_{\eta}}}{\langle s, s \rangle_{L^2_{\eta}}}	\bigg)^{1/2}
			\bigg(
				\bigg( \frac{\langle \laplcomp^{E_M^{\xi, n}}_{\eta} s, s \rangle_{L^2_{\eta}}}{\langle s, s \rangle_{L^2_{\eta}}}	\bigg)^{1/2} - 
			\frac{1}{4}
			\bigg)
			\geq
			\frac{1}{32}.
		\end{equation}
		Also, by Cauchy inequality, we have
		\begin{equation}\label{equal_cauch}
			\bigg( \frac{\langle \laplcomp^{E_M^{\xi, n}} s, \laplcomp^{E_M^{\xi, n}} s \rangle_{L^2}}{\langle s, s \rangle_{L^2}} \bigg)^{1/2} \geq
			\frac{\langle \laplcomp^{E_M^{\xi, n}} s, s \rangle_{L^2}}{\langle s, s \rangle_{L^2}}	
		\end{equation}				
		Then (\ref{eq_lapl_comp3}) and (\ref{equal_cauch}) imply (\ref{lem_comp_23}), and thus (\ref{eqn_spec_gap}) holds for $n = 0$ and any $(\xi, h^{\xi})$.
		\par 
		We remark that similarly to (\ref{eq_lapl_comp}), we have
		\begin{equation}\label{eq_lapl_comp_xi_inter}
			\textstyle			
			\langle \laplcomp^{E_M^{\xi, n}}_{\eta} s, s \rangle_{L^2_{\eta}}
			\geq 
			\frac{1}{2}
			\langle \laplcomp^{E_M^{\xi, n}} s, s \rangle_{L^2}
			-
			2 m |\eta^2 \ln |\eta| | \big( \langle s, s \rangle_{L^2} \cdot \langle \laplcomp^{E_M^{\xi, n}} s, s \rangle_{L^2} \big)^{1/2}.
		\end{equation}
		From (\ref{lem_comp_23}) and (\ref{eq_lapl_comp_xi_inter}), in a similar fashion as we got (\ref{lem_comp_23}), we deduce that there exists $\mu > 0$ such that for any $\eta$ small enough, we have
		\begin{equation}\label{spec_gap_epsilon_unif}
			 \spec \Big(\laplcomp^{E_M^{\xi, n}}_{\eta} \cap ]0, \mu]\Big) = \emptyset.
		\end{equation}
		\par \textbf{Now let's show that (\ref{eqn_spec_gap}) holds for $n < 0$ and any $(\xi, h^{\xi})$.} 
		Similarly, we will prove that there are $k \in \nat$, $F \subset L^2(E_M^{\xi, n})$, ${\rm{codim}} \overline F = k$ satisfying (\ref{lem_comp_23}).
		Then, as before, we would get (\ref{eqn_spec_gap}).
		\par Let $\eta_0 > 0$ be chosen such that $g^{TM}$ is induced by (\ref{reqr_poincare}) over $\cup_i V_i^{M}(\eta_0)$, and
		\begin{equation}\label{est_lapl_gap_452}
			\big| [ \imun R^{\xi}, \Lambda^{TM}] \big| \leq 1/4, \qquad \text{over $\cup_i V_i^{M}(\epsilon_0)$,}
		\end{equation}
		where $R^{\xi}$ is the curvature of the Chern connection on $(\xi, h^{\xi})$, and $\Lambda^{TM}$ is the contraction with the Hermitian norm induced by $g^{TM}$.
		Such $\epsilon_0$ exists since as $(\xi, h^{\xi})$ is a Hermitian vector bundle over $\overline{M}$ and $\Lambda^{TM}$ multiplies by a term $O(|z_i^{M} \ln|z_i^{M}| |^2)$ over $V_i^{M}(\epsilon_0)$, which can be made arbitrarily small by replacing $\epsilon_0$ by a smaller number.
		\par Let $\rho : \overline{M} \to [0,1]$ be a smooth cut-off function satisfying 
		\begin{equation}\label{defn_rho_gap}
			\rho(x) = 
			\begin{cases} 
      			\hfill 1 & \text{ for } x \in \cup_i V_i^{M}(\epsilon_0/2), \\
      			\hfill 0  & \text{ for } x \in M \setminus (\cup_i V_i^{M}(\epsilon_0)). \\
 			\end{cases}
		\end{equation}	
		For $s \in \ccal^{\infty}_{c}(M, E_M^{\xi, n})$, we have
		\begin{multline}\label{est_lapl_gap_1}
			\scal{\laplcomp^{E_M^{\xi, n}} s}{s}_{L^2} = 
			\scal{\laplcomp^{E_M^{\xi, n}} (\rho s)}{ \rho s}_{L^2}
			\\
			+ \scal{\laplcomp^{E_M^{\xi, n}} ((1-\rho) s)}{ (1-\rho) s}_{L^2} 
			+ 2 \scal{\laplcomp^{E_M^{\xi, n}} (\rho s)}{ (1-\rho) s}_{L^2}.
		\end{multline}
		By Cauchy inequality, we see that there is $c_1 > 0$ such that for any $\epsilon > 0$, we have
		\begin{equation}\label{est_lapl_gap_2}
			\big| \scal{\laplcomp^{E_M^{\xi, n}} (\rho s)}{ (1-\rho) s}_{L^2} \big| \leq
			\big| \scal{\rho (\laplcomp^{E_M^{\xi, n}} s)}{ (1-\rho) s}_{L^2} \big| 
			\\
			+
			\big| \scal{[\laplcomp^{E_M^{\xi, n}}, \rho]  s}{ (1-\rho) s}_{L^2} \big|,
		\end{equation}
		Since $[\laplcomp^{E_M^{\xi, n}}, \rho]$ is a differential operator of order $1$ with support in a compact subspace of $M$, there is $C > 0$ such that for any $s \in \ccal^{\infty}_{c}(M, E_M^{\xi, n})$, we have
		\begin{equation}\label{eq_bound_order1_sp_gap}
			\big\|[\laplcomp^{E_M^{\xi, n}}, \rho]  s\big\|_{L^2}^{2} \leq C \Big( \big\| \laplcomp^{E_M^{\xi, n}}  s \big\|_{L^2}^{2} + \norm{s}_{L^2}^{2} \Big).
		\end{equation}
		By (\ref{eq_bound_order1_sp_gap}) and Cauchy inequality, there is $c_2 > 0$ such that for any $\epsilon > 0$, we have		
		\begin{equation}\label{est_lapl_gap_23}
		\begin{aligned}
			& \big| \scal{[\laplcomp^{E_M^{\xi, n}}, \rho]  s}{ (1-\rho) s}_{L^2} \big| \leq 
			\epsilon 
			\Big( 
				\big\| \laplcomp^{E_M^{\xi, n}}  s \big\|_{L^2}^{2} + \norm{s}_{L^2}^{2} 
			\Big)
			+ (c_2/\epsilon) \lVert (1-\rho) s \rVert_{L^2}^{2},
			\\
			& \big| \scal{\rho (\laplcomp^{E_M^{\xi, n}} s)}{ (1-\rho) s}_{L^2} \big| \leq  \lVert \laplcomp^{E_M^{\xi, n}} s \rVert_{L^2}^{2} + \lVert (1-\rho) s \rVert_{L^2}^{2}.
		\end{aligned}
		\end{equation}
		Thus, by (\ref{est_lapl_gap_1}), (\ref{est_lapl_gap_2}) and (\ref{est_lapl_gap_23}), we see that
		\begin{multline}\label{est_lapl_gap_3}
			\scal{\laplcomp^{E_M^{\xi, n}} s}{s}_{L^2}  
			+ (3 + 2c_1 / \epsilon) \lVert \laplcomp^{E_M^{\xi, n}} s \rVert_{L^2}^{2} 
			\geq
			 \scal{\laplcomp^{E_M^{\xi, n}} (\rho s)}{ \rho s}_{L^2}
			 \\
			 + 
			 \scal{\laplcomp^{E_M^{\xi, n}} ((1-\rho) s)}{ (1-\rho) s}_{L^2}
			 -
			 4 \epsilon \norm{s}_{L^2}^{2} - (2 + 2 c_2 / \epsilon) \lVert {(1-\rho) s} \rVert_{L^2}^{2}.
		\end{multline}
		Recall that by Nakano's inequality (cf. \cite[Theorem 1.4.14]{MaHol}), we have
		\begin{equation}\label{est_lapl_gap_45}
			 \scal{\laplcomp^{E_M^{\xi, n}} (\rho s)}{ \rho s}_{L^2} \geq
			 \scal{[ \imun R^{E_M^{\xi, n}}, \Lambda^{TM}] (\rho s)}{ \rho s}_{L^2},
		\end{equation}
		where $R^{E_M^{\xi, n}}$ is the curvature of the Chern connection on $E_M^{\xi, n}$. 
		We decompose 
		\begin{equation}\label{est_lapl_gap_45dec}
			R^{E_M^{\xi, n}} = R^{\xi} + n {\rm{Id}}_{\xi} \cdot R^{\omega_M(D)},
		\end{equation}
		where $R^{\omega_M(D)}$ is the curvature of the Chern connection on $(\omega_M(D), \norm{\cdot}_M)$. Now, by (\ref{reqr_poincare}), over $V_i^{M}(\eta_0)$, we have
		\begin{equation}\label{est_lapl_gap_451}
			[ \imun R^{{\omega_M(D)}}, \Lambda^{TM}] = -1/2.
		\end{equation}
		We conclude by (\ref{est_lapl_gap_452}), (\ref{est_lapl_gap_45}), (\ref{est_lapl_gap_45dec}) and (\ref{est_lapl_gap_451}) that for $d := -n/2 - 1/4 > 0$, we have
		\begin{equation}\label{est_lapl_gap_4}
			 \scal{\laplcomp^{E_M^{\xi, n}} (\rho s)}{ \rho s}_{L^2} \geq
			 d \norm{\rho s}_{L^2}^{2}.
		\end{equation}
		\begin{sloppypar}
		As the closure of $M \setminus ( \cup_i V_i^{M}(\epsilon) )$ is a compact manifold with boundary, the Dirichlet problem for $\laplcomp^{E_M^{\xi, n}}$ on $M \setminus ( \cup_i V_i^{M}(\epsilon) )$ has a discrete set of eigenvalues.
		Let $\phi_1, \phi_2, \ldots$ be the eigenvectors  corresponding to the eigenvalues in the increasing order.
		There exists $k \in \nat$ such that for any $s$, satisfying $s \perp (1-\rho) \phi_i$, $i = 1, \ldots, k$, we have
		\end{sloppypar}
		\begin{equation}\label{est_lapl_gap_5}
			\scal{\laplcomp^{E_M^{\xi, n}} ((1-\rho) s)}{ (1-\rho) s}_{L^2} \geq  (2 + d + 2 c_2/\epsilon) \norm{(1-\rho) s}_{L^2}^{2}.
		\end{equation}
		Thus, we conclude from (\ref{est_lapl_gap_3}), (\ref{est_lapl_gap_4}) and (\ref{est_lapl_gap_5}) that for some $c_1, c_2 > 0, k \in \nat$ and for any $\epsilon > 0$ and $s$ satisfying $s \perp (1-\rho) \phi_i$, $i = 1, \ldots, k$, we have
		\begin{equation}\label{est_lapl_gap_6}
			\scal{\laplcomp^{E_M^{\xi, n}} s}{s}_{L^2}  
			+ (2 + 2 c_1 / \epsilon) \lVert \laplcomp^{E_M^{\xi, n}} s \rVert_{L^2}^{2} 
			\geq
			(d/2 - 4 \epsilon) \norm{s}_{L^2}^{2}.
		\end{equation}
		We set $F = \langle (1-\rho) \phi_1, \ldots, (1-\rho) \phi_k \rangle^{\perp}$, where the orthogonal complement is taken with respect to the $L^2$-scalar product. Then we take $\epsilon = d/16$ and deduce (\ref{lem_comp_23}) from (\ref{equal_cauch}) and (\ref{est_lapl_gap_6}).
	\end{proof}
	\par We recall that the function $\rho_M : M \to [1, + \infty[$ was defined in (\ref{defn_rho}).
	To prove Theorems \ref{thm_hk_est} and \ref{thm_rel_hk_infty}, we need the following technical lemma, the proof of which is given in Appendix \ref{app_1}.
	\begin{lem}\label{lem_ell_est}
			For any $\alpha > 0$, $k \in \nat$, there is $C >0$, such that for any $n \in \integ$, $\sigma \in \ccal^{\infty}(M, E_M^{\xi, n})$, $x \in M$, we have
		\begin{equation}\label{lem_ell_est_eqn}
			\textstyle \big| \nabla^k \sigma(x) \big|_{h} \leq C \rho_M(x)
			\sum_{i=0}^{2 + k} (n^{4(2 + 2k - i)} + 1) \norm{(\laplcomp^{E_M^{\xi, n}})^i \sigma}_{L^2(B^M(x, \alpha))}.
		\end{equation}
	\end{lem}
	\begin{rem}
	In the present article we fix $n \in \integ$, so the precise dependence on $n$ of the right-hand side of (\ref{lem_ell_est_eqn}) is not important. However, if one wishes to study the asymptotics of the analytic torsion for non-compact hyperbolic surfaces (similarly to the compact case, considered in \cite{FinI1}) or the asymptotics of the associated Bergman kernel (as in \cite{Auvr}), then the polynomial dependence in $n$ of the coefficients in the right-hand side of (\ref{lem_ell_est_eqn}) is vitally important. Similar estimates were obtained in \cite[Proposition 4.2]{Auvr}.
	\end{rem}
	To prove Theorem \ref{thm_hk_est}, we will need the following
	\begin{lem}\label{lem_from_t0_to_infty}
		Let $f(t)$, $t > 0$ be a semigroup of operators acting on $L^2(E_M^{\xi, n})$ with smooth kernels $f(t,x,y)$, $x, y \in M$ associated with $dv_M(y)$. 
		Suppose that for any $l, l', l'' \in \nat$, there are some $t_0 > 0$, $c', C_1 > 0$, such that for any $t \in ]0, t_0]$, $x, y \in M$, we have
		\begin{equation}\label{eq_lem_t_1}
			\Big| (\nabla_{x})^l (\nabla_{y})^{l'} (\partial_t)^{l''} f(t, x, y) \Big|_{h \times h} 
			\leq 
			C_1 t^{- 1 - (l + l')/2 - l''}   
			 \rho_M(x) \rho_M(y) \exp(- c' \dist(x, y)^2/t).
		\end{equation}
		Then there are $c, C > 0$ such that for any $t > 0$, $x, y \in M$, we have
		\begin{multline}\label{eq_lem_t_2}
			\Big| (\nabla_{x})^l (\nabla_{y})^{l'} (\partial_t)^{l''} f(t, x, y) \Big|_{h \times h} \leq C t^{- 1 - (l + l')/2 - l''}
			\\
			\cdot  \rho_M(x) \rho_M(y) \exp(ct - c' \dist(x, y)^2/t).
		\end{multline}
	\end{lem}
	\begin{proof}
		\begin{sloppypar}
		There are essentially three different cases to consider $x, y \in M \setminus (\cup_i V_i^{M}(1/2))$, $x \in V_i^{M}(1/2), y \in V_j^{M}(1/2)$ for some $i \neq j$ and $x, y \in V_i^{M}(1/2)$ for some $i = 1, \ldots, m$. We will only treat the last one, which is the most difficult one, and we leave the rest to the reader.
		\end{sloppypar} 
		\par 
		We denote $u  = z_i^{M}(x)$, $v =  z_i^{M}(y)$.
		Let's prove by induction that there exists $c, C > 0$ such that for any  $k \in \nat$, $t < 2^k t_0$, we have
		\begin{multline}\label{eq_lem_t_3}
			\Big| (\nabla_{u})^l (\nabla_{v})^{l'} (\partial_t)^{l''} f(t, u, v) \Big|_{h \times h} \leq C t^{- 1 - (l + l')/2 - l''}  (1 + |\ln |u||)^{1/2}(1 + |\ln |v||)^{1/2}
			\\
			\cdot  \exp \Big(c (2^n - n)- c' \cdot \dist(u, v)^2/t \Big).
		\end{multline}
		Now, for $k = 0$, (\ref{eq_lem_t_3}) is simply (\ref{eq_lem_t_1}). 
		Once the induction step is done, (\ref{eq_lem_t_3}) would imply (\ref{eq_lem_t_2}).
		For simplicity, we treat the case $l = l' = l'' = 0$, as the generalization is straightforward.
		\par
		Let $k \in \nat$ and $2^{k-1} t_0 \leq t < 2^k t_0$, then by the semigroup property, we have
		\begin{equation}
			\big| f(2 t, u, v) \big|_{h \times h} 
			\leq 
			\int_{M}
			\big|  f(t, u, z) \big|_{h \times h}  \cdot
			\big|   f(t, z, v) \big|_{h \times h} dv_M(z) 
		\end{equation}
		Without losing the generality, suppose $|u| \leq |v|$.
		We decompose the integration over $M$ into four parts: over $V_i^{M}(|u|)$, over $V_i^{M}(|v|) \setminus V_i^{M}(|u|)$, over $V_i^{M}	(1/2) \setminus V_i^{M}(|v|)$ for $i=1, \ldots, m$, and over $M \setminus (\cup V_i^{M}(1/2))$. 	
		We will suppose that $|u|$ is small enough, as if it is not, then the treatment of all those cases reduces to the last one, which is the easiest one.
		Before treating those cases, let's recall some facts about the geometry of $(\dd^*, g^{T\dd^*})$ and the induced $S^1$-action by rotations.
		First of all, by (\ref{reqr_poincare}), for any $u_0 \in V_i^{M}(1/2)$, the length of the $S^1$-orbit of $u_0$ is given by $2 \pi /|\ln(|u_0|)|$. Thus, by triangle inequality and $S^1$-symmetry, for any $u_1 \in V_i^{M}(1/2)$, we have
		\begin{equation}\label{eq_bound_dist_mhyp}
			\dist(|u_0|, |u_1|) \leq \dist(u_0, u_1)  \leq \dist(|u_0|, |u_1|) + \min \bigg\{ \frac{2 \pi}{|\ln |u_0||} ; \frac{2 \pi}{|\ln |u_1||} \bigg\}.
		\end{equation}
		Also, by a trivial calculation (see \ref{reqr_poincare}), we have
		\begin{equation}\label{dist_hyp_comp}
			\dist(|u_0|, |u_1|) = \Big| \ln |\ln |u_0| | -  \ln |\ln |u_1| | \Big|.
		\end{equation}
		Let $z \in V_i^{M}(1/2)$, by abuse of notation, we denote $z := z_i^M(z)$.
		\par \textbf{Let's treat the integration over $|z| < |u|$.}
		By (\ref{eq_bound_dist_mhyp}), we have
		\begin{equation}\label{dist_hyp_comp_bound}
			\dist(z, v) \geq \dist(|u|, |v|) \geq \dist(u, v) - \frac{2 \pi}{|\ln |u||}.
		\end{equation}
		By (\ref{dist_hyp_comp}) and (\ref{dist_hyp_comp_bound}), since $u$ is small enough, we deduce
		\begin{equation}\label{eq_bound_sq_dist}
			\dist(z, v)^2 \geq \dist(|u|, |v|)^2 \geq \dist(u, v)^2 - 4 \pi.
		\end{equation}
		From the induction hypothesis (\ref{eq_lem_t_3}), (\ref{dist_hyp_comp_bound}) and (\ref{eq_bound_sq_dist}), we deduce
		\begin{multline}\label{eq_lem_t_44}
			\int_{|z| < |u|}
			\big|  f(t, u, z) \big|_{h \times h}  \cdot
			\big|   f(t, z, v) \big|_{h \times h} dv_M(z) 
			\leq 
			2 C^2 t^{- 2}  (1 + |\ln |u||)^{1/2}
			\\
			\cdot
			(1 + |\ln |v||)^{1/2}
			\exp \Big(c (2 \cdot 2^{n-1} - 2(n-1)) + \frac{4 \pi c'}{t} - \frac{c'}{t} \dist(u, v)^2 \Big)
			\\
			\cdot
			\int_{|z| < |u|}		
			\exp \Big(- \frac{c'}{t} \dist(|u|, |z|)^2 \Big) \frac{\imun dz d \overline{z}}{|z|^2 \ln|z|}.
		\end{multline}
		Now, there exists $C_2 > 0$ such that for any $t > 0$, we have
		\begin{equation}\label{eq_lem_t_4444}
			\int_{|z| < |u|}		
			\exp \Big(- \frac{c'}{t} \dist(|u|, |z|)^2 \Big) \frac{\imun dz d \overline{z}}{|z|^2 \ln|z|} = 
			4 \pi \int_{0}^{\infty}
			\exp \Big(- \frac{c'}{t} r^2\Big) dr \leq C_2 \sqrt{t}.
		\end{equation}
		From (\ref{eq_lem_t_44}) and (\ref{eq_lem_t_4444}), we deduce
		\begin{multline}\label{eq_lem_t_44fin}
			\int_{|z| < |u|}
			\big|  f(t, u, z) \big|_{h \times h}  \cdot
			\big|   f(t, z, v) \big|_{h \times h} dv_M(z) 
			\\
			\leq 
			2 C^2 C_2 \exp(4 \pi^2 c' / t) t^{- 3/2}  \big(1 + |\ln |u|| \big)^{1/2} \big(1 + |\ln |v|| \big)^{1/2}
			\\
			\cdot
			\exp \Big(c \big(2 \cdot 2^{n-1} - 2(n-1)\big) - \frac{c'}{t}\dist(u, v)^2 \Big)
		\end{multline}
		Thus, by choosing $c, C$ appropriately, we bound the contribution from the integral over $\{ |z| < |u| \}$ by the right-hand side of (\ref{eq_lem_t_3}).
		\par \textbf{Now let's treat the integral over $|u| < |z| < |v|$.}
		From (\ref{eq_bound_dist_mhyp}) and the boundness of the Gaussian integral, for some $C > 0$, we deduce 
		\begin{multline}\label{eq_lem_t_5}
			\int_{|u| < |z| < |v|}		
			\exp \Big(
			- \frac{c'}{t} \big(\dist(u, z)^2 + \dist(v, z)^2\big) \Big) \frac{\imun dz d \overline{z}}{|z|^2 \ln|z|}
			\\
			\leq
			2 \pi
			\int_{\ln | \ln|u| |}^{\ln |\ln |v| |}
			\exp \Big(
			- \frac{c'}{t}  \Big( \big( y - \ln | \ln|u| |\big)^2 + \big(\ln | \ln|v| | - y\big)^2\Big) \Big) dy
			\\
			=			
			2 \pi
			\int_{0}^{\dist(|u|, |v|)}
			\exp \Big(
			- \frac{c' r^2}{t} -  \frac{c'}{t} \big(\dist(|u|, |v|) - r\big)^2 \Big) dr
			=
			2 \pi
			\exp \Big(
			- \frac{c'}{2t} \dist(|u|, |v|)^2 \Big)
			\\
			\cdot
			\int_{-\dist(|u|, |v|)/2}^{\dist(|u|, |v|)/2}
			\exp \Big(
			- \frac{c' r^2}{2t} \Big) dr
			\leq
			C \sqrt{t}
			\exp \Big(
			- \frac{c'}{2t} \dist(|u|, |v|)^2 \Big)
		\end{multline}
		From the induction hypothesis (\ref{eq_lem_t_3}), (\ref{eq_bound_sq_dist}) and the bounds on $t$, we bound the contribution of the integration over $|u| < |z| < |v|$ by the right-hand side of the induction step (\ref{eq_lem_t_3}).
		\par \textbf{The integral over $|v| < |z| < 1/2$} is treated similarly to the integral over $|z| < |u|$.
		\par \textbf{The integral over $z \in M \setminus \cup_i V_i^{M}(1/2)$} follows from (\ref{eq_triangle_exp}).
	\end{proof}
	\begin{proof}[Proof of Theorem \ref{thm_hk_est}.]
		Let's prove (\ref{thm_est_exp_perp2}) first.
		From Lemma \ref{lem_ell_est}, there is $C > 0$, such that for any $x,x' \in M$, we have
		\begin{multline}\label{eqn_2_9_pf}
			 \big| (\nabla_x)^l (\nabla_{x'})^{l'}  \exp^{\perp}(-t \laplcomp^{E_M^{\xi, n}})(x, x')  \big|_{h \times h} 	 
			   \leq
			 C \rho_M(x) \rho_M(x')
			 \cdot
			 \\
			 \cdot \sum_{i = 0}^{2 + l} \sum_{j = 0}^{2 + l'}
			  \big\lVert (\laplcomp^{E_M^{\xi, n}})^i \exp^{\perp}(-t \laplcomp^{E_M^{\xi, n}}) (\laplcomp^{E_M^{\xi, n}})^j  \big\rVert^{0, 0},	
		\end{multline}
			where $\norm{\cdot}^{0, 0}$ is the operator norm between the corresponding $L^2$ spaces.
			For any $l \in \nat$, $c > 0$, there is $C > 0$ such that for any $t > 0$, we have
			\begin{equation}\label{equal_sup_exp}
				\textstyle \sup_{u \geq c} u^l \exp(-tu) \leq C t^{-l} \exp(-ct/2).
			\end{equation}
			By Theorem \ref{spec_gap_thm}, for any $i, j \in \nat$, there are $c, C > 0$ such that for any $t \geq 0$, we have
			\begin{equation}\label{eq_aux_cicj2}
				\norm{(\laplcomp^{E_M^{\xi, n}})^i \exp^{\perp}(-t \laplcomp^{E_M^{\xi, n}}) (\laplcomp^{E_M^{\xi, n}})^j}^{0, 0} \leq C t^{-(i+j)} \exp(-ct).
			\end{equation}
			From (\ref{eqn_2_9_pf}) and (\ref{eq_aux_cicj2}), we get (\ref{thm_est_exp_perp2}).
			\par \textbf{Let's proceed with a proof of (\ref{thm_est_exp2}).} By Lemma \ref{lem_from_t0_to_infty}, it's enough to prove it for $t < t_0$ for some $t_0 > 0$. We fix $\epsilon > 0$ small enough, and consider several cases.
			\par \textbf{Case 1:} $x,x' \in M \setminus (\cup_i V_i^{M}(\epsilon))$. The estimate (\ref{thm_est_exp2}) for small $t$ is classical and it is proved by using finite propagation speed of solutions of hyperbolic equations (cf. \cite[Theorems D.2.1, 4.2.8]{MaHol}) and the parametrix estimates of the heat kernel similar to \cite[\S 2.4, 2.5]{BGV}.
			\par \textbf{Case 2:} $x \in V_i^{M}(\epsilon)$, $x' \notin V_i^{M}(2 \epsilon)$, for some $i = 1, \ldots, m$. In this case, we prove the estimate (\ref{thm_est_exp2}) for $t$ small enough by using finite propagation speed of solutions of hyperbolic equations.
			\par More precisely, for $r > 0$, we introduce smooth even functions (cf. \cite[(4.2.11)]{MaHol})
		\begin{equation}\label{def_kth_lth}
		\begin{aligned}
			& K_{t, r}(a) = \int_{- \infty}^{+ \infty} \exp(\imun v \sqrt{2t} a) \exp \Big( -\frac{v^2}{2} \Big) \Big(1 - \psi\Big( \frac{\sqrt{2t} v}{r} \Big)\Big) \frac{d v}{\sqrt{2 \pi}}, \\
			& G_{t, r}(a)  = \int_{- \infty}^{+ \infty} \exp(\imun v \sqrt{2t} a) \exp \Big( -\frac{v^2}{2} \Big) \psi\Big( \frac{\sqrt{2t} v}{r} \Big) \frac{d v}{\sqrt{2 \pi}},
		\end{aligned}
		\end{equation}
		where $\psi: \real \to [0,1]$ was defined in (\ref{eq_defn_psi}).
		Let $\widetilde{K}_{t,r}, \widetilde{G}_{t,r} : \real_+ \to \real$ be the smooth functions given by $\widetilde{K}_{t,r}(a^2) = K_{t,r}(a), \widetilde{G}_{t,r}(a^2) = G_{t, r}(a)$. Then the following identities hold
		\begin{equation}\label{exp_scin_kuh}
			\exp(-t \laplcomp^{E_M^{\xi, n}} ) 
			 = 
			\widetilde{G}_{t,r}(\laplcomp^{E_M^{\xi, n}}) + 
			\widetilde{K}_{t,r}(\laplcomp^{E_M^{\xi, n}} ).
		\end{equation}
		\par
		By the finite propagation speed of solutions of hyperbolic equations (cf. \cite[Theorems D.2.1, 4.2.8]{MaHol}), the section $\widetilde{G}_{t,r}(\laplcomp^{E_M^{\xi, n}}) \big(y, \cdot \big)$, $y \in M$, depends only on the restriction of $\laplcomp^{E_M^{\xi, n}}$ onto $B^M(y, r)$, and
		\begin{equation}\label{eq_guh_zero}
			{\rm{supp}}\, \widetilde{G}_{t,r}(\laplcomp^{E_M^{\xi, n}}) \big(y, \cdot) \subset B^M(y, r).
		\end{equation}
		Also, from (\ref{exp_scin_kuh}) and (\ref{eq_guh_zero}), we get
		\begin{equation}\label{eq_exp_kuh_id}
			\exp(-t \laplcomp^{E_M^{\xi, n}} )(y, z) = \widetilde{K}_{t,r}(\laplcomp^{E_N^{n}} )(y, z) \quad \text{if} \quad \dist(y, z) > r.
		\end{equation}
		For any $r_0 > 0$ fixed, from (\ref{def_kth_lth}), there exists $c' > 0$ such that for any $m \in \nat$, there is $C>0$ such that for any $t \in ]0, 1], r > r_0, a\in \real$, the following inequality holds (cf. \cite[(4.2.12)]{MaHol})
		\begin{equation}\label{est_kuh}
			|a|^m | K_{t, r}(a) |  \leq C \exp (- c' r^2/t ).
		\end{equation}
		Thus, by (\ref{est_kuh}), for $t \in ]0,1], r > r_0, a \in \real_+$, we have
		\begin{equation}\label{est_tilde_kuh}
			|a|^m | \widetilde{K}_{t, r}(a) |  \leq C \exp ( -c' r^2/ t ).
		\end{equation}
		Now, by (\ref{est_tilde_kuh}), there exists $c' > 0$ such that for any $k, k' \in \nat$, there is $C > 0$ such that for any $t \in ]0,1]$ and $r > r_0$, we have
		\begin{equation} \label{eqn_kth_norm}
			\norm{(\laplcomp^{E_M^{\xi, n}})^k \widetilde{K}_{t, r}( \laplcomp^{E_M^{\xi, n}} )(\laplcomp^{E_M^{\xi, n}})^{k'} }^{0, 0} 
			\leq
			C \exp ( - c' r^2/ t),
		\end{equation}
		where $\norm{\cdot}^{0, 0}$ is the operator norm between the corresponding $L^2$-spaces. 
		Thus, by Lemma \ref{lem_ell_est}, for any $l, l' \in \nat$, there are $c',C > 0$ such that for any $x, x' \in M$, $r > r_0$, we have
		\begin{equation} \label{eqn_kth_norm_sup}
			\big|  (\nabla_x)^{l} (\nabla_{x'})^{l'} \widetilde{K}_{t, r}( \laplcomp^{E_M^{\xi, n}} ) (x, x') \big|_{h \times h}
			\leq
			C \rho_M(x) \rho_M(x') \exp ( - c' r^2/ t),
		\end{equation}
		We get (\ref{thm_est_exp2}) from (\ref{eq_exp_kuh_id}) and (\ref{eqn_kth_norm_sup}) by taking $r_0 = \frac{1}{4} \dist(V_i^{M}(\epsilon), M \setminus V_i^{M}(2 \epsilon))$ and $r = \frac{1}{2} \dist(x,y)$.
		\par \textbf{Case 3:} $x, x' \in V_i^{M}(2 \epsilon)$ for some $i = 1, \ldots, m$. In this case, we prove the estimate (\ref{thm_est_exp2}) for $t$ small enough by its local analogue  (\ref{eq_bound_exp_sm_t}), and once again by using finite propagation speed of solutions of hyperbolic equations.
		\par 
		We choose a holomorphic trivialization $e_1, \ldots, e_{\rk{\xi}}$ of $\xi$ over $V_i^{M}(\epsilon)$. By the map $\comp^{\rk{\xi}} \to \xi$, given by $(z_1, \ldots, z_{\rk{\xi}}) \mapsto z_1 e_1 + \cdots + z_{\rk{\xi}} e_{\rk{\xi}}$, we induce the Hermitian metric $h^{\xi}_{0}$ on the trivial vector bundle $\xi_0 := \comp^{\rk{\xi}}$ over $D(2 \epsilon)$. 
		By using a bump function, we extend $h^{\xi}_{0}$ to a Hermitian metric on $\comp^{\rk{\xi}}$ over $\dd$, which is trivial away from a compact set, and by abuse of notation denote the resulting Hermitian metric by $h^{\xi}_{0}$. 
		Let $\laplcomp^{E_{\dd}^{n}}$ be the Kodaira Laplacian on $(\dd, g^{T \dd^*})$ associated with $(\xi_0 \otimes \omega_{\dd}(0), h^{\xi}_{0} \otimes (\, \norm{\cdot}_{\dd})^{2n})$.
		\par We denote $u := z_i^{M}(x)$, $u' := z_i^{M}(x')$, and without losing the generality, we put $r := \dist_{\dd^*}(u, 2 \epsilon) < \dist_{\dd^*}(u', 2 \epsilon)$. 
		By (\ref{eq_bound_dist_mhyp}) and (\ref{dist_hyp_comp}), for some $c > 0$, we have
		\begin{equation}\label{eq_dist_cor_comp}
			\dist(u, 2\epsilon) \geq \ln |\ln|u|| - c.
		\end{equation}
		Now, from the fact that the restriction of $\laplcomp^{E_M^{\xi, n}}$ onto $B^M(x, r)$ coincides with the restriction of $\laplcomp^{E_{\dd}^{\xi_0, n}}$ onto $B^{\dd}(u, r)$, and by the finite propagation speed of solutions of hyperbolic equations:
		\begin{equation}\label{eq_guh_equal_disc_rel}
			\widetilde{G}_{t,r}(\laplcomp^{E_M^{\xi, n}}) \big(x, x') = \widetilde{G}_{t,r}(\laplcomp^{E_{\dd}^{\xi_0, n}}) \big(u, u'),
		\end{equation}
		for $E_{\dd}^{n} := \xi_0 \otimes \omega_{\dd}(D)^n$.
		Now, from (\ref{exp_scin_kuh}) and (\ref{eq_guh_equal_disc_rel}), we get
		\begin{equation}\label{exp_scin_kuh_rel}
			\exp(-t \laplcomp^{E_M^{\xi, n}} )(x, x') - \exp(-t \laplcomp^{E_{\dd}^{\xi_0, n}} )(u, u')
			 = 
			\widetilde{K}_{t,r}(\laplcomp^{E_M^{\xi, n}} )(x, x') - \widetilde{K}_{t,r}(\laplcomp^{E_{\dd}^{\xi_0, n}})(u, u').
		\end{equation}
		Now, we conclude by (\ref{eq_bound_exp_sm_t}), (\ref{eqn_kth_norm_sup}), (\ref{eq_dist_cor_comp}) and (\ref{exp_scin_kuh_rel}).
	\end{proof}
	\begin{proof}[Proof of Theorem \ref{thm_rel_hk_infty}.]
		First of all, in the case when $(\xi, h^{\xi})$ is trivial around the cusps, by choosing $\epsilon$ small enough in Case 3 of the proof of Theorem \ref{thm_hk_est}, we see that the Hermitian vector bundle $(\xi_0, h^{\xi}_{0})$ becomes trivial. Thus, (\ref{hk_infty_sharp}) follows from (\ref{eqn_kth_norm_sup}), (\ref{exp_scin_kuh_rel}).
		\par \textbf{Now let's prove the estimates (\ref{hk_infty}), (\ref{hk_infty2}).}
		Consider a family of Hermitian metrics $h^{\xi}_{\epsilon}$, $\epsilon \in [0,1]$ on $\xi$ such that they coincide with $h^{\xi}$ over $M \setminus (\cup_i V_i^{M}(1/2))$ and over $V_i^{M}(1/2)$, we have
		\begin{equation}\label{eq_hxi_flat222}
			h^{\xi}_{\epsilon}((z_i^{M})^{-1}(u))(e_i, e_j) := (1 - \epsilon \psi(4 |u|^2)) h^{\xi}((z_i^{M})^{-1}(u))(e_i, e_j) + \epsilon \psi(4 |u|^2) \delta_{ij},
		\end{equation}
		where $\psi$ is defined in (\ref{eq_defn_psi}), $e_i$, $i = 1,\ldots, \rk{\xi}$ is as in (\ref{defn_normal_triv}), and $\delta_{ij}$ is the Kronecker delta symbol.
		We denote by $\laplcomp^{E_M^{\xi, n}}_{\epsilon}$ the Kodaira Laplacian on $(M, g^{TM})$, associated with $h^{\xi}_{\epsilon}$ and $\norm{\cdot}_M$. Then we have (\ref{eq_comp_e_lapl}) for $\eta := \epsilon$. Moreover, (\ref{eq_bound_cusped_der}) still holds uniformly on $\eta := \epsilon \in [0,1]$.
		By Duhamel's formula (cf. \cite[Theorem 2.48]{BGV}), there exists $\epsilon_0 > 0$ such for any $u \in V_i^M(\epsilon_0)$, we have 
		\begin{multline}\label{eq_duh}
			\partial_{\epsilon} \exp(-t \laplcomp^{E_M^{\xi, n}}_{\epsilon})(u, u) 
			= 
			- 
			\int_0^{t} \int_{v \in M} 
			\exp(-(t-s) \laplcomp^{E_M^{\xi, n}}_{\epsilon})(u, v) \cdot
			\\
			\cdot
			\Big( \partial_{\epsilon} (\laplcomp^{E_M^{\xi, n}}_{\epsilon})_v
			\exp(-s \laplcomp^{E_M^{\xi, n}}_{\epsilon})(v, u) \Big)
			d v_{M}(v) ds.
		\end{multline}
		Now, the operator (\ref{eq_comp_e_lapl}) has support over $V_i^M (1/2)$, thus, the integration in (\ref{eq_duh}) is done only over $V_i^{M}(1/2)$. 
		By coordinate function $z_i^{M}$, we identify $V_i^M (1/2)$ with $D^*(1/2) \subset \dd$.
		Now, since the family of Hermitian metrics (\ref{eq_hxi_flat222}) is smooth, the estimate  (\ref{thm_est_exp2}) holds uniformly in $\epsilon$, and by (\ref{thm_est_exp2}), (\ref{eq_bound_cusped_der}), (\ref{eq_comp_e_lapl}), there is $C > 0$ such that 
		\begin{multline}\label{eq_duh_decomp}
			\Big| \partial_{\epsilon} \exp(-t \laplcomp^{E_M^{\xi, n}}_{\epsilon})(u, u) \Big| 
			\leq
			C (1 + |\ln |u||) \exp(ct)
			\int_0^{t} \int_{v \in D^*(\frac{1}{2})} 
			|v| (1 + |\ln |v||) \frac{1}{t-s} \cdot \\
			\cdot
			\frac{1}{s^{3/2}}
			\cdot \exp\Big(- \frac{\dist(u,v)^2}{4} ( s^{-1} + (t - s)^{-1} ) \Big)
			d v_{\dd^*}(v) ds.
		\end{multline}
		For $r \in \real_+$, we decompose
		\begin{multline}\label{eq_disc_decomp}
			\int_{v \in D^*(\frac{1}{2})} 
			|v| (1 + |\ln |v||) \exp\Big(- \frac{\dist(u,v)^2}{4} ( s^{-1} + (t - s)^{-1} ) \Big) d v_{\dd^*}(v) 
			\\
			=
			\int_{v \in B^{\dd}(u, r) \cap D^*(\frac{1}{2})}  + \int_{v \in D^*(\frac{1}{2}) \setminus B^{\dd}(u, r)}.
		\end{multline}
		By simple geometric considerations, for $\tilde{u} \in \hh$ such that for the covering $\rho$ from Section \ref{sect_pametr}, $\rho(\tilde{u}) = u$, we have 
		\begin{multline}\label{eq_comp_dd_hh_disc}
			\int_{v \in B^{\dd}(u, r)} \exp\Big(- \frac{\dist(u,v)^2}{4} ( s^{-1} + (t - s)^{-1} ) \Big) d v_{\dd^*}(v)  
			\\
			\leq 
			\int_{\tilde{v} \in B^{\hh}(\tilde{u}, r)} \exp\Big(- \frac{\dist(\tilde{u},\tilde{v})^2}{4} ( s^{-1} + (t - s)^{-1} ) \Big) d v_{\hh}(\tilde{v}).
		\end{multline}
		However, since $(\hh, g^{T \hh})$ is isometrically transitive, the right-hand side of (\ref{eq_comp_dd_hh_disc}) doesn't depend on $\tilde{u}$, i.e. it is a function of $r > 0$. Thus, in further estimation of the right-hand side of (\ref{eq_comp_dd_hh_disc}), we may suppose that $\tilde{u} = \imun$.
		\par
		Now let's take $r =1$. Then, since over a compact subset of $\hh$, the metric $g^{T \hh}$ is equivalent to the standard Euclidean metric, by the Gaussian integral on $\comp$, for some $C > 0$, we have
		\begin{equation}\label{eq_comp_dd_hh_disc2}
			\int_{\tilde{v} \in B^{\hh}(\imun, r)} \exp\Big(- \frac{\dist(\imun,\tilde{v})^2}{4} ( s^{-1} + (t - s)^{-1} ) \Big) d v_{\hh}(\tilde{v}) 
			\leq \frac{C}{(s^{-1} + (t - s)^{-1})}.
		\end{equation}
		Now, there is $C > 0$ such that
		\begin{equation}\label{eq_comp_dd_hh_disc3}
		\begin{aligned}
			\int_{v \in D^*(\frac{1}{2}) \setminus B^{\dd}(u, r)} \exp\Big( &- \frac{\dist(u,v)^2}{4} ( s^{-1} + (t - s)^{-1} ) \Big) d v_{\dd^*}(v)  
			\\			
			&
			\leq 
			\int_{v \in D^*(\frac{1}{2})} \exp\Big( -( s^{-1} + (t - s)^{-1} )/4 \Big)  d v_{\dd^*}(v)  
			\\
			& \leq
			C \exp \Big(- ( s^{-1} + (t - s)^{-1} )/4 \Big),
		\end{aligned}
		\end{equation}
		where in the last line we used the fact that the volume of $D^*(1/2)$ is finite.
		By (\ref{eq_disc_decomp}), (\ref{eq_comp_dd_hh_disc}), (\ref{eq_comp_dd_hh_disc2}), (\ref{eq_comp_dd_hh_disc3}), and by the fact that from (\ref{reqr_poincare}) and (\ref{dist_hyp_comp}), for $v \in B^{\dd}(u, 1)$, we have $|v| \leq |u|^{1/e}$, we deduce that there are $c, C > 0$ such that
		\begin{multline}\label{eq_disc_decomp_final}
			\int_{v \in D^*(\frac{1}{2})} 
			|v| (1 + |\ln |v||)  \exp\Big(- \frac{\dist(u,v)^2}{4} ( s^{-1} + (t - s)^{-1} ) \Big) d v_{\dd^*}(v) 
			\\
			\leq
			 \frac{C |u|^{1/e} |\ln |u||}{s^{-1} + (t - s)^{-1}}  + C \exp \Big(- c ( s^{-1} + (t - s)^{-1} ) \Big).
		\end{multline}
		From (\ref{eq_duh_decomp}) and (\ref{eq_disc_decomp_final}), we get (\ref{hk_infty2}).
		\par 
		\textbf{Now let's prove (\ref{hk_infty}).}
		Now let's fix $k \in \nat$ and take $r = \dist(|u|, |\ln |u||^{-k})$.
		By (\ref{dist_hyp_comp}):
		\begin{equation}\label{est_r_discs}
			r = - \int_{|u|}^{|\ln |u||^{-k}} \frac{dr}{r |\ln r|}  \approx \ln | \ln |u| | .
		\end{equation}
		Then by (\ref{eq_comp_dd_hh_disc2}) and (\ref{eq_comp_dd_hh_disc3}),  as $r \geq 1$, for some $c, C > 0$, we have
		\begin{multline}\label{eq_comp_dd_hh_disc2nd}
			\int_{v \in B^{\dd}(u, r) \cap D^*(\frac{1}{2})} \exp\Big(- \frac{\dist(u,v)^2}{4} ( s^{-1} + (t - s)^{-1} ) \Big) d v_{\dd^*}(v)  
			\\
			=
			\int_{v \in B^{\dd}(u, 1) \cap D^*(\frac{1}{2})} +  \int_{v \in B^{\dd}(u, r) \cap D^*(\frac{1}{2}) \setminus B^{\dd}(u, 1)}
			\leq \frac{C}{s^{-1} + (t - s)^{-1}} 
			\\
			+ C \exp \Big(- c ( s^{-1} + (t - s)^{-1} ) \Big).
		\end{multline}
		Also, by (\ref{est_r_discs}) and the fact that the volume of $D^*(\frac{1}{2})$ is finite, there are $c, C  > 0$, such that
		\begin{multline}\label{eq_comp_dd_hh_disc3nd}
			\int_{v \in D^*(\frac{1}{2}) \setminus B^{\dd}(u, r)} \exp\Big(- \frac{\dist(u,v)^2}{4} ( s^{-1} + (t - s)^{-1} ) \Big) d v_{\dd^*}(v)  
			\\
			\leq 
			C \exp \Big(- c (\ln|\ln |u||)^2 ( s^{-1} + (t - s)^{-1} ) \Big).
		\end{multline}
		By (\ref{eq_disc_decomp}), (\ref{eq_comp_dd_hh_disc2nd}) and (\ref{eq_comp_dd_hh_disc3nd}), we have 
		\begin{multline}\label{eq_disc_decomp_final2nd}
			\int_{v \in D^*(\frac{1}{2})} 
			|v| (1 + |\ln |v||) \exp\Big(- \frac{\dist(u,v)^2}{4} ( s^{-1} + (t - s)^{-1} ) \Big) d v_{\dd^*}(v) 
			\\
			\leq
			C \Big( \frac{1}{s^{-1} + (t - s)^{-1}} + \exp \Big(- c ( s^{-1} + (t - s)^{-1} ) \Big) \Big) \Big(1 + \ln |\ln |u||^k \Big)
			\\
			\cdot |\ln |u||^{-k}  
			+ 
			C \exp \Big(- c (\ln|\ln |u||)^2 ( s^{-1} + (t - s)^{-1} ) \Big).
		\end{multline}
		By (\ref{eq_duh_decomp}) and (\ref{eq_disc_decomp_final2nd}), we get (\ref{hk_infty}).
		\par \textbf{Now let's prove the estimate (\ref{hk_infty_perp}).} 
		We have the identity 
		\begin{equation}\label{eqn_exp_perp_exp_rel_norm}
			\exp(-t \laplcomp^{E_M^{\xi, n}}) (x, x') = \exp^{\perp}(-t \laplcomp^{E_M^{\xi, n}}) (x, x') + \sum s_i(x) s_i(x')^*,
		\end{equation}
		where $s_i$ is an orthonormal basis of $H^0(\overline{M}, E_M^{\xi, n})$ with respect to $\langle \cdot, \cdot \rangle_{L^2}$, see (\ref{defn_L_2}).
		From (\ref{eqn_ker_lapl}), (\ref{hk_infty}) and (\ref{eqn_exp_perp_exp_rel_norm}), we conclude that there are $c', C > 0$, such that for any $t > 0$, $u \in D^*(1/2)$:
		\begin{multline}\label{perp_infty_sim}
			\textstyle
			\Big|
				\exp^{\perp}(-t \laplcomp^{E_M^{\xi, n}}) \big( (z_i^{M})^{-1}(u), (z_i^{M})^{-1}(u) \big)
				- {\rm{Id}_{\xi}} \cdot \exp^{\perp}(-t \laplcomp^{E_N^{n}}) \big( (z_i^{N})^{-1}(u), (z_i^{N})^{-1}(u) \big)
			\Big|  \\ 
				\textstyle \leq C \exp(ct) \Big( | \ln |u||  \exp (-  c' (\ln |\ln |u||)^2 /t ) + 1 \Big). 
		\end{multline}
		Also, from (\ref{thm_est_exp_perp2}), there are $c', C > 0$, such that for any $t > 0$, $u \in D^*(1/2)$, we have
		\begin{multline}\label{perp_infty_sim3}
			\Big|
				\exp^{\perp}(-t \laplcomp^{E_M^{\xi, n}}) \big( (z_i^{M})^{-1}(u), (z_i^{M})^{-1}(u) \big)
				- {\rm{Id}_{\xi}} \cdot \exp^{\perp}(-t \laplcomp^{E_N^{n}}) \big( (z_i^{N})^{-1}(u), (z_i^{N})^{-1}(u) \big)
			\Big|  
			\\ 
			\leq C | \ln |u|| t^{-4}  \exp (-  ct). 
		\end{multline}
		By Cauchy inequality, we have 
		\begin{equation}\label{eq_cauchy_exp}
			\exp(-ct - c' (\ln |\ln |u||)^2 /t ) \leq |\ln |u||^{- 2 \sqrt{cc'}}.
		\end{equation}
		We get (\ref{hk_infty_perp}) by multiplying appropriate powers of (\ref{perp_infty_sim}) with (\ref{perp_infty_sim3}) and using (\ref{eq_cauchy_exp}).
	\end{proof}
	
	\begin{proof}[Proof of Theorem \ref{thm_small_time_exp}.]
		By finite propagation speed of solutions of hyperbolic equations and small-time asymptotics of the heat kernel in a compact manifold, we get (\ref{eq_small_time_exp_111}). Moreover, the constant $C$ from (\ref{eq_small_time_exp_111}) could be chosen independently of $x \in M \setminus (\cup_i V_{i}^{M}(\epsilon))$, for some $\epsilon > 0$.
		\par Now let's suppose $x \in V_i^{M}(\epsilon)$, for some $i = 1,\ldots, m$. We note $u = z_i^{M}(x)$, and we use (\ref{exp_scin_kuh_rel}) for $h = \dist_{\dd^*}(u, 2 \epsilon)$. Then by (\ref{eq_local_sm_time}), (\ref{eqn_kth_norm_sup}) and (\ref{exp_scin_kuh_rel}), we see that there are smooth sections $a_{\xi, j}^{M, n}: M \to \enmr{\xi}$, as described, and there is $C > 0$ such that for any $x \in M$, $t \in ]0, t_0]$:
		\begin{multline}\label{eq_small_time_exp_real}
			\Big| \exp(-t \laplcomp^{E_M^{\xi, n}}) \big( x, x \big) - \sum_{j = - 1}^{k} a_{\xi, j}^{M, n}(x) t^j \Big| 
			\leq 
			C \rho_M(x) \bigg( t^k  + \frac{1}{t} \exp \Big(- \frac{c'}{t |\ln |z_i^{M}(x)||^2} \Big)
			\\
			 + \exp\Big(-c (\ln |\ln |z_i^{M}(x)||)^2/t \Big) \bigg),
		\end{multline}
		and for $a_{\xi, j}^{\dd^*, n}$, defined as in Theorem \ref{thm_local_summary}, we have
		\begin{equation}\label{eq_ai_compar}
			a_{\xi, j}^{M, n}(x) = a_{\xi, j}^{\dd^*, n}(z_i^{M}(x)).
		\end{equation}
		From (\ref{eq_small_time_exp_real}), we conclude that if $x \in M \setminus (\cup_i V_i^{M}(e^{-t^{-1/3}}))$, then $C$ in (\ref{eq_local_sm_time}) can be chosen independently of $t \in ]0, t_0]$ and $x$.
		\par The statement (\ref{eq_small_time_coeff_rel}) and the boundness of $a_{\xi, j}^{M, n}(x)$ follows from (\ref{eq_small_time_coeff_rel_local}) and (\ref{eq_ai_compar}).
	\end{proof}

\section{Compact perturbation of the cusp: a proof of Theorem \ref{thm_comp_appr}}\label{sect_compact_pert}
	In this section we will prove Theorem \ref{thm_comp_appr}. 
	The proof consists of two steps. 
	In the first step, Section \ref{sect_spec_case}, we prove that by successive “flattenings" of the Hermitian metric $h^{\xi}$, the associated Quillen norm converges to the Quillen norm associated with $h^{\xi}$. For this, essentially, we use the estimations developed in Section \ref{sect_pametr} along with analytic localization techniques of Bismut-Lebeau \cite[\S 11]{BisLeb91}.
	In the second step, Section \ref{sect_tight}, we restrict ourselves to the case when $(\xi, h^{\xi})$ is trivial near the cusps, and we construct a family of flattenings which “approach" the cusp metric in such a way that the associated analytic torsion converges. 
	In this step we use the analytic localization techniques of Bismut-Lebeau \cite[\S 11]{BisLeb91} along with the maximal principle and some comparison results.
	Finally, as we explain in Section \ref{sect_gen_str_compact}, those two results are enough to give a complete proof of Theorem \ref{thm_comp_appr}.
	Moreover, as we will see along the way, we actually prove Theorem \ref{thm_anomaly_cusp} for $g^{TM}_{0} = g^{TM}$, i.e. for the variation of $h^{\xi}$.
	\subsection{General strategy of a proof of Theorem \ref{thm_comp_appr}}\label{sect_gen_str_compact}
	\begin{sloppypar}
		Let's recall the setting of the problem and describe the main idea of the proof more precisely. We fix surfaces with cusps $(\overline{M}, D_M, g^{TM}), (\overline{N}, D_N, g^{TN})$, a Hermitian vector bundle $(\xi, h^{\xi})$  over  $\overline{M}$ and $n \in \integ$ as in (\ref{data_rel_tors}). 
		We consider a family of Hermitian metrics $h^{\xi}_{\eta}$, $\eta \in ]0, 1/2]$ on $\xi$ constructed in (\ref{eq_hxi_flat}). The main goal of Section \ref{sect_spec_case} is to prove the following formula
		\begin{equation}\label{eq_spec_case_thm_b}
			\lim_{\eta \to 0} \,
			\norm{\cdot}_{Q}  \big(g^{TM}, h^{\xi}_{\eta} \otimes \, \norm{\cdot}_{M}^{2n} \big) 
			=
			\norm{\cdot}_{Q} \big(g^{TM}, h^{\xi} \otimes \, \norm{\cdot}_{M}^{2n}  \big).
		\end{equation}
		As $h^{\xi}_{\eta}|_{D_M} = h^{\xi}|_{D_M}$, we see that (\ref{eq_spec_case_thm_b}) is compatible with Theorem \ref{thm_anomaly_cusp}.
		\par
		In Section \ref{sect_tight} we construct specific families of flattenings $g^{TM}_{\rm{f}, \theta}, \norm{\cdot}_{M}^{\rm{f}, \theta}$, $\theta \in ]0,1]$ such that the corresponding $\nu$ from (\ref{fl_exterior}) tends to $0$, as $\theta \to 0$. We consider the flattenings $g^{TN}_{\rm{f}, \theta}, \norm{\cdot}_{N}^{\rm{f}, \theta}$, which are compatible to $g^{TM}_{\rm{f}, \theta}, \norm{\cdot}_{M}^{\rm{f}, \theta}$, see (\ref{eq_gtm_compat_cusp}), (\ref{eq_normm_compat_cusp}). Then we prove that for any Hermitian metric $h^{\xi}_{2}$ on $\xi$ over $\overline{M}$, for which $(\xi, h^{\xi}_{2})$ is trivial around the cusps we have
		\begin{equation}\label{eq_lim_comp111}
			\lim_{\theta \to 0}
			\frac{\norm{\cdot}_Q (g^{TM}_{\rm{f}, \theta}, h^{\xi}_{2} \otimes (\, \norm{\cdot}_{M}^{\rm{f}, \theta})^{2n})}{\norm{\cdot}_Q (g^{TN}_{\rm{f}, \theta},  (\, \norm{\cdot}_{N}^{\rm{f}, \theta})^{2n})^{\rk{\xi}}}
			=
			\frac{\norm{\cdot}_Q (g^{TM}, h^{\xi}_{2} \otimes \norm{\cdot}_{M}^{2n})}{\norm{\cdot}_Q (g^{TN},  \norm{\cdot}_{N}^{2n})^{\rk{\xi}}}.
		\end{equation}
		This is the most technical and challenging part of this section.
		For $(\xi, h^{\xi})$ trivial, and $n = 0$, this was proved by Jorgenson-Lundelius \cite[Theorem 6.3]{JorLundMain}. Our methods are very different from theirs, in particular we don't study the convergence of small eigenvalues.
		\par
		Now let's explain how (\ref{eq_spec_case_thm_b}) and (\ref{eq_lim_comp111}) imply Theorem \ref{thm_comp_appr}.
		Recall that $\widetilde{\td}$ and $\widetilde{\ch}$ are given by (\ref{ch_bc_0}) and (\ref{ch_bc_2}).
		Let's recall the following theorem of Bismut-Gillet-Soulé \cite[Theorem 1.23]{BGS3}:
	\begin{thm}[Anomaly formula]\label{thm_anomaly_BGS}
		Let $\overline{M}$ be endowed with two (smooth) metrics $g^{T\overline{M}}_{1}, g^{T\overline{M}}_{2}$ over $\overline{M}$. We denote by $\norm{\cdot}_{1}^{\omega}, \norm{\cdot}_{2}^{\omega}$ the Hermitian norms on $\omega_{\overline{M}}$ induced by $g^{T\overline{M}}_{1}, g^{T\overline{M}}_{2}$ over $\overline{M}$.
		Let $\xi$ be a holomorphic vector bundle with Hermitian metrics $h^{\xi}_{1}$, $h^{\xi}_{2}$ over $\overline{M}$.
		We have the following identity
		\begin{multline}\label{eq_anomaly}
			2 \ln \Big( 
			\, \norm{\cdot}_{Q}(g^{T\overline{M}}_2, h^{\xi}_{2}) 
			\big/ 
			\norm{\cdot}_{Q}(g^{T\overline{M}}_1, h^{\xi}_{1}) 
			\Big)
			\\ 
			=
			 	\int_{\overline{M}} 
			 	\Big[ 
			 		\widetilde{\td}(\omega_{\overline{M}}^{-1}, (\, \norm{\cdot}_{1}^{\omega})^{-2}, (\, \norm{\cdot}_{2}^{\omega})^{-2})  \ch(\xi, h^{\xi}_{1})  
					 +			 					 
			 		\td(\omega_{\overline{M}}^{-1}, (\, \norm{\cdot}_{2}^{\omega})^{-2})\widetilde{\ch}(\xi, h^{\xi}_{1}, h^{\xi}_{2})
			 	\Big].
		\end{multline}
	\end{thm}
		Now, by Theorem \ref{thm_anomaly_BGS}, by the fact that the flattenings $g^{TM}_{\rm{f}, \theta}, \norm{\cdot}_{M}^{\rm{f}, \theta}$ and  $g^{TN}_{\rm{f}, \theta}, \norm{\cdot}_{N}^{\rm{f}, \theta}$ are compatible, and by the fact that $(\xi, h^{\xi})$ is trivial around the cusps, we see that the term inside of limit in left-hand side of (\ref{eq_lim_comp111}) doesn't depend on the choice of the flattenings for $\theta$ small enough. Thus, for any $\theta > 0$ such that $(\xi, h^{\xi}_{\eta})$ is trivial over $\cup_i V_i^{M}(\theta)$ (for example, for $\theta^2 < \eta$), by (\ref{eq_lim_comp111}), we have
		\begin{equation}\label{eq_thma_flat_theta}
				\frac{\norm{\cdot}_Q (g^{TM}_{\rm{f}, \theta}, h^{\xi}_{\eta} \otimes (\, \norm{\cdot}_{M}^{\rm{f}, \theta})^{2n})}{\norm{\cdot}_Q (g^{TN}_{\rm{f}, \theta},  (\, \norm{\cdot}_{N}^{\rm{f}, \theta})^{2n})^{\rk{\xi}}}
			=
			\frac{\norm{\cdot}_Q (g^{TM}, h^{\xi}_{\eta} \otimes \norm{\cdot}_{M}^{2n})}{\norm{\cdot}_Q (g^{TN},  \norm{\cdot}_{N}^{2n})^{\rk{\xi}}}.
		\end{equation}
	\end{sloppypar}
	Now, by Theorem \ref{thm_anomaly_BGS}, for any $\theta \in ]0, 1]$, we have
		\begin{multline}\label{eq_anomal_simple}
				2 \ln  \Big(  
				\norm{\cdot}_{Q} \big(g^{TM}_{\rm{f}, \theta}, h^{\xi}_{\eta} \otimes (\, \norm{\cdot}_{M}^{\rm{f}, \theta})^{2n} \big)
				\big/				 
				\norm{\cdot}_{Q} \big(g^{TM}_{\rm{f}, \theta}, h^{\xi}_{2} \otimes (\, \norm{\cdot}_{M}^{\rm{f}, \theta})^{2n} \big)
				 \Big) 
				\\
		 		=	\int_{\overline{M}} \td \big(\omega_M^{-1}, g^{TM}_{\rm{f}, \theta} \big) \widetilde{\ch} \big(\xi, h^{\xi}_{2}, h^{\xi}_{\eta} \big)  \ch \big(\omega_M(D)^n, (\, \norm{\cdot}_{M}^{\rm{f}, \theta})^{2n} \big).
		\end{multline}
		From (\ref{eq_thma_flat_theta}) and (\ref{eq_anomal_simple}), for any $\theta^2 < \eta$, we have
		\begin{multline}\label{eq_final_pf_thma1}
			2 \ln  \Big(  
				\norm{\cdot}_{Q} \big(g^{TM}, h^{\xi}_{\eta} \otimes \, \norm{\cdot}_{M}^{2n} \big)
				/
				\norm{\cdot}_{Q} \big(g^{TM}_{\rm{f}, \theta}, h^{\xi}_{2} \otimes (\, \norm{\cdot}_{M}^{\rm{f}, \theta})^{2n} \big)
			\Big)
			\\
			-
			2 \rk{\xi} \ln  \Big(  
				\norm{\cdot}_{Q} \big(g^{TN}, \norm{\cdot}_{N}^{2n} \big)
				/
				\norm{\cdot}_{Q} \big(g^{TN}_{\rm{f}, \theta}, (\, \norm{\cdot}_{N}^{\rm{f}, \theta})^{2n} \big)
			\Big)
				\\
				=
				\int_{\overline{M}} \td \big(\omega_M^{-1}, g^{TM}_{\rm{f}, \theta} \big) \widetilde{\ch} \big(\xi, h^{\xi}_{2}, h^{\xi}_{\eta} \big)  \ch \big(\omega_M(D)^n, (\, \norm{\cdot}_{M}^{\rm{f}, \theta})^{2n} \big).
		\end{multline}
		Now, by (\ref{eq_der_tilde_tdch}), (\ref{ch_bc_0}) and (\ref{eq_ch_similarity}), we have
		\begin{multline}\label{eq_an_pf_300_aux}
			\int_{\overline{M}}\td \big(\omega_M(D)^{-1}, \norm{\cdot}_M^{-2} \big) \widetilde{\ch} \big(\xi, h^{\xi}_{2}, h^{\xi}_{\eta} \big) 
			\\ 
			=
			\int_{\overline{M}} \td \big(\omega_M^{-1}, g^{TM} \big) \widetilde{\ch} \big(\xi, h^{\xi}_{2}, h^{\xi}_{\eta} \big) 
			- \frac{1}{2} \sum \ln \Big(\det (h^{\xi}_{2} / h^{\xi}_{\eta})|_{P_i^{M}} \Big).
		\end{multline}
		Now, by (\ref{eq_der_tilde_tdch}), we have
		\begin{multline}\label{eq_an_pf_3_aux}
			\int_{\overline{M}} \Big( \td \big(\omega_M^{-1}, g^{TM}_{\rm{f}, \theta} \big) - \td \big(\omega_M^{-1}, g^{TM} \big) \Big) \widetilde{\ch} \big(\xi, h^{\xi}_{2}, h^{\xi}_{\eta} \big) 
			\\ 
			=
			\int_{\overline{M}} \widetilde{\td} \big(\omega_M^{-1}, g^{TM}_{\rm{f}, \theta}, g^{TM} \big) \big( c_1 \big(\xi, h^{\xi}_{2} \big) - c_1 \big(\xi, h^{\xi}_{\eta} \big) \big).
		\end{multline}
		Similarly, by (\ref{eq_der_tilde_tdch}), we have
		\begin{multline}\label{eq_an_pf_4_aux}
			\int_{\overline{M}} \widetilde{\ch} \big(\xi, h^{\xi}_{2}, h^{\xi}_{\eta} \big)  \Big( \ch \big(\omega_M(D)^{n}, (\, \norm{\cdot}_M^{\rm{f}, \theta})^{2n} \big) - \ch \big(\omega_M(D)^{n}, \norm{\cdot}_M^{2n} \big) \Big)
			\\
			=
			\int_{\overline{M}} \big( c_1 \big(\xi, h^{\xi}_{2} \big) - c_1 \big(\xi, h^{\xi}_{\eta} \big) \big)
			\widetilde{\ch} \big(\omega_M(D)^n, (\, \norm{\cdot}_M^{\rm{f}, \theta})^{2n}, \norm{\cdot}_M^{2n} \big).
		\end{multline}
		By (\ref{eq_final_pf_thma1}), (\ref{eq_an_pf_300_aux}), (\ref{eq_an_pf_3_aux}) and (\ref{eq_an_pf_4_aux}), we get
		\begin{equation}\label{eq_final_pf_thma2}
		\begin{aligned}
			2 & \ln \Big(  
				\norm{\cdot}_{Q} \big(g^{TM}, h^{\xi}_{\eta} \otimes \, \norm{\cdot}_{M}^{2n} \big)
				/
				\norm{\cdot}_{Q} \big(g^{TM}_{\rm{f}, \theta}, h^{\xi}_{2} \otimes (\, \norm{\cdot}_{M}^{\rm{f}, \theta})^{2n} \big)
			\Big)
			\\
			& 
			\qquad \qquad \qquad \qquad \qquad \qquad 
			-
			2 \rk{\xi} \ln  \Big(  
				\norm{\cdot}_{Q} \big(g^{TN}, \norm{\cdot}_{N}^{2n} \big)
				/
				\norm{\cdot}_{Q} \big(g^{TN}_{\rm{f}, \theta}, (\, \norm{\cdot}_{N}^{\rm{f}, \theta})^{2n} \big)
			\Big)
				\\
				& =
				\int_{\overline{M}} \Big( \widetilde{\td} \big(\omega_M^{-1}, g^{TM}_{\rm{f}, \theta}, g^{TM} \big) + \widetilde{\ch} \big(\omega_M(D)^n, (\, \norm{\cdot}_M^{\rm{f},  \theta})^{2n}, \norm{\cdot}_M^{2n} \big) \Big) 
				\big( c_1 \big(\xi, h^{\xi}_{2} \big) - c_1 \big(\xi, h^{\xi}_{\eta} \big) \big)
				\\
				& 
				\quad				
				+
				\int_{\overline{M}} \td \big(\omega_M(D)^{-1}, \norm{\cdot}_M^{-2} \big) \widetilde{\ch} \big(\xi, h^{\xi}_{2}, h^{\xi}_{\eta} \big)
				+
				\int_{\overline{M}} \widetilde{\ch} \big(\xi, h^{\xi}_{2}, h^{\xi}_{\eta} \big)  \ch \big(\omega_M(D)^{n}, \norm{\cdot}_M^{2n} \big)
				\\
				& 
				\qquad \qquad \qquad \qquad \qquad \qquad \qquad \qquad \qquad \qquad \qquad 
				+
				\frac{1}{2} \sum \ln \Big(\det (h^{\xi}_{2} / h^{\xi}_{\eta})|_{P_i^{M}} \Big).
		\end{aligned}
		\end{equation}	
		Now we make $\theta \to 0$ in (\ref{eq_final_pf_thma2}). By (\ref{eq_lim_comp111}), the uniform bounds on $g^{TM}_{\rm{f}, \theta}$ and $\norm{\cdot}_M^{\rm{f},  \theta}$ from (\ref{lem_flat_compar}), Lebesgue dominated convergence theorem and the fact that the Todd class appears in the first term of the right hand side of (\ref{eq_final_pf_thma2}) only in degree 0, we deduce
		\begin{equation}\label{eq_final_pf_thma23}
		\begin{aligned}
			2 & \ln \Big(  
				\norm{\cdot}_{Q} \big(g^{TM}, h^{\xi}_{\eta} \otimes \, \norm{\cdot}_{M}^{2n} \big)
				/
				\norm{\cdot}_{Q} \big(g^{TM}, h^{\xi}_{2} \otimes (\, \norm{\cdot}_{M})^{2n} \big)
			\Big)
				\\
				& 
				=
				\int_{\overline{M}} \td \big(\omega_M(D)^{-1}, \norm{\cdot}_M^{-2} \big) \widetilde{\ch} \big(\xi, h^{\xi}_{2}, h^{\xi}_{\eta} \big)
				+
				\int_{\overline{M}} \widetilde{\ch} \big(\xi, h^{\xi}_{2}, h^{\xi}_{\eta} \big)  \ch \big(\omega_M(D)^{n}, \norm{\cdot}_M^{2n} \big)
				\\
				& 
				\qquad \qquad \qquad \qquad \qquad \qquad \qquad \qquad \qquad \qquad \qquad 
				+
				\frac{1}{2} \sum \ln \Big(\det (h^{\xi}_{2} / h^{\xi}_{\eta})|_{P_i^{M}} \Big).
		\end{aligned}
		\end{equation}	
		Now we let $\eta \to 0$. Then by (\ref{eq_spec_case_thm_b}), the fact that the first Chern forms of $(\xi, h^{\xi}_{\eta})$, $\eta \in ]0, 1]$ are uniformly bounded and by Lebesgue dominated convergence theorem, we get Theorem \ref{thm_anomaly_cusp} for $g^{TM}_{0} = g^{TM}$ and $h^{\xi}_{0} := h^{\xi}_{2}$, i.e trivial around the cusps. 
		By applying this result twice for $h^{\xi} := h^{\xi}$, $h^{\xi}_{0} := h^{\xi}_{2}$ and $h^{\xi} := h^{\xi}_{0}$, $h^{\xi}_{0} := h^{\xi}_{2}$, and by taking the difference, we get Theorem \ref{thm_anomaly_cusp} for $g^{TM}_{0} = g^{TM}$ and any $h^{\xi}_{0}$. 
		By this, Theorem \ref{thm_anomaly_BGS}, (\ref{eq_thma_flat_theta}), (\ref{eq_an_pf_3_aux}) and (\ref{eq_an_pf_4_aux}) we deduce Theorem \ref{thm_comp_appr}.
		
	\subsection{Flattening the Hermitian metric: a proof of (\ref{eq_spec_case_thm_b})}\label{sect_spec_case}
		In this section, we reduce Theorem \ref{thm_comp_appr} to the case $(\xi, h^{\xi})$ is trivial near the cusps. For this, we prove (\ref{eq_spec_case_thm_b}). As we explained in Section \ref{sect_compact_pert}, we consider a family of Hermitian metrics $h^{\xi}_{\eta}$, $\eta \in ]0, 1/2]$ on $\xi$ constructed in (\ref{eq_hxi_flat}). We denote by $\laplcomp^{E_M^{\xi, n}}_{\eta}$ the Kodaira Laplacian on $(M, g^{TM})$, associated with $(\xi \otimes \omega_M(D)^n, h^{\xi}_{\eta} \otimes \norm{\cdot}_M^{2n})$. 
		Similarly, for all the geometric objects we considered before, the subscript $\eta$ would mean that instead of $h^{\xi}$, we use $h^{\xi}_{\eta}$.
		\begin{thm}\label{spec_gap_thm_unif_eps}
		For $n \leq 0$, there is $\eta_0 > 0$ such that the operators $\laplcomp^{E_M^{\xi, n}}_{\eta}$, $\eta \in ]0, \eta_0]$ have a uniform spectral gap near $0$, i.e. there is $\mu > 0$ such that for any $\eta \in ]0, \eta_0]$, we have
		\begin{align}
			& H^{0}(\overline{M}, E_M^{\xi, n}) = \ker (\laplcomp^{E_M^{\xi, n}}_{\eta}),
			\\
			& \spec \big( \laplcomp^{E_M^{\xi, n}}_{\eta} \big) \cap \, ]0, \mu] = \emptyset.
		\end{align}
	\end{thm}
		\begin{proof}
		For $n = 0$, the statement of Theorem \ref{spec_gap_thm_unif_eps} is exactly (\ref{spec_gap_epsilon_unif}). For $n < 0$, the proof of Theorem \ref{spec_gap_thm} remains unchanged, since the first Chern form of $(\xi, h^{\xi}_{\eta})$ is bounded, and thus the inequality (\ref{est_lapl_gap_452}) continues to hold.
	\end{proof}
	In this section, we denote by $\nabla$ the connection, induced by the Levi-Civita connection and the Chern connections associated with $(\xi, h^{\xi}_{\eta})$ and $(\omega_{M}(D), \norm{\cdot}_{M})$. We denote by $\dist(\cdot, \cdot)$ the distance function on $(M, g^{TM})$.
	\begin{lem}\label{lem_ell_est_unif_eps}
		For any $l, l' \in \nat$, $n \in \integ$, there are $\eta_0, C > 0$, such that for any $\sigma \in \ccal^{\infty} \big(\overline{M} \times \overline{M}, (E_M^{\xi, n}) \boxtimes (E_M^{\xi, n})^* \big)$, $x, x' \in M$ and any $\eta \in ]0, \eta_0]$, we have
		\begin{equation}\label{lem_ell_est_unif_eps_equn}
			 \big| (\nabla_x)^l (\nabla_{x'})^{l'} \sigma(x, x') \big|_{h \times h} 
			 \leq C \rho_M(x) \rho_M(x') \Big\lVert (\laplcomp^{E_M^{\xi, n}}_{\eta, z})^{i}(\laplcomp^{E_M^{\xi, n}}_{\eta, z'})^{j} \sigma(z, z')\Big\rVert_{L^2, \eta}.
		\end{equation}
	\end{lem}
	\begin{proof}
		It follows from Lemma \ref{lem_ell_est}, (\ref{eq_boud_hxiepsi}) and (\ref{eq_comp_e_lapl}).
	\end{proof}
	\begin{thm}\label{thm_hk_est_unif_eps}
		For any $l,l' \in \nat$, there are $\eta_0, c, c', C > 0$ such that for any $t > 0$, $x, x' \in M$, $\eta \in ]0, \eta_0]$, we have
		\begin{multline}\label{thm_est_exp2_unif_eps} 
			\textstyle \Big| (\nabla_x)^{l} (\nabla_{x'})^{l'} \exp(-t \laplcomp^{E_M^{\xi, n}}_{\eta})(x, x') \Big|_{h \times h} \textstyle \leq C \rho_M(x) \rho_M(x') t^{-1 - (l + l')/2} 
			\cdot \\
			\cdot  \exp \big(ct - c' \cdot \dist(x, x')^2/t\big).
		\end{multline}
		\par Also, if $n \leq 0$, then there are $c, C > 0$ such that for any $t > 0$, $\eta \in ]0, \eta_0]$, we have
		\begin{equation}\label{thm_est_exp_perp2_unif_eps}
			\textstyle \Big| (\nabla_x)^{l} (\nabla_{x'})^{l'} \exp^{\perp}(-t \laplcomp^{E_M^{\xi, n}}_{\eta})(x, x') \Big|_{h \times h}  \textstyle\leq C  \rho_M(x) \rho_M(x') t^{-4 -l - l'} \exp (-ct ). 
		\end{equation}
	\end{thm}
	\begin{proof}
		By Remark \ref{rem_xiesp}, the proof of (\ref{thm_est_exp2}) works uniformly on $\eta$, thus, we get (\ref{thm_est_exp2_unif_eps}).
		Now, (\ref{thm_est_exp_perp2_unif_eps}) follows from Theorem \ref{spec_gap_thm_unif_eps} and Lemma \ref{lem_ell_est_unif_eps}.
	\end{proof}
	\begin{thm}\label{thm_rel_hk_infty_unif_eps}
		For any $k \in \nat$, there are $\eta_0, \epsilon_1, c, c', C > 0$ such that for any $t > 0$, $u \in \comp, |u| \leq \epsilon_1$, $\eta \in ]0, \eta_0]$, $i = 1,\ldots, m$, we have
		\begin{multline}
			 \textstyle
			\Big|
				\Big(
				\exp(-t \laplcomp^{E_M^{\xi, n}}_{\eta})   - 
				\exp(-t \laplcomp^{E_M^{\xi, n}}) 
				\Big)				
				\big( (z_i^{M})^{-1}(u), (z_i^{M})^{-1}(u) \big)
			\Big|   \\  
				\leq
				C | \ln |u| | \exp(ct) \Big(
				| \ln |u| |^{-k} + \exp(-c' (\ln |\ln |u||)^2/t )
				\Big).  \label{hk_infty_unif_eps}
		\end{multline}
		Moreover, if $n \leq 0$, then there are $\varsigma  < 1$ and $c, C > 0$ such that 
		\begin{multline}\label{hk_infty_perp_unif_eps}
			\textstyle
			\Big|
				\Big(
				\exp^{\perp}(-t \laplcomp^{E_M^{\xi, n}}_{\eta}) 	
				- 
				\exp^{\perp}(-t \laplcomp^{E_M^{\xi, n}}) 
				\Big)
				\big( (z_i^{M})^{-1}(u), (z_i^{M})^{-1}(u) \big)   
			\Big|  \\    
				\textstyle \leq C |\ln |u||^{\varsigma} t^{-4}  \exp (-ct). 
		\end{multline}
	\end{thm}
	\begin{proof}
		As the proof of Theorem \ref{thm_rel_hk_infty} is based on (\ref{thm_est_exp2}), which works uniformly on $\eta \in ]0, 1/2]$, the proof of Theorem \ref{thm_rel_hk_infty} also works uniformly on $\eta$, which implies (\ref{hk_infty_unif_eps}). The proof of (\ref{hk_infty_perp_unif_eps}) remains identical to the proof of (\ref{hk_infty_perp}), one only has to use (\ref{thm_est_exp_perp2_unif_eps}) instead of (\ref{hk_infty_perp}).
	\end{proof}
	\begin{thm}\label{thm_small_time_exp_unif_eps}
		There are smooth bounded functions $a_{\xi, \eta, j}^{M, n} : M \to \enmr{\xi}$, $j \geq -1$ such that for any $x \in M$, $t_0 > 0$, $k \in \nat$ there is $C > 0$ such that for any $t \in ]0, t_0]$, $\eta \in ]0,1/2]$, we have
		\begin{equation}\label{eq_small_time_exp_111_unif_eps}
			\Big| \exp(-t \laplcomp^{E_M^{\xi, n}}_{\eta}) \big( x, x \big) - \sum_{j = - 1}^{k} a_{\xi, \eta, j}^{M, n}(x) t^j \Big| \leq C t^k.
		\end{equation}
		Moreover, if $x \in M \setminus (\cup_i V_i^{M}(e^{-t^{-1/3}}))$, then $C$ can be chosen independently of $t \in ]0, t_0]$ and $x$.
		\par 
		Also, there is $\epsilon_1 > 0$, such that for any $l \in \nat$, $j \geq -1$, there is $C > 0$ such that for any $u \in \comp$, $0 < |u| \leq \epsilon_1$, $i = 1, \ldots, m$, $\eta \in ]0,1/2]$, we have
		\begin{equation}\label{eq_small_time_coeff_rel_unif_eps}
			\Big| (\nabla_u)^l \Big( a_{\xi, \eta, j}^{M, n} - a_{\xi, j}^{M, n}  \Big) \big( (z_i^{M})^{-1}(u) \big) \Big|_{h} \leq C |u|^{1/3},
		\end{equation}
		Moreover, for any $x \in M \setminus (\cup_i V_i^{M}(\eta^{1/2}))$, we have
		\begin{equation}\label{eq_identity_aiepsilon}
			a_{\xi, \eta, j}^{M, n}(x) = a_{\xi, j}^{M, n}(x).
		\end{equation}
	\end{thm}
	\begin{proof}
		By Remark \ref{rem_xiesp}, the proof of Theorem \ref{thm_small_time_exp} works uniformly on $\eta$. 
		Thus, only the last statement (\ref{eq_identity_aiepsilon}) needs to be justified, as we don't have its analogue in Theorem \ref{thm_small_time_exp}. But it simply follows from the fact that the coefficients of the small-time expansion of the heat kernel are local and the fact that $h^{\xi}_{\eta}$ coincides with $h^{\xi}$ over $M \setminus ( \cup_i V_i^{M}(\eta^{1/2}))$.
	\end{proof}
	\begin{thm}\label{thm_comp_unif_eps}
		\begin{sloppypar}
		There are $\eta_0, c', C > 0$ such that for any $t > 0, \eta \in ]0, \eta_0]$ and $x \in M \setminus (\cup_i V_i^{M}(|\ln \eta|^{-1}))$, we have
		\begin{equation}
			\label{eqn_thm_comp_3_unif_eps}				
			\Big|
				\big( \exp(-t \laplcomp^{E_M^{\xi, n}}_{\eta}) 
				- 
				\exp(-t \laplcomp^{E_M^{\xi, n}}) 
				\big)(x, x)
			\Big|
			\leq C \rho_{M}(x)^2 \exp(- c' (\ln |\ln \eta|)^2/t).
		\end{equation}
		\end{sloppypar}
	\end{thm}
	\begin{proof}
		\begin{sloppypar}
		We put $r = \dist(V_i^{M}(\eta^{1/2}), M \setminus V_i^{M}(|\ln \eta|^{-1}))$. Then by (\ref{dist_hyp_comp}), we have $r \simeq \ln |\ln \eta|$.
 		Similarly to (\ref{eq_guh_equal_disc_rel}), using (\ref{dist_hyp_comp}) and the fact that $h^{\xi}_{\eta}$ coincides with $h^{\xi}$ over $M \setminus ( \cup_i V_i^{M}(\eta))$, by the finite propagation speed of solutions of hyperbolic equations, there is $c > 0$ such that for any $x \in M \setminus (\cup V_i^{M}(|\ln \eta|^{-1}))$, we have
		\begin{equation}\label{fin_prop_luh_conv}
			\widetilde{G}_{t,r}(\laplcomp^{E_M^{\xi, n}}_{\eta}) ( x, \cdot ) = 
			\widetilde{G}_{t,r}(\laplcomp^{E_M^{\xi, n}}) ( x, \cdot ).
		\end{equation}
		\end{sloppypar}
		\noindent Then, similarly to (\ref{exp_scin_kuh_rel}), we have
		\begin{equation}\label{eqn_inf_kuh_repr_conv}
			\big( \exp(-t \laplcomp^{E_M^{\xi, n}}_{\eta})
			- 
			\exp(-t \laplcomp^{E_M^{\xi, n}})
			\big)			
			(x, x) 
			= 
			\big( \widetilde{K}_{t,r}(\laplcomp^{E_M^{\xi, n}}_{\eta} ) 
			 - 
			\widetilde{K}_{t,r}(\laplcomp^{E_M^{\xi, n}}) 
			\big)
			(x, x).
		\end{equation}
		Now, similarly to (\ref{eqn_kth_norm}), for any $k, k' \in \nat$, there are $c', C > 0$ such that for any $t > 0$, we have
		\begin{equation} \label{eqn_kth_norm_conv}
			\Big\| (\laplcomp^{E_M^{\xi, n}}_{\eta})^k 
			 \big( \widetilde{K}_{t, r}( \laplcomp^{E_M^{\xi, n}}_{\eta} ) \big)
			 (\laplcomp^{E_M^{\xi, n}}_{\eta})^{k'} \Big\|^{(0,0)}
			\leq
			C \exp ( - c' r^2/ t).
		\end{equation}
		From (\ref{lem_ell_est_unif_eps_equn}) and (\ref{eqn_kth_norm_conv}), similarly to (\ref{eqn_kth_norm_sup}), for some $c',C > 0$ and for any $x \in M$, we get
		\begin{equation} \label{eqn_kth_norm_sup_conv}
			\Big|
			\widetilde{K}_{t, r}( \laplcomp^{E_M^{\xi, n}}_{\eta} ) (x, x) 
			\Big|
			\leq
			C \rho_{M}(x)^2 \exp ( - c' r^2/ t).
		\end{equation}
		Now, from (\ref{eqn_kth_norm_sup}), (\ref{eqn_inf_kuh_repr_conv}) and (\ref{eqn_kth_norm_sup_conv}), we get (\ref{eqn_thm_comp_3_unif_eps}).
	\end{proof}
	From all above preparations, we could relate the regularized traces associated with $h^{\xi}_{\eta}$ and $h^{\xi}$.
	\begin{thm}\label{conv_main_thm_unif_eps}
		There are $c, C > 0$, $\varsigma > 0$, $t_0 > 0$, such that for any $t > t_0, \eta \in ]0, e^{-3}]$, we have
		\begin{equation}\label{eq_conv_main_thm_unif_eps}
			\Big|
			{\rm{Tr}}^{\reg} \big[ \exp^{\perp} (-t \laplcomp^{E_M^{\xi, n}}_{\eta}) \big]  - 
			{\rm{Tr}}^{\reg} \big[ \exp^{\perp}(-t \laplcomp^{E_M^{\xi, n}}) \big]
			\Big| 
			\leq C \big( \ln |\ln \eta | \big)^{-\varsigma} \exp ( -ct).
		\end{equation}
	\end{thm}
	\begin{sloppypar}
	\begin{proof}
		First of all, by (\ref{thm_est_exp_perp2_unif_eps}) and (\ref{hk_infty_perp_unif_eps}), in the same way as in Proposition \ref{rel_tr_large_time}, we get 
		\begin{equation}\label{eq_conv_main_thm_unif_eps11}
			\Big|
			{\rm{Tr}}^{\reg} \big[ \exp^{\perp}(-t \laplcomp^{E_M^{\xi, n}}_{\eta}) \big] - 
			{\rm{Tr}}^{\reg} \big[ \exp^{\perp}(-t \laplcomp^{E_M^{\xi, n}}) \big]
			\Big| 
			\leq C \exp ( -ct).
		\end{equation}
		Now, by (\ref{eq_rel_nonperp}), we have
		\begin{multline}\label{eq_conv_main_thm_unif_eps22}
			{\rm{Tr}}^{\reg} \big[ \exp^{\perp}(-t \laplcomp^{E_M^{\xi, n}}_{\eta})\big] - 
			{\rm{Tr}}^{\reg} \big[ \exp^{\perp}(-t \laplcomp^{E_M^{\xi, n}}) \big] 
			\\
			= 
			{\rm{Tr}}^{\reg} \big[ \exp(-t \laplcomp^{E_M^{\xi, n}}_{\eta}) \big] - 
			{\rm{Tr}}^{\reg} \big[ \exp(-t \laplcomp^{E_M^{\xi, n}}) \big].
		\end{multline}
		Trivially, there is $C > 0$ such that for any $\eta \in ]0, e^{-3}]$, we have
		\begin{equation}\label{eq_bound_vol_like}
			\int_{D(1/2) \setminus D(|\ln \eta|^{-1})} \frac{\imun dz d \overline{z}}{|z|^2 |\ln |z||} \leq C \ln \ln |\ln \eta|.
		\end{equation}
		We decompose the integration in the definition of $\tr{\exp(-t \laplcomp^{E_M^{\xi, n}}_{\eta})}$, analogical to Definition \ref{defn_rel_HT}, into two parts: over $\cup_i V_i^{M}(|\ln \eta|^{-1})$ and over $M \setminus (\cup_i V_i^{M}(|\ln \eta|^{-1}))$.
		By bounding the first part of the integral corresponding to the right-hand side of (\ref{eq_conv_main_thm_unif_eps22}) by (\ref{eqn_small_t_as25}), (\ref{hk_infty_unif_eps}) and second part by (\ref{eqn_thm_comp_3_unif_eps}) and (\ref{eq_bound_vol_like}), we see that there are $c,c',C > 0$ such that for any $t > 0$, $\eta \in ]0, e^{-3}]$, we have
		\begin{multline}\label{eq_conv_main_thm_unif_eps33}
			{\rm{Tr}}^{\reg} \big[ \exp(-t \laplcomp^{E_M^{\xi, n}}_{\eta}) \big] - {\rm{Tr}}^{\reg} \big[ \exp(-t \laplcomp^{E_M^{\xi, n}}) \big] 
			\leq 
			\frac{C \exp(ct)}{\ln | \ln \eta |} 
			\\
			+ C(1+t) 
			\ln \ln |\ln \eta| 
			\exp \Big(ct -  \frac{c'}{t} (\ln \ln |\ln \eta|)^2 \Big).
		\end{multline}
		By multiplying (\ref{eq_conv_main_thm_unif_eps11}) and (\ref{eq_conv_main_thm_unif_eps33}) with suitable powers, and using (\ref{eq_cauchy_exp}), (\ref{eq_conv_main_thm_unif_eps22}), we get (\ref{eq_conv_main_thm_unif_eps}).
	\end{proof}
	\end{sloppypar}
	Now, for $j \geq -1$, we denote (compare with (\ref{defn_ajMN}))
		\begin{multline}\label{defn_ajMN_unif_eps}
			A_{\xi, \eta, j}^{M, n} = \int_M {\rm{Tr}} \big[ a_{\xi, \eta, j}^{M,n}(x) \big] dv_M(x) - \frac{\rk{\xi}}{3}  \int_P a_{j}^{P,n} (x) dv_N(x) 
			\\
			-  \dim H^0(\overline{M}, E_M^{\xi, n}) + \frac{\rk{\xi}}{3}  \dim H^0(\overline{P}, E_P^{n}).
		\end{multline}
	This makes sense due to Theorem \ref{thm_small_time_exp} or Remark \ref{rem_integ_small_time}.
	\begin{thm}\label{rel_tr_small_time_unif_eps}
		For any $t_0 > 0$, there is $C > 0$, such that for any $t \in ]0, t_0]$, $\eta \in ]0, e^{-3}]$:
		\begin{multline}\label{rel_tr_small_time_unif_eps0}
			\bigg|
			\int_0^{1}
			\bigg(
			\Big(			
			{\rm{Tr}}^{\reg} \big[ \exp^{\perp}(-t  \laplcomp^{E_M^{\xi, n}}_{\eta})\big] 
			- 
			\sum_{j=-1}^{0} A_{\xi, \eta, j}^{M, n} t^j
			\Big)
			\\
			-
			\Big(			
			{\rm{Tr}}^{\reg} \big[ \exp^{\perp}(-t  \laplcomp^{E_M^{\xi, n}})\big] 
			- 
			\sum_{j=-1}^{0} A_{\xi, j}^{M, n} t^j
			\Big)
			\bigg)
			\frac{dt}{t}
			\bigg|  
			\leq 
			C
			(\ln \ln |\ln \eta|)^{- 1/3}.
		\end{multline}
	\end{thm}
	\begin{proof}
		First of all, by Proposition \ref{rel_tr_small_time} and the analogical statement for $\laplcomp^{E_M^{\xi, n}}_{\eta}$ (which is proved by the same way with one modification: instead of using Theorem \ref{thm_small_time_exp}, we use Theorem \ref{thm_small_time_exp_unif_eps}), we get
		\begin{equation}\label{eq_eps_bound_sm_t0}
			\bigg|
			\Big(			
			{\rm{Tr}}^{\reg} \big[ \exp^{\perp}(-t  \laplcomp^{E_M^{\xi, n}}_{\eta})\big] 
			- 
			\sum_{j=-1}^{0} A_{\xi, \eta, j}^{M, n} t^j
			\Big)
			-
			\Big(			
			{\rm{Tr}}^{\reg} \big[ \exp^{\perp}(-t  \laplcomp^{E_M^{\xi, n}})\big] 
			- 
			\sum_{j=-1}^{0} A_{\xi, j}^{M, n} t^j
			\Big)
			\bigg|  
			\leq 
			C t.
		\end{equation}
		Now, by (\ref{eq_small_time_coeff_rel_unif_eps}) and (\ref{eq_identity_aiepsilon}), there are $\eta_0, C > 0$ such that for any $\eta \in ]0, \eta_0]$, $j = -1, 0$, we have
		\begin{equation}\label{eq_eps_bound_sm_t1}
			\big| A_{\xi, \eta, j}^{M, n} - A_{\xi, j}^{M, n} \big|  \leq C \eta^{1/6}.
		\end{equation}
		Also, by Theorem \ref{thm_comp_unif_eps} and (\ref{eq_bound_vol_like}), there are $\eta_0, c', C > 0$ such that for any $t \in ]0,t_0], \eta \in ]0, \eta_0]$:
		\begin{multline}\label{eq_eps_bound_sm_t2}	
			\int_{M \setminus (\cup_i V_i^{M}(|\ln \eta|^{-1}))}	
			\Big|
				{\rm{Tr}}
				\Big[
				\big( \exp(-t \laplcomp^{E_M^{\xi, n}}_{\eta}) 
				- 
				\exp(-t \laplcomp^{E_M^{\xi, n}}) 
				\big)(x, x)
				\Big]
			\Big|
			\\
			\leq C \ln \ln |\ln \eta| \exp \Big(- \frac{c'}{t} (\ln |\ln \eta|)^2 \Big).
		\end{multline}
		Also, by (\ref{eqn_small_t_as25}) and (\ref{hk_infty_unif_eps}), there are $\eta_0, c', C > 0$ such that for any $t \in ]0,t_0], \eta \in ]0, \eta_0]$, we have
		\begin{multline}\label{eq_eps_bound_sm_t3}
			\int_{V_i^{M}(|\ln \eta|^{-1})}	
			\Big|
				{\rm{Tr}}
				\Big[
				\big( \exp(-t \laplcomp^{E_M^{\xi, n}}_{\eta}) 
				- 
				\exp(-t \laplcomp^{E_M^{\xi, n}}) 
				\big)(x, x)
				\Big]
			\Big|
			\\
			\leq \frac{C}{\ln |\ln \eta| } + C \exp \Big(- \frac{c'}{t} (\ln \ln |\ln \eta|)^2 \Big).	
		\end{multline}
		Thus, by (\ref{eq_eps_bound_sm_t1}), (\ref{eq_eps_bound_sm_t2}) and (\ref{eq_eps_bound_sm_t3}), there are $c', C > 0$ such that for any $t \in ]0,t_0]$, we have
		\begin{multline}\label{eq_eps_bound_sm_t4}
			\bigg|
			\Big(			
			{\rm{Tr}}^{\reg} \big[ \exp^{\perp}(-t  \laplcomp^{E_M^{\xi, n}}_{\eta})\big] 
			- 
			\sum_{j=-1}^{0} A_{\xi, \eta, j}^{M, n} t^j
			\Big)
			-
			\Big(			
			{\rm{Tr}}^{\reg} \big[ \exp^{\perp}(-t  \laplcomp^{E_M^{\xi, n}})\big] 
			- 
			\sum_{j=-1}^{0} A_{\xi, j}^{M, n} t^j
			\Big)
			\bigg| 
			\\ 
			\leq 
			 \frac{C \eta^{1/6}}{t} + \frac{C}{\ln |\ln \eta| } + C \ln \ln |\ln \eta| \exp \Big(- \frac{c'}{t} (\ln \ln |\ln \eta|)^2 \Big).	
		\end{multline}
		Now, by multiplying (\ref{eq_eps_bound_sm_t0}) and (\ref{eq_eps_bound_sm_t4}) with appropriate powers, and integrating on $t$ from $0$ to $1$, we conclude Theorem \ref{rel_tr_small_time_unif_eps}.
	\end{proof}
	\begin{proof}[Proof of (\ref{eq_spec_case_thm_b})]
		By Theorems \ref{conv_main_thm_unif_eps}, \ref{rel_tr_small_time_unif_eps}, (\ref{eq_der_zeta}), (\ref{eq_eps_bound_sm_t1}) and Lebesgue dominated convergence theorem, we have
		\begin{equation}\label{eqn_tors_conv_epsi}
			\lim_{\eta \to 0} T (g^{TM}, h^{\xi}_{\eta} \otimes \norm{\cdot}_{M}^{2n})  = T (g^{TM}, h^{\xi} \otimes \norm{\cdot}_{M}^{2n}).
		\end{equation}
		However, trivially from (\ref{defn_L_2}), we have
			\begin{equation}\label{eqn_l_2_conv_epsi}
				\lim_{\eta \to 0} \, \norm{\cdot}_{L^2}(g^{TM}, h^{\xi}_{\eta} \otimes \, \norm{\cdot}_{M}^{2n}) 
				= 
				\norm{\cdot}_{L^2}(g^{TM}, h^{\xi} \otimes \norm{\cdot}_{M}^{2n}).
		\end{equation}
		By (\ref{eqn_tors_conv_epsi}) and (\ref{eqn_l_2_conv_epsi}), we get (\ref{eq_spec_case_thm_b}).
	\end{proof}

	\subsection{Flattening the Riemannian metric: a proof of (\ref{eq_lim_comp111})}\label{sect_tight}
	In this section we introduce the notion of a \textit{tight family of flattenings} and study some of its properties.
	Those families of flattenings are useful for studying the convergence of the spectrum for a family of surfaces “approaching" the cusped surface. 
	We study how the relative heat trace converges as this family of flattenings “converges" to the cusped metric, and from this study we deduce (\ref{eq_lim_comp111}).
	From now and till the end of Section \ref{sect_compact_pert}, we suppose that $(\xi, h^{\xi})$ is trivial near the cusps.
	\begin{defn}\label{defn_tight}
		We say that the flattenings $g^{TM}_{\rm{f}, \theta}$, $\norm{\cdot}_{M}^{\rm{f}, \theta}$, $\theta \in ]0,1]$ (cf. Definition \ref{defn_flat_metr}) of $g^{TM}$, $\norm{\cdot}_{M}$ are $n$-\textit{tight}, $n \in \integ$ if they satisfy the following requirements:
	\par \textbf{1.} We have $g^{TM}_{\rm{f}, \theta}|_{M \setminus (\cup_i V_i^{M}(\theta))} = g^{TM}|_{M \setminus (\cup_i V_i^{M}(\theta))}$ and similarly for $\norm{\cdot}_{M}^{\rm{f}, \theta}$.
	\par \textbf{2.} For all $i = 1, \ldots, m$, the following identity holds over $V_i^{M}(\theta^2)$ : $ \norm{ d z_i^{M} \otimes s_{D_M} / z_i^{M}}_{M}^{\rm{f}, \theta} = |\ln \theta|$, where $s_{D_M}$ is the canonical section of $\mathscr{O}_M(D_M)$.
	\par \textbf{3.} There are flattenings $g^{TM}_{\rm{sm}}$, $\norm{\cdot}_{M}^{\rm{sm}}$  of $g^{TM}$, $\norm{\cdot}_{M}$ such that for any $\theta \in ]0, e^{-3}]$, we have
		\begin{equation}\label{lem_flat_compar}
			\begin{aligned}
			& g^{TM}_{\rm{sm}} \otimes ( \, \norm{\cdot}_{M}^{\rm{sm}})^{2n} \leq g^{TM}_{\rm{f}, \theta} \otimes ( \, \norm{\cdot}_{M}^{\rm{f}, \theta})^{2n} \leq g^{TM} \otimes \, \norm{\cdot}_{M}^{2n}, \\
			& \norm{\cdot}_{M}^{\rm{sm}} \leq \norm{\cdot}_{M}^{\rm{f}, \theta} \leq \norm{\cdot}_M.
		\end{aligned}
		\end{equation}
	\par \textbf{4.} We have the following analogue of Lemma \ref{lem_ell_est}: \textit{there is $C > 0$ such that for any $\sigma \in \ccal^{\infty} \big(\overline{M} \times \overline{M}, E_M^{\xi, n} \boxtimes (E_M^{\xi, n})^* \big)$, $\theta \in ]0, e^{-3}]$, and for any $x, x' \in \overline{M}$ satisfying either $x, x' \in M \setminus (\cup_i V_i^{M}(\theta^3))$, or $x  = x'$, we have 
		\begin{equation}\label{lem_ell_est_eqn_A}
			\big| \sigma(x, x') \big|_{h \times h, \theta} \leq C \rho_{M, \theta}(x) \rho_{M, \theta}(x')
				 \sum_{i,j = 0}^{2} \big\lVert (\laplcomp^{E_M^{\xi, n}}_{\rm{f}, \theta, z})^{i}(\laplcomp^{E_M^{\xi, n}}_{\rm{f}, \theta, z'})^{j} \sigma(z, z')\big\lVert_{L^2, \theta},
		\end{equation}
		where $\laplcomp^{E_M^{\xi, n}}_{\rm{f}, \theta}$ is the Laplacian associated with  $g^{TM}_{\rm{f}, \theta}, h^{\xi}$ and $\norm{\cdot}_{M}^{\rm{f}, \theta}$; $| \cdot |_{h \times h, \theta}$ is the pointwise norm induced by $h^{\xi}$ and $\norm{\cdot}_{M}^{\rm{f}, \theta}$;  $\norm{\cdot}_{L^2, \theta}$ is the $L^2$ norm induced by $g^{TM}_{\rm{f}, \theta}, h^{\xi}$, $\norm{\cdot}_{M}^{\rm{f}, \theta}$; and the function $\rho_{M, \theta} : \overline{M} \to [1, \infty[$ is given by 
		\begin{equation}\label{def_rho_theta}
		\rho_{M, \theta}(x) =
		\begin{cases} 
      		\hfill 1 & \text{ for } x \in M \setminus (\cup_i V_i(1/2)), \\
      		\hfill \sqrt{| \ln |z_i^{M}(x)||}  & \text{ for }  x \in V_i^{M}(1/2) \setminus V_i^{M}(\theta^3), \\
      		\hfill ( \ln \theta )^6  & \text{ for }  x \in  V_i^{M}(\theta^3).
 		\end{cases}
		\end{equation}}
	\end{defn}
	\par In Appendix \ref{app_2} we show that for any $n \in \integ$, $n$-tight families of flattenings exist.
	\par We fix $n \in \integ, n \leq 0$ and $n$-tight families of flattenings $g^{TM}_{\rm{f}, \theta}$, $\norm{\cdot}_{M}^{\rm{f}, \theta}$, $\theta \in ]0,1]$.
	From (\ref{lem_flat_compar}):
	\begin{equation}\label{vol_compar}
		g^{TM}_{\rm{f}, \theta} \leq g^{TM}.
	\end{equation}
	\par Recall that $\laplcomp^{E_M^{\xi, n}}$ is the Kodaira Laplacian associated to $g^{TM}$, $h^{\xi}$ and $\norm{\cdot}_M$.
	We set $\mu > 0$ as in (\ref{eqn_spec_gap}).
	We defer the proof of the following theorem until Section \ref{sect_pf_aux_tight}.
	\begin{thm}\label{spec_gap_thm_A}
		The operator $\laplcomp^{E_M^{\xi, n}}_{\rm{f}, \theta}$ has a uniform spectral gap near $0$, i.e. for any $\theta \in ]0, e^{-3}]$:
		\begin{align}
			\label{eqn_ker_lapl_A}
			& \ker (\laplcomp^{E_M^{\xi, n}}_{\rm{f}, \theta}) \simeq H^{0}(\overline{M}, E_M^{\xi, n}),
			\\
			\label{eqn_spec_gap_A}
			& \spec(\laplcomp^{E_M^{\xi, n}}_{\rm{f}, \theta}) \cap \, ]0, \mu[ = \emptyset.
		\end{align}
	\end{thm}
	\begin{sloppypar}
	In what follows, we denote the smooth kernels of $\exp(-t \laplcomp^{E_M^{\xi, n}}_{\rm{f}, \theta}), \exp^{\perp}(-t \laplcomp^{E_M^{\xi, n}}_{\rm{f}, \theta})$ with respect to the Riemannian volume form $d \vol_{M, \theta}$ induced by $g^{TM}_{\rm{f}, \theta}$ by
	\end{sloppypar}
	\begin{equation}
		\exp(-t \laplcomp^{E_M^{\xi, n}}_{\rm{f}, \theta})(x, y),
		\exp^{\perp}(-t \laplcomp^{E_M^{\xi, n}}_{\rm{f}, \theta})(x, y) 
		\in (E_M^{\xi, n})^{*}_{x} \boxtimes (E_M^{\xi, n})_{y},
		 \quad \text{for} \quad x,y \in M.
	\end{equation}
	\begin{thm}\label{thm_hk_estA}
		There are $c, C > 0$ such that for any $t > 0$, $x \in M$, $\theta \in ]0, e^{-3}]$, we have
		\begin{equation}\label{thm_est_exp_perp2A}
			\textstyle 
			\Big|
				\exp^{\perp}(-t \laplcomp^{E_M^{\xi, n}}_{\rm{f}, \theta})(x, x)
			\Big|  \textstyle\leq C  \rho_{M, \theta}(x)^2 t^{-4} \exp (- ct ). 
		\end{equation}
	\end{thm}
	\begin{proof}
		The proof is the same as the proof of Theorem \ref{thm_hk_est}, one only has to change the use of Lemma \ref{lem_ell_est} by  (\ref{lem_ell_est_eqn_A}) and of Theorem \ref{spec_gap_thm} by Theorem \ref{spec_gap_thm_A}.
	\end{proof}
	\begin{thm}\label{thm_conv}
	\begin{sloppypar}
		There are $c', C > 0$ such that for any $t > 0, \theta \in ]0, e^{-3}]$ and $x \in M \setminus (\cup_i V_i^{M}(|\ln \theta|^{-1}))$, we have
		\begin{equation}
			\label{eqn_thm_conv_3}				
			\Big|
				\big( \exp(-t \laplcomp^{E_M^{\xi, n}}_{\rm{f}, \theta}) 
				- 
				\exp(-t \laplcomp^{E_M^{\xi, n}}) 
				\big)(x, x)
			\Big|
			\leq C \rho_{M, \theta}(x)^2 \exp(- c' (\ln |\ln \theta|)^2/t).
		\end{equation}
	\end{sloppypar}
	\end{thm}		
	\begin{proof}
			The proof is the same as the proof of Theorem \ref{thm_comp_unif_eps}, we leave the details to the reader.
	\end{proof}
	We construct the flattenings $g^{TN}_{\rm{f}, \theta}$, $\norm{\cdot}_{N}^{\rm{f}, \theta}$, $\theta \in ]0,1]$ of $g^{TN}$, $\norm{\cdot}_{N}$, which are compatible with $g^{TM}_{\rm{f}, \theta}$, $\norm{\cdot}_{M}^{\rm{f}, \theta}$. 
	Trivially, the flattenings $g^{TN}_{\rm{f}, \theta}$, $\norm{\cdot}_{N}^{\rm{f}, \theta}$, $\theta \in ]0,1]$  are $n$-tight.	
	The following theorem is an analogue of Theorem \ref{thm_rel_hk_infty}, and it forms the core of this section. 
	Its proof is defered to Section \ref{sect_pf_aux_tight}.	
	\begin{thm}\label{thm_A_bound_cups}
		There are $c, c', C > 0$, $\varsigma < 1$ such that for any $t > 0, \theta \in ]0, e^{-3}], i = 1, \ldots, m$, and $u \in \comp, |u| \leq |\ln \theta|^{-1}$, we have
		\begin{multline}\label{eqn_A_bound_cups_perp_1}
			\Big|
				\exp(-t \laplcomp^{E_M^{\xi, n}}_{\rm{f}, \theta}) \big((z_i^{M})^{-1}(u), (z_i^{M})^{-1}(u) \big) 
				- 
				{\rm{Id}}_{\xi} \exp(-t \laplcomp^{E_N^{n}}_{\rm{f}, \theta})\big((z_i^{N})^{-1}(u), (z_i^{N})^{-1}(u)\big)
			\Big| 
			\\
			\leq C | \ln \max(\theta, |u|)| \cdot  \exp( - c' (\ln |\ln  \max(\theta, |u|)|)^2/t).	
		\end{multline}
		\vspace*{-0.9cm}
		\begin{multline}\label{eqn_A_bound_cups_perp_2}
			\Big|
				\exp^{\perp}(-t \laplcomp^{E_M^{\xi, n}}_{\rm{f}, \theta}) \big((z_i^{M})^{-1}(u), (z_i^{M})^{-1}(u) \big) 
				- 
				{\rm{Id}}_{\xi} \exp^{\perp}(-t \laplcomp^{E_N^{n}}_{\rm{f}, \theta})\big((z_i^{N})^{-1}(u), (z_i^{N})^{-1}(u)\big)
			\Big| 
			\\
			\leq C | \ln  \max(\theta, |u|)|^{\varsigma} \cdot t^{-4} \exp( - c t).	
		\end{multline}
	\end{thm}
	Now, for brevity, we denote
	\begin{multline}\label{eq_x_theta_defn}
		X_{\theta}(t) :=
		\tr{\exp^{\perp}(-t \laplcomp^{E_M^{\xi, n}}_{\rm{f}, \theta})} 
		- 
		\rk{\xi} \tr{\exp^{\perp}(-t \laplcomp^{E_N^{n}}_{\rm{f}, \theta})}
		\\			
		- 
		{\rm{Tr}}^{\reg} \big[ \exp^{\perp}  (-t  \laplcomp^{E_M^{\xi, n}})\big] 
		+ 
		\rk{\xi}  {\rm{Tr}}^{\reg} \big[ \exp^{\perp}(-t  \laplcomp^{E_N^{n}}) \big]
	\end{multline}
	Let's use those theorems to study the convergence of heat traces. The main theorem here is 
	\begin{thm}\label{conv_main_thm}
		There are $c, c', C > 0$ such that for any $t > 0, \theta \in ]0, e^{-3}]$, we have
		\begin{equation}\label{eq_conv_main_thm}
			\big|
			X_{\theta}(t)
			\big| 
			\leq C \exp ( -ct - c' (\ln |\ln \theta| )^2/t ).
		\end{equation}
	\end{thm}
	\begin{proof}
		Let's denote
		\begin{align}
				& \nonumber
				A_M^{\perp}(t) = \int_{M \setminus (\cup_i V_i^{M}(|\ln \theta|^{-1}))} 
				\Big(
				\tr{\exp^{\perp}(-t \laplcomp^{E_M^{\xi, n}}_{\rm{f}, \theta})(x,x)} - 
				\tr{\exp^{\perp}(-t \laplcomp^{E_M^{\xi, n}})(x,x)} 
				\Big) 
				d \vol_{M, \theta}(x), 
				\\
				& \nonumber
				A_N^{\perp}(t) = \int_{N \setminus (\cup_i V_i^{N}(|\ln \theta|^{-1}))} 
				\Big(
				\tr{\exp^{\perp}(-t \laplcomp^{E_N^{n}}_{\rm{f}, \theta})(x,x)} - 
				\tr{\exp^{\perp}(-t \laplcomp^{E_N^{n}})(x,x)} 
				\Big)				
				d \vol_{N, \theta}(x), 
				\\
				& 
				B_{\theta}^{\perp}(t) = \sum_i 
				\int_{D(|\ln \theta|^{-1})}
				\Big(
				{\rm{Tr}} \Big[ \exp^{\perp}(-t \laplcomp^{E_M^{\xi, n}}_{\rm{f}, \theta})\big((z_i^{M})^{-1}(u), (z_i^{M})^{-1}(u)\big) \Big] 
				\label{bound_appr_4}
				\\
				& \nonumber  \qquad  \qquad \qquad \qquad \qquad  - 
				\rk{\xi} 
				{\rm{Tr}} \Big[ \exp^{\perp}(-t \laplcomp^{E_N^{n}}_{\rm{f}, \theta}) \big((z_i^{N})^{-1}(u), (z_i^{N})^{-1}(u)\big) \Big] 
				\Big)				
				d \vol_{\theta}(u),
								\\
				& \nonumber
				B^{\perp}(t) = \sum_i 
				\int_{D(|\ln \theta|^{-1})}
				\Big(
				{\rm{Tr}} \Big[  \exp^{\perp}(-t \laplcomp^{E_M^{\xi, n}})\big((z_i^{M})^{-1}(u), (z_i^{M})^{-1}(u)\big) \Big] 
				\\ \nonumber
				& \qquad  \qquad \qquad \qquad \qquad  - 
				\rk{\xi} 
				{\rm{Tr}} \Big[  \exp^{\perp}(-t \laplcomp^{E_N^{n}}) \big((z_i^{N})^{-1}(u), (z_i^{N})^{-1}(u) \big) \Big]
				\Big)				
				d \vol_{\dd^*}(u),
		\end{align}
		and $d \vol_{N, \theta}, d \vol_{\theta}$ are the Riemannian volume forms induced by $g^{TN}_{\rm{f}, \theta}$ and $((z_i^{M})^{-1})^* g^{TM}_{\rm{f}, \theta}$ correspondingly. 
		Then we have
		\begin{equation}\label{bound_appr_1}
			X_{\theta}(t)
			=
			A_M^{\perp}(t) + A_N^{\perp}(t) + B_{\theta}^{\perp}(t) + B^{\perp}(t).
		\end{equation}
		\par
		By (\ref{reqr_poincare}), (\ref{def_rho_theta}) and (\ref{vol_compar}), there is $C > 0$ such that for any $\theta \in ]0, e^{-3}]$, we have
		\begin{equation}\label{bound_appr_231}
			\int_{M \setminus (\cup_i V_i^{M}(|\ln \theta|^{-1}))} \rho_{M, \theta}(x)^{2} d \vol_{M, \theta}(x) \leq C (\ln \ln |\ln \theta| ).
		\end{equation}
		By Theorem \ref{thm_hk_estA}, (\ref{thm_est_exp_perp2}) and (\ref{bound_appr_231}), there are $c,C  > 0$ such that
		\begin{equation}\label{bound_appr_23}
			| A_M^{\perp}(t) |, | A_N^{\perp}(t) | \leq C ( \ln \ln |\ln \theta| ) t^{- 4} \exp(-ct).
		\end{equation}
		By (\ref{reqr_poincare}), (\ref{eq_int_mu_fin}) and (\ref{vol_compar}), for any $\varsigma < 1$, there is $C>0$ such that for any $\theta \in ]0, e^{-3}]$, we have
		\begin{equation}\label{bound_appr_251}
			\int_{V_i^{M}(|\ln \theta|^{-1})} \big| \ln \max(\theta, |z_i^{M}(x)|) \big|^{\varsigma} d \vol_{M, \theta}(x) \leq C.
		\end{equation}
		By (\ref{eqn_A_bound_cups_perp_2}) and (\ref{bound_appr_251}), there are $c,C  > 0$ such that
		\begin{equation}\label{bound_appr_253}
			| B_{\theta}^{\perp}(t) | \leq C t^{-4} \exp(-ct).
		\end{equation}
		By (\ref{hk_infty_perp}) and (\ref{eq_int_mu_fin}), there are $c,C  > 0$ such that
		\begin{equation}\label{bound_appr_25}
			| B^{\perp}(t) | \leq C t^{- 4} \exp(-ct).
		\end{equation}
		 By (\ref{bound_appr_23}), (\ref{bound_appr_253}) and (\ref{bound_appr_25}), for some $c,C  > 0$, and for any $t > 0$, $\theta \in ]0, e^{-3}]$, we have
		\begin{equation}\label{bound_appr_fin_1}
			\big|
				X_{\theta}(t)
			\big| 
			\\
			\leq C(1 + t^{- 4}) (  \ln \ln |\ln \theta| ) \exp ( -ct ).
		\end{equation}
		\par Now, alternatively, we also write
		\begin{equation}\label{bound_appr_3}
			X_{\theta}(t)
			=
			A_M(t) + A_N(t) + B_{\theta}(t) + B(t),	
		\end{equation}
		where $A_M(t)$, $A_N(t)$, $B_{\theta}(t)$, $B(t)$ are given by the same formulas as in (\ref{bound_appr_4}), where we put $\exp$ in place of $\exp^{\perp}$.
		\par 
		By Theorem \ref{thm_conv} and (\ref{bound_appr_231}), there are $c', C > 0$ such that for any $\theta \in ]0, e^{-3}]$, we have
		\begin{equation}\label{bound_appr_5}
			\begin{aligned}
				& | A_M(t) |, | A_N(t) |  \leq C (\ln \ln |\ln \theta|) \exp ( - c' ( \ln | \ln \theta|)^2/t ).
			\end{aligned}
		\end{equation}
		By (\ref{vol_compar}), we have
		\begin{equation}\label{bound_appr_51}
			\int_{V_i^{M}(|\ln \theta|^{-1})} |\ln \max(\theta, |u|)|  d \vol_{M, \theta}(x) \leq C \ln |\ln \theta|.
		\end{equation}
		By (\ref{eqn_A_bound_cups_perp_1}) and (\ref{bound_appr_51}), for some $c', C > 0$, we have
		\begin{equation}\label{bound_appr_52}
			\begin{aligned}
				& | B_{\theta}(t) |  \leq C \ln |\ln \theta| \exp ( - c' (\ln \ln |\ln \theta|)^2/t ).
			\end{aligned}
		\end{equation}
		By (\ref{hk_infty}) and (\ref{eqn_small_t_as25}), for some $c', C > 0$, we have
		\begin{equation}\label{bound_appr_55}
			| B(t) | \leq C (1 + t)  \exp ( - c' ( \ln \ln |\ln \theta|)^2/t ).
		\end{equation}
		By (\ref{bound_appr_5}), (\ref{bound_appr_52}) and (\ref{bound_appr_55}), we conclude that 
		\begin{equation}\label{bound_appr_fin_2}
			\big|
				X_{\theta}(t)
			\big| 
			\leq 
			C (1 + t) ( \ln |\ln \theta|) \exp ( - c' (\ln \ln | \ln \theta| )^2/t ).
		\end{equation}
		By multiplying (\ref{bound_appr_fin_1}) with power $1 - \mu \in ]1/2, 1]$ and (\ref{bound_appr_fin_2}) with power $\mu$, for some $c, c', C > 0$:
		\begin{equation}\label{bound_appr_fin_245}
			\big|
				X_{\theta}(t)
			\big| 
			\\
			\leq 
			C (1 + t)(1 + t^{-4}) (\ln | \ln \theta|)^{\mu} \big( \ln \ln | \ln \theta| \big) \exp ( -ct - c' \mu ( \ln \ln |\ln \theta |)^2/t ).
		\end{equation}
		By (\ref{eq_cauchy_exp}) and (\ref{bound_appr_fin_245}), we deduce (\ref{eq_conv_main_thm}) by taking $\mu$ small enough.
	\end{proof}
	
	For $s \in \comp$, $\Re(s) > 1$, let's denote the approximated regularized zeta-function by
	\begin{equation}\label{defn_zeta_appr100}
		\zeta_{M}^{\theta}(s) = 
		\frac{1}{\Gamma(s)} \int_{0}^{+ \infty}
			{\rm{Tr}} \Big[ \exp^{\perp}(-t \laplcomp^{E_M^{\xi, n}}_{\rm{f}, \theta}) \Big]
			t^{s} \frac{dt}{t}.
	\end{equation}
	As usually, $\zeta_{M}^{\theta}(s)$ has a meromorphic extension to the entire $s$-plane, and this extension is holomorphic at $0$.
	We recall that the zeta-function $\zeta_{M}$ was defined in Definition \ref{defn_rel_zeta}.
	\begin{prop}\label{zeta_conv}
		For any $\theta \in ]0, e^{-3}]$, the difference $\zeta_{M}^{\theta}(s) - \rk{\xi} \zeta_{N}^{\theta}(s) - \zeta_{M}(s) + \rk{\xi} \zeta_{N}(s)$ is a holomorphic function on $\comp$. Moreover, as $\theta \to 0$, we have
		\begin{equation}\label{eq_zeta_conv}
			\zeta_{M}^{\theta}(s) - \rk{\xi} \zeta_{N}^{\theta}(s) - \zeta_{M}(s) + \rk{\xi} \zeta_{N}(s) \to 0,
		\end{equation}
		uniformly for $s$ varying in a compact subset of $\comp$.
		In particular, we have, as $\theta \to 0$,
		\begin{equation}\label{cor_tors_conv_eqn}
			\frac{T(g^{TM}_{\rm{f}, \theta}, h^{\xi} \otimes (\, \norm{\cdot}_{M}^{\rm{f}, \theta})^{2n})}{T(g^{TN}_{\rm{f}, \theta}, (\, \norm{\cdot}_{N}^{\rm{f}, \theta})^{2n})^{\rk{\xi}}}
			\to			
			\frac{T(g^{TM}, h^{\xi} \otimes \norm{\cdot}_{M}^{2n})}{T(g^{TN}, \norm{\cdot}_{N}^{2n})^{\rk{\xi}}}.
		\end{equation}
	\end{prop}
	\begin{proof}
		First of all, by Definition \ref{defn_rel_zeta}, (\ref{eq_x_theta_defn}) and (\ref{defn_zeta_appr100}), we have
		\begin{equation}\label{int_x_theta_through_zeta}
			\frac{1}{\Gamma(s)} \int_{0}^{+ \infty}
			X_{\theta}(t)
			t^{s} \frac{dt}{t}
			=
			\zeta_{M}^{\theta}(s) - \rk{\xi} \zeta_{N}^{\theta}(s) - \zeta_{M}(s) + \rk{\xi} \zeta_{N}(s),
		\end{equation}
		Now, by Theorem \ref{conv_main_thm}, the function $X_{\theta}(t)$ has subexponential growth near 0 and $\infty$, thus, by (\ref{int_x_theta_through_zeta}), the left-hand side of (\ref{eq_zeta_conv}) is a holomorphic function over $\comp$ for any $\theta \in ]0, e^{-3}]$.
		\par Also, by Theorem \ref{conv_main_thm}, there are $c, c', C > 0$, $\varsigma > 0$ such that for any $t > 0$, $\theta \in ]0, e^{-3}]$:
		\begin{equation}\label{eq_x_theta_triv_bound}
			X_{\theta}(t) \leq C | \log \theta|^{-\epsilon} \exp(-ct - c'/t).
		\end{equation}
		In particular, by (\ref{eq_x_theta_triv_bound}), as $\theta \to 0$, we have
		\begin{equation}\label{cor_A_1}
			X_{\theta}(t) \to 0.
		\end{equation}
		By (\ref{int_x_theta_through_zeta}), (\ref{eq_x_theta_triv_bound}),  (\ref{cor_A_1}) and Lebesgue dominated convergence theorem, we deduce (\ref{eq_zeta_conv}).
		Now, (\ref{cor_tors_conv_eqn}) follows from Definitions \ref{defn_rel_zeta}, \ref{defn_rel_tor}, (\ref{defn_zeta_appr100}) and (\ref{eq_zeta_conv}).
	\end{proof}
		We denote by $\norm{\cdot}_{L^2}(g^{TM}_{\rm{f}, \theta}, h^{\xi} \otimes (\, \norm{\cdot}_{M}^{\rm{f}, \theta})^{2n})$ the $L^2$-norm over the line bundle (\ref{defn_det_line}) induced by $g^{TM}_{\rm{f}, \theta}$, $h^{\xi}$, $\norm{\cdot}_{M}^{\rm{f}, \theta}$.
		By \textit{properties 1,3} of tight families and Lebesgue dominated convergence, we have
	\begin{equation}\label{eqn_l_2_conv}
		\begin{aligned}
			& \lim_{\theta \to 0} \, \norm{\cdot}_{L^2} \big( g^{TM}_{\rm{f}, \theta}, h^{\xi} \otimes (\, \norm{\cdot}_{M}^{\rm{f}, \theta})^{2n} \big) 
			= 
			\norm{\cdot}_{L^2} \big(g^{TM}, h^{\xi} \otimes \norm{\cdot}_{M}^{2n} \big), \\
			& \lim_{\theta \to 0} \,  \norm{\cdot}_{L^2} \big( g^{TN}_{\rm{f}, \theta}, (\, \norm{\cdot}_{N}^{\rm{f}, \theta})^{2n} \big)
			= 
			\norm{\cdot}_{L^2} \big( g^{TN}, \norm{\cdot}_{N}^{2n} \big).
		\end{aligned}
	\end{equation}
	From  (\ref{cor_tors_conv_eqn}) and (\ref{eqn_l_2_conv}), we get (\ref{eq_lim_comp111}), as $\theta \to 0$.

\subsection{Proofs of Theorems  \ref{spec_gap_thm_A}, \ref{thm_A_bound_cups} }\label{sect_pf_aux_tight}
In this section we prove Theorems  \ref{spec_gap_thm_A}, \ref{thm_A_bound_cups}, which were announced in Section \ref{sect_tight}.
In the proof of Theorem  \ref{spec_gap_thm_A} we use the homogeneity of the Laplacian.
In the proof of Theorem \ref{thm_A_bound_cups}, we use the analytic localization techniques, the maximal principle and sup-characterization of the Bergman kernel. We recall that we suppose that $(\xi, h^{\xi})$ is trivial around the cusps.
	\begin{proof}[Proof of Theorem 	\ref{spec_gap_thm_A}]
	First of all, (\ref{eqn_ker_lapl_A}) is simply a consequence of Hodge theory for compact manifolds. 
	To prove (\ref{eqn_spec_gap_A}), by (\ref{lem_flat_compar}), it is enough to prove the following:
	\textit{let $g^{TM}_{0}$ be a Kähler metric on $\overline{M}$ and let $\, \norm{\cdot}_{M}^{0}$ be a Hermitian norm on $\omega_M(D)$ over $\overline{M}$ such that
		\begin{align}
				&g^{TM}_{0} \otimes (\, \norm{\cdot}_{M}^{0})^{2n} \leq 
				g^{TM}\otimes \norm{\cdot}_{M}^{2n}, 
				\label{lem_comp_cond_1}
				\\	
				\label{lem_comp_cond_2}
				&\norm{\cdot}_{M}^{0} \leq \norm{\cdot}_{M}.
		\end{align}
		Let $\langle \cdot, \cdot \rangle_{L^2_{0}}$ be the $L^2$-scalar product associated with $g^{TM}_{0}$, $h^{\xi}$, $\norm{\cdot}_{M}^{0}$, and let $\laplcomp^{E_M^{\xi, n}}_{0}$ be the associated Kodaira Laplacian. 
		Then for $n \leq 0$, we have
		\begin{equation}\label{lem_eig_comp_eqn}
			\inf \Big\{ \spec \, \big( \laplcomp^{E_M^{\xi, n}} \big) \setminus \{ 0 \} \Big\}
			\leq
			\inf \Big\{ \spec \, \big(  \laplcomp^{E_M^{\xi, n}}_{0} \big) \setminus \{ 0 \} \Big\}.
		\end{equation}
		}
		Now, by (\ref{eq_int_mu_fin}), we deduce
		\begin{equation}\label{eq_dom_lapl}
			\ccal^{\infty}(\overline{M}, E_M^{\xi, n}) \subset {\rm{Dom}}(\laplcomp^{E_M^{\xi, n}}).
		\end{equation}
		In the following series of transformations, we use (\ref{lem_comp_cond_2}) and $n \leq 0$ to get the inequality. For $s \in \ccal^{\infty}(\overline{M}, E_M^{\xi, n})$ we have
		\begin{equation}\label{lem_comp_1}
			\scal{ \laplcomp^{E_M^{\xi, n}}_{0} s }{s}_{L^{2}_{0}} 
			=
			\scal{ \overline{\partial}^{E_M^{\xi, n}} s }{\overline{\partial}^{E_M^{\xi, n}} s}_{L^{2}_{0}} 
			\geq
			\scal{ \overline{\partial}^{E_M^{\xi, n}} s }{\overline{\partial}^{E_M^{\xi, n}} s}_{L^{2}} 
			=
			\scal{ \laplcomp^{E_M^{\xi, n}} s }{s}_{L^{2}}.
		\end{equation}
		Also, from (\ref{lem_comp_cond_1}), we have
		\begin{equation}\label{lem_comp_10}
			\scal{s}{s}_{L^2_{0}} \leq \scal{s}{s}_{L^2}.
		\end{equation}
		From (\ref{lem_comp_1}) and (\ref{lem_comp_10}), we deduce
		\begin{equation}\label{lem_comp_111}
			\frac{\scal{ \laplcomp^{E_M^{\xi, n}}_{0} s }{s}_{L^{2}_{0}}}{\scal{s}{s}_{L^2_{0}}} \geq  \frac{\scal{ \laplcomp^{E_M^{\xi, n}} s }{s}_{L^{2}}}{\scal{s}{s}_{L^2}}.
		\end{equation}
		We denote $k = \dim H^0(\overline{M}, E_M^{\xi, n})$.
		By the min-max theorem (cf. \cite[(C.3.3)]{MaHol}) and (\ref{eq_dom_lapl}), we have
		\begin{equation}\label{lem_comp_2}
			\inf \Big\{ \spec \, \big( \laplcomp^{E_M^{\xi, n}} \big) \setminus \{ 0 \} \Big\}
			= \inf_{F \subset \ccal^{\infty}(\overline{M}, E_M^{\xi, n})} 
			\bigg\{
			 \sup_{s \in F}
			 \bigg\{
				\frac{\scal{ \laplcomp^{E_M^{\xi, n}} s }{s}_{L^{2}}}{\scal{s}{s}_{L^2}}
			\bigg\} : \dim F = k+1 
			\bigg\}.
		\end{equation}
		Then (\ref{lem_eig_comp_eqn}) follows from  (\ref{lem_comp_111}) and (\ref{lem_comp_2}).
	\end{proof}
	\begin{proof}[Proof of Theorem \ref{thm_A_bound_cups}.]
		This proof uses all the properties of tight families.
		The presence of the line bundle $\omega_M(D)$ makes analysis more difficult, and we have to consider 2 cases: $\theta^3 < |u| < |\log \theta|^{-1}$ and $|u| \leq \theta^3$. The main feature exploited in the first case is that we have elliptic estimate with the needed power of logarithm (\ref{lem_ell_est_eqn_A}), (\ref{def_rho_theta}). The main feature exploited in the second case is the \textit{property 2} of tight families along with the maximal principle (cf. \cite[p. 180]{Chav}).
		\par \textbf{Let's prove (\ref{eqn_A_bound_cups_perp_1}) for $\theta^3 \leq |u| \leq |\log \theta|^{-1}$.} 	
		We put $r = \dist(u, 1/2)$, then by (\ref{dist_hyp_comp}), $r \approx \ln |\ln  |u| |$.
		In this case, similarly to (\ref{eq_guh_equal_disc_rel}), by the fact that our flattenings are compatible and by the finite propagation speed of solutions of hyperbolic equations, we have
		\begin{equation}\label{fin_prop_luh_A}
			\widetilde{G}_{t,r}(\laplcomp^{E_M^{\xi, n}}_{\rm{f}, \theta}) \big( (z_i^{M})^{-1}(u), \cdot \big) = 
			{\rm{Id}_{\xi}} \cdot \widetilde{G}_{t,r} (\laplcomp^{E_N^{n}}_{\rm{f}, \theta}) \big( (z_i^{N})^{-1}(u), \cdot \big) .
		\end{equation}
		Then, similarly to (\ref{exp_scin_kuh_rel}), by (\ref{fin_prop_luh_A}), we have
		\begin{multline}\label{eqn_inf_kuh_repr_A}
			\exp(-t \laplcomp^{E_M^{\xi, n}}_{\rm{f}, \theta}) \big( (z_i^{M})^{-1}(u), \cdot \big)    - 
				{\rm{Id}_{\xi}} \cdot \exp(-t \laplcomp^{E_N^{n}}_{\rm{f}, \theta}) \big( (z_i^{N})^{-1}(u), \cdot \big) \\ 
				= 
				\widetilde{K}_{t,r}  (\laplcomp^{E_M^{\xi, n}}_{\rm{f}, \theta} ) \big( (z_i^{M})^{-1}(u), \cdot \big) - 
				{\rm{Id}_{\xi}}  \widetilde{K}_{t,r} (\laplcomp^{E_N^{n}}_{\rm{f}, \theta} ) \big( (z_i^{N})^{-1}(u), \cdot \big).
		\end{multline}
		Now, similarly to (\ref{eqn_kth_norm_sup_conv}), from (\ref{est_tilde_kuh}), (\ref{lem_ell_est_eqn_A}) and (\ref{eqn_inf_kuh_repr_A}), for any $\theta^3 \leq |u|, |v| \leq |\log \theta|^{-1}$,  we get 
		\begin{multline}\label{pf_thm_A_bound_cups_132}
			\Big|
				\exp(-t \laplcomp^{E_M^{\xi, n}}_{\rm{f}, \theta}) \big((z_i^{M})^{-1}(u), (z_i^{M})^{-1}(v) \big) 
				- 
				{\rm{Id}}_{\xi} \exp (-t \laplcomp^{E_N^{n}}_{\rm{f}, \theta})\big((z_i^{N})^{-1}(u), (z_i^{N})^{-1}(v)\big)
			\Big|_{h \times h, \theta}
			\\
			\leq C  | \ln |u| |  \exp ( - c' (\ln |\ln |u| |)^2/t ).	
		\end{multline}
		In particular, (\ref{pf_thm_A_bound_cups_132}) implies (\ref{eqn_A_bound_cups_perp_1}) for  $\theta^3 \leq |u| \leq |\log \theta|^{-1}$. 
		\begin{sloppypar}
		\textbf{Let's prove (\ref{eqn_A_bound_cups_perp_1}) for  $|u| \leq \theta^3$.} 
		This case is more subtle. We trivialize $(\omega_M(D), \norm{\cdot}_M)$, $(\omega_N(D), \norm{\cdot}_N)$ as in \textit{property 2} of tight families.
		Then, since $(\xi, h^{\xi})$ is trivial around the cusps, for $v, w \in D(\theta^3)$, we look at $\exp^{\perp}(-t \laplcomp^{E_M^{\xi, n}}_{\rm{f}, \theta}) ((z_i^{M})^{-1}(v), (z_i^{M})^{-1}(w) )$ and ${\rm{Id}}_{\xi} \exp^{\perp}(-t \laplcomp^{E_N^{n}}_{\rm{f}, \theta})((z_i^{N})^{-1}(v), (z_i^{N})^{-1}(w))$ as at the functions over $D(\theta^3) \times D(\theta^3)$ with values in $\enmr{\xi|_{P_i^{M}}}$.
		\par
		For $v, w \in D(\theta^3)$, we denote 
		\begin{multline}
			F(v, w, t) := \exp (-t \laplcomp^{E_M^{\xi, n}}_{\rm{f}, \theta}) \big((z_i^{M})^{-1}(v), (z_i^{M})^{-1}(w) \big)
			\\
			- 
				{\rm{Id}}_{\xi} \exp (-t \laplcomp^{E_N^{n}}_{\rm{f}, \theta}) \big((z_i^{N})^{-1}(v), (z_i^{N})^{-1}(w)\big). 
		\end{multline}
		\end{sloppypar}
		\noindent We write $F(v, w, t) = (F_{kl}(v, w, t))_{k,l = 1}^{\dim \xi}$ for the components of the matrix from $\enmr{\xi|_{P_i^{M}}}$.
		We notice that the functions $F_{kl}(v, w, t)$ satisfy the heat equation with zero initial data in $D(\theta^3) \times D(\theta^3) \times ]0, +\infty[$, i.e. for any $k,l = 1, \ldots, \dim \xi$, we have
		\begin{equation}
			\Big( \frac{\partial}{\partial t} + \laplcomp_{\rm{f}, \theta} \Big) F_{kl} (u, v, t) = 0
			\quad 
			\text{and}
			\quad
			\lim_{t \to 0} F_{kl} (u, v, t) = 0,
		\end{equation}
		where $\laplcomp_{\rm{f}, \theta}$ is the Laplace–Beltrami operator induced by $((z_i^{M})^{-1})^* g^{TM}_{\rm{f}, \theta}$ on $D(\theta^{3})$.
		Thus, by the maximal principle (cf. \cite[p. 180]{Chav}), for $|u| \leq \theta^3$, we get
		\begin{equation}\label{eqn_max_pr1}
			| F_{kl} (u, u, t) | \leq \sup_{\tau' \in [0, t]} 
			\sup_{|w| = \theta^{3}} | F_{kl} (u, w, \tau') |.
		\end{equation}
		By applying the maximal principle again, we get
		\begin{equation}\label{eqn_max_pr2}
			| F_{kl} (u, w, \tau') | \leq \sup_{\tau \in [0, \tau']} 
			\sup_{|v| = \theta^3} | F_{kl} (v, w, \tau) |.
		\end{equation}
		By (\ref{pf_thm_A_bound_cups_132}), there are $c', C > 0$ such that for any $\theta \in ]0, e^{-3}]$, and $|v|, |w| = \theta^3$, we have
			\begin{align}\label{est_off_diag}
				| F_{kl} (v, w, \tau) | \leq | \ln \theta|  \exp ( - c' (\ln |\ln \theta|)^2/\tau ).
			\end{align}
			By (\ref{eqn_max_pr1}), (\ref{eqn_max_pr2}) and (\ref{est_off_diag}), we get (\ref{eqn_A_bound_cups_perp_1}) for $|u| \leq \theta^3$. Thus, (\ref{eqn_A_bound_cups_perp_1}) is completely proved.
			\par \textbf{Now let's prove (\ref{eqn_A_bound_cups_perp_2}).}
			By Theorem \ref{thm_hk_estA}, there are $c, C > 0$ such that for any $|u| \leq | \ln \theta|^{-1}$, $\theta \in ]0, e^{-3}]$, $t > 0$, we have
			\begin{multline}\label{pf_est_diag_111}
				\Big|
					\exp^{\perp}(-t \laplcomp^{E_M^{\xi, n}}_{\rm{f}, \theta}) \big((z_i^{M})^{-1}(u), (z_i^{M})^{-1}(u) \big) 
					- 
					{\rm{Id}}_{\xi} \exp^{\perp}(-t \laplcomp^{E_N^{n}}_{\rm{f}, \theta})\big((z_i^{N})^{-1}(u), (z_i^{N})^{-1}(u)\big)
				\Big| 
				\\
				\leq C ( \ln \max( \theta, |u| ) )^{12} t^{-4} \exp ( - ct ).	
			\end{multline}
		Now, for any $x, x' \in M$, we have
		\begin{equation}\label{eq_rel_perp_noperp}
			\exp(-t \laplcomp^{E_M^{\xi, n}}_{\rm{f}, \theta})(x, x') = \exp^{\perp}(-t \laplcomp^{E_M^{\xi, n}}_{\rm{f}, \theta})(x, x') +  B^{E_M^{\xi, n}}_{\theta}(x,x'), 
		\end{equation}
		where $B^{E_M^{\xi, n}}_{\theta}(x,x')$ is the Bergman kernel, defined by
		\begin{equation}
			B^{E_M^{\xi, n}}_{\theta}(x,x') = \sum s_i(x) (s_i(x'))^*_{\theta},
		\end{equation}
		for an orthonormal base $\{s_i\}$ of $H^{0}(\overline{M}, E_M^{\xi, n})$ with respect to the $L^2$-scalar product induced by $g^{TM}_{\rm{f}, \theta}$, $h^{\xi}$, $\norm{\cdot}_M^{\rm{f}, \theta}$, and $(\cdot)^*_{\theta}$ is the dual with respect to $| \cdot |_{h, \theta}$. 
		By \cite[Lemma 3.1]{ComMar}, we have 
		\begin{equation}\label{eq_berg_CM}
			B^{E_M^{\xi, n}}_{\theta}(x,x) = \max \bigg\{ 
			\frac{|s(x)|_{h, \theta}^{2}}{\norm{s}_{L^2, \theta}^{2}} : s \in H^{0}(\overline{M}, E_M^{\xi, n}) \setminus \{0\}
			\bigg\}.
		\end{equation}
		By (\ref{lem_flat_compar}) and the fact that $n \leq 0$, we see that for any $s \in \ccal^{\infty}(\overline{M}, E_M^{\xi, n})$, we have
		\begin{equation}\label{eq_compar_berg}
			|s(x)|_{h, \theta} \leq |s(x)|_{h, {\rm{sm}}}, \qquad
			\norm{s}_{L^2, \theta} \geq \norm{s}_{L^2, {\rm{sm}}},
		\end{equation}
		where $|\cdot|_{h, {\rm{sm}}}$ is the pointwise norm induced by $h^{\xi}$, $\norm{\cdot}_{{\rm{sm}}}$, and $\norm{\cdot}_{L^2, {\rm{sm}}}$ is the $L^2$-norm induced by $h^{\xi}$, $\norm{\cdot}_{{\rm{sm}}}$, $g^{TM}_{{\rm{sm}}}$. From (\ref{eq_berg_CM}) and (\ref{eq_compar_berg}), we deduce 
		\begin{equation}\label{eq_compar_berg2}
			B^{E_M^{\xi, n}}_{\theta}(x,x) \leq B^{E_M^{\xi, n}}_{{\rm{sm}}}(x,x),
		\end{equation}
		where $B^{E_M^{\xi, n}}_{{\rm{sm}}}(x, x')$ is the Bergman kernel associated with $h^{\xi}$, $\norm{\cdot}_{{\rm{sm}}}$, $g^{TM}_{{\rm{sm}}}$. Thus, from (\ref{eqn_A_bound_cups_perp_1}), (\ref{eq_rel_perp_noperp}) and (\ref{eq_compar_berg2}), there is $C > 0$ such that for any $\theta \in ]0, 1/2], |u| < \theta^3$, we have
		\begin{multline}\label{pf_thm_A_bound_cups_11}
			\Big|
				\exp^{\perp}(-t \laplcomp^{E_M^{\xi, n}}_{\rm{f}, \theta}) \big((z_i^{M})^{-1}(u), (z_i^{M})^{-1}(u) \big) 
				- 
				{\rm{Id}}_{\xi} \exp^{\perp}(-t \laplcomp^{E_N^{n}}_{\rm{f}, \theta})\big((z_i^{N})^{-1}(u), (z_i^{N})^{-1}(u)\big)
			\Big| 
			\\
			\leq C \Big(1 +  | \ln \max (\theta, |u| ) |  \exp \big( - c' (\ln |\ln \max (\theta, |u| )|)^2/t \big) \Big).	
		\end{multline}
			By multiplying (\ref{pf_est_diag_111}) with power $\mu \in ]0, 1/2[$ and (\ref{pf_thm_A_bound_cups_11}) with power $1 - \mu$, we have
			\begin{multline}\label{pf_est_diag_112}
				\Big|
					\exp^{\perp}(-t \laplcomp^{E_M^{\xi, n}}_{\rm{f}, \theta}) \big((z_i^{M})^{-1}(u), (z_i^{M})^{-1}(u) \big) 
					- 
					{\rm{Id}}_{\xi} \exp^{\perp}(-t \laplcomp^{E_N^{n}}_{\rm{f}, \theta})\big((z_i^{N})^{-1}(u), (z_i^{N})^{-1}(u)\big)
				\Big| 
				\\
				\leq C | \ln \max (\theta, |u| )|^{1 + 11 \mu} t^{-4} \exp ( - c \mu t - c'  (\ln |\ln \max (\theta, |u| )|)^2/t )
				\\
				+
				C | \ln \max (\theta, |u| )|^{12 \mu} t^{-4} \exp ( - c \mu t ).	
			\end{multline}
			By (\ref{eq_cauchy_exp}) and (\ref{pf_est_diag_112}), we finally get (\ref{eqn_A_bound_cups_perp_2}) by taking $\mu$ small enough.
	\end{proof}

\section{The anomaly formula: a proof of Theorem \ref{thm_anomaly_cusp}}
	In this section we prove Theorem \ref{thm_anomaly_cusp}. 
	First of all, we recall that in Section \ref{sect_compact_pert} we proved Theorem \ref{thm_anomaly_cusp} for $g^{TM}_{0} = g^{TM}$, i.e. when we have only the variation of $h^{\xi}$.
	Thus, it's left to prove Theorem \ref{thm_anomaly_cusp} for $h^{\xi}_{0} = h^{\xi}$ and under the supposition that $(\xi, h^{\xi})$ is trivial around the cusps. 
	The idea of the proof is as follows: we construct a family of flattenings which “approach" the cusp metric and we use Theorem \ref{thm_comp_appr} to relate the corresponding relative Quillen norms. 
	Then we apply the anomaly formula Bismut-Gillet-Soulé \cite[Theorem 1.23]{BGS3} (see Theorem \ref{thm_anomaly_BGS}) and calculate the limit of the right-hand side of (\ref{eq_anomaly}), as the family of flattenings “approach" the cusp metric.  
	\par Before giving a proof of Theorem \ref{thm_anomaly_cusp}, let's fix some notation. By suppositions of Theorem \ref{thm_anomaly_cusp}, there are germs at $0 \in D(\epsilon)$, $\epsilon > 0$ of holomorphic functions $h_i^{\phi} : D(\epsilon) \to D(1)$, $i = 1,\ldots, m$, such that $g^{TM}_{0}$ is Poincaré-compatible with coordinates $h_i^{\phi}(z_i^{M})$ around $P_i^{M} \in D_M$. We note 
	\begin{equation}
		z_i^{0, M} := h_i^{\phi}(z_i^{M}).
	\end{equation}
	By Definition \ref{defn_wolpert_norm} of the Wolpert norm, we have the following identity
	\begin{equation}\label{eqn_rel_wolp_der_hol}
		\ln \Big( \norm{\cdot}^W / \norm{\cdot}^W_{0} \Big) =  \sum \ln \big| (h_i^{\phi})'(0) \big|.
	\end{equation}
	\par \textbf{First of all, let's describe why the right-hand side of (\ref{eq_anomaly_cusp}) is finite.} For $\epsilon > 0$, in  $V_i^{M}(\epsilon)$:
	\begin{equation}\label{est_c_1}
		c_1 \big( \omega_M(D), (\, \norm{\cdot}_{M})^{2} \big)|_{M} = \frac{\partial \dbar}{2 \pi \imun} \ln \big(\, \norm{s}_M^{2} \big) 
		= O \big( |z_i^{M}  \ln |z_i^{M}||^{-2}\big),
	\end{equation}
	where $s$ is a local holomorphic frame of $\omega(D_M)$. Similar estimation holds for the norm $ \norm{\cdot}_{M}^{0}$.
	The identity (\ref{anomaly_rel_metrics}) says
	\begin{equation}\label{eq_phi_conf_ch}
		\frac{e^{2 \phi} dz_i^{M} d\overline{z}_i^{M}}{ \big| z_i^{M}  \ln |z_i^{M}| \big|^2} 
		=  
		\frac{ dz_i^{0,M} d\overline{z}_i^{0,M}}{ \big| z_i^{0,M}  \ln |z_i^{0,M}| \big|^2}.
	\end{equation}
	By (\ref{eqn_norms_local}) and (\ref{eq_phi_conf_ch}), we see that over $V_i^{M}(\epsilon)$, we have
	\begin{equation}\label{est_phi}
		\ln \big(\, \norm{\cdot}_{M}^{0} / \norm{\cdot}_{M} \big) 
		= 
		O \big( | \ln |z_i^{M}| |^{-1} \big). 
	\end{equation}		
	By (\ref{ch_bc_0}), (\ref{ch_bc_2}), (\ref{eq_int_mu_fin}), (\ref{est_c_1}) and (\ref{est_phi}), we conclude that the right-hand side of (\ref{eq_anomaly_cusp}) is finite.
	\begin{sloppypar}
	\par \textbf{Now let's describe the precise family of flattenings we choose.} Recall that the function $\psi: \real \to [0,1]$ was defined in (\ref{eq_defn_psi}).
	Let $g^{TM}_{\rm{f}, \theta}$ be a metric over $\overline{M}$ such that it coincides with $g^{TM}$ away from $\cup_i V_i^{M}(\theta)$, and over $V_i^{M}(\theta)$ it is induced by
	\begin{equation}\label{defn_g_A_an}
		\Big| z_i^{M} \ln |z_i^{M}| \Big|^{-2 \psi(\ln |z_i^{M}| / \ln \theta)} \imun d z_i^{M} d \overline{z}_i^{M},
	\end{equation}
	for all $i = 1, \ldots, m$.
	Similarly, let $\norm{\cdot}_{M}^{\rm{f}, \theta}$ be the metric on $\omega_M(D)$ over $\overline{M}$ such that it coincides with $\norm{\cdot}_{M}$ away from $\cup_i V_i^{M}(\theta)$, and over $V_i^{M}(\theta)$, $i = 1, \ldots, m$, we have
	\begin{equation}\label{defn_norm_A_an}
		\Big\| dz_i^{M} \otimes s_{D_M} / z_i^{M} \Big\|_{M}^{\rm{f}, \theta}= \Big| \ln |z_i^{M}| \Big|^{\psi(\ln |z_i^{M}| / \ln \theta)},
	\end{equation}
	where $s_{D_M}$ is the canonical section of $\mathscr{O}_{\overline{M}}(D_M)$, ${\rm{div}}(s_{D_M}) = D_M$.
	By (\ref{eqn_norms_local}), this defines a smooth metric over $\overline{M}$.
	\end{sloppypar}
	For $\epsilon > 0$, $i = 1, \ldots, m$, we denote
	\begin{equation}
		V_i^{0, M}(\epsilon) := \{ x \in M : | z_i^{0, M}(x) | \leq \epsilon \}.
	\end{equation}		
	Let $g^{TM}_{0, {\rm{f}}, \theta}, \norm{\cdot}_{0,M}^{\rm{f}, \theta}$ be the flattenings of $g^{TM}_{0}, \norm{\cdot}_{M}^{0}$, compatible with the flattenings $g^{TM}_{\rm{f}, \theta}, \norm{\cdot}_{M}^{\rm{f}, \theta}$ (cf. (\ref{eq_gtm_compat_cusp}), (\ref{eq_comp_gtn})). More precisely, the metrics $g^{TM}_{0, {\rm{f}}, \theta}$, $\norm{\cdot}_{0, M}^{\rm{f}, \theta}$ coincide with $g^{TM}_{0}$, $\norm{\cdot}_{M}^{0}$ away from $\cup_i V_i^{0, M}(\theta)$, and over $V_i^{0, M}(\theta)$ the metric $g^{TM}_{0, {\rm{f}}, \theta}$ is induced by 
	\begin{equation}\label{defn_g_A_0}
		\Big| z_i^{0, M} \ln |z_i^{0, M}| \Big|^{-2 \psi(\ln |z_i^{0, M}| / \ln \theta)} \imun d z_i^{0, M} d \overline{z}_i^{0, M}.
	\end{equation}
	Also, for $s_{D_M}$ as in (\ref{defn_norm_A_an}), we have 
	\begin{equation}\label{defn_norm_A_0}
		\Big\| dz_i^{0, M} \otimes s_{D_M} /z_i^{0, M} \Big\|_{0, M}^{\rm{f}, \theta}= \Big| \ln |z_i^{0, M}| \Big|^{\psi(\ln |z_i^{0, M}| / \ln \theta)}.
	\end{equation}
	Let's denote by $\norm{\cdot}_{\rm{f}, \theta, M}^{\omega, 0}$, $\norm{\cdot}_{\rm{f}, \theta, M}^{\omega}$ the norms on $\omega_{\overline{M}}$ over $\overline{M}$ induced by $g^{TM}_{0, {\rm{f}}, \theta}$ and $g^{TM}_{{\rm{f}}, \theta}$ respectively.
	\par \textbf{Proof of 	(\ref{eq_anomaly_cusp}).}
	By Theorem \ref{thm_comp_appr}, for any $\theta \in ]0,1]$, we have
	\begin{equation}\label{eq_quil_0_id}
		2\ln \Bigg( 
		\frac{\norm{\cdot}_{Q}(g^{TM}_{0}, h^{\xi} \otimes \, \norm{\cdot}_{0, M}^{2n}) }{\norm{\cdot}_{Q} (g^{TM}, h^{\xi} \otimes \, \norm{\cdot}_{M}^{2n} ) }
		\Bigg) 
		=
		2\ln \Bigg( 
		\frac{\norm{\cdot}_{Q}  \big( g^{TM}_{0, {\rm{f}}, \theta}, h^{\xi} \otimes (\, \norm{\cdot}_{0, M}^{\rm{f}, \theta})^{2n} \big) }{\norm{\cdot}_{Q}  \big( g^{TM}_{\rm{f}, \theta}, h^{\xi} \otimes (\, \norm{\cdot}_{M}^{\rm{f}, \theta})^{2n}  \big)}
		\Bigg).
	\end{equation}
	We will show that the limit of the right-hand side of (\ref{eq_quil_0_id}), as $\theta \to 0$ is exactly the right-hand side of (\ref{eq_anomaly_cusp}). Once it will be done, Theorem \ref{thm_anomaly_cusp} would follow from (\ref{eq_quil_0_id}). 
	\begin{sloppypar} 
	We denote by $\norm{\cdot}_{\omega, M}^{\rm{f}, \theta}, \norm{\cdot}_{\omega, 0, M}^{\rm{f}, \theta}$ the norms on $\omega_M$ induced by $g^{TM}_{\rm{f}, \theta}$ and $g^{TM}_{0, {\rm{f}}, \theta}$.
	Set
	\begin{multline}
		\Phi(\theta) := \Big[ 
			\widetilde{\td} \big(\omega_{M}^{-1}, ( \, \norm{\cdot}_{\rm{f}, \theta, M}^{\omega} )^{-2}, ( \, \norm{\cdot}_{\rm{f}, \theta, M}^{\omega, 0} )^{-2} \big) \ch \big( \xi, h^{\xi} \big) \ch \big( \omega_M(D), (\, \norm{\cdot}_{M}^{\rm{f}, \theta})^{2n} \big)  
			\\ 
			+			 					 
			\td \big( \omega_{M}^{-1}, ( \, \norm{\cdot}_{\rm{f}, \theta, M}^{\omega, 0} )^{-2} \big) 
			\ch
			\big( \xi, h^{\xi} \big)			
			\widetilde{\ch}
			\big( 
			\omega_M(D), (\, \norm{\cdot}_{M}^{\rm{f}, \theta})^{2n}, (\, \norm{\cdot}_{0, M}^{\rm{f}, \theta})^{2n} \big) 			
		\Big]^{[2]}.
	\end{multline}
	Then, by Theorem \ref{thm_anomaly_BGS}, we have
	\begin{equation}\label{anom_flat_a_1}
		2\ln  
		\Bigg( 
		\frac{\norm{\cdot}_{Q}  \big( g^{TM}_{0, {\rm{f}}, \theta}, h^{\xi} \otimes (\, \norm{\cdot}_{0, M}^{\rm{f}, \theta})^{2n} \big) }{\norm{\cdot}_{Q}  \big( g^{TM}_{\rm{f}, \theta}, h^{\xi} \otimes (\, \norm{\cdot}_{M}^{\rm{f}, \theta})^{2n}  \big)}
		\Bigg) = \int_M \Phi(\theta).
	\end{equation}
	where $\widetilde{\td}$ and $\widetilde{\ch}$ are given by (\ref{ch_bc_0}) and (\ref{ch_bc_2}).
	We decompose the right-hand side of (\ref{anom_flat_a_1}) into integrals over $M \setminus ( \cup_i ( V_i^{M}(\theta) \cup V_i^{0, M}(\theta) ) )$ and over $V_i^{M}(\theta) \cup V_i^{0, M}(\theta)$, $i = 1, \ldots, m$. Since the flattenings $g^{TM}_{0, {\rm{f}}, \theta}, g^{TM}_{\rm{f}, \theta}$ and $\norm{\cdot}_{0, M}^{\rm{f}, \theta}, \norm{\cdot}_{M}^{\rm{f}, \theta}$ coincide with $g^{TM}_{0}, g^{TM}$ and $\norm{\cdot}_{M}^{0}, \norm{\cdot}_{M}$ over $M \setminus (\cup_i ( V_i^{M}(\theta) \cup V_i^{0, M}(\theta)))$, and the quantities under the integration in the anomaly formula are local, we see by Lebesgue dominated convergence theorem, by the finiteness of the right-hand side of (\ref{eq_anomaly_cusp}) and by (\ref{eq_ch_similarity}), that the integral of $\Phi(\theta)$ over $M \setminus (\cup_i ( V_i^{M}(\theta) \cup V_i^{0, M}(\theta) ))$ converges to the integral part in the right-hand side of (\ref{eq_anomaly_cusp}), as $\theta \to 0$. 
	\end{sloppypar}
	\par Now let's study the contribution over $\cup_i (V_i^{M}(\theta) \cup V_i^{0, M}(\theta))$ of the integral in (\ref{anom_flat_a_1}). We note that in the case when $\phi$ has compact support in $M$, this integral is actually zero for $\theta$ sufficiently small (which is consistent with the statement of Theorem \ref{thm_anomaly_cusp}).

	From the discussion above, (\ref{eqn_rel_wolp_der_hol}), (\ref{anom_flat_a_1}), and the fact that we restrict ourselves to the case $(\xi, h^{\xi})$ trivial around the cusps, Theorem \ref{thm_anomaly_cusp}  would follow from the following
	\begin{lem}
		As $\theta \to 0$, we have
		\begin{equation}\label{eqn_lim_h_j}
			\int_{V_i^{M}(\theta) \cup V_i^{0, M}(\theta)} \Phi(\theta) \to -\frac{\rk{\xi}}{6} \ln |(h_i^{\phi})'(0)|.
		\end{equation}
	\end{lem}
	\begin{proof}
	All the subsequent formulas should be regarded as being valid over $V_i^{M}(\theta) \cup V_i^{0, M}(\theta)$. 
	By (\ref{defn_g_A_an}) and (\ref{defn_norm_A_an}), we have
	\begin{multline}\label{est_an_1}
		c_1 \big(\omega_M, (\, \norm{\cdot}_{\rm{f}, \theta, M}^{\omega})^2 \big) = - \frac{\imun \partial \dbar}{2 \pi} 
		\bigg( 
			\psi
			\big(
				\ln |z_i^{M}| / \ln \theta
			\big)
			\cdot
			\big(
				2 \ln |z_i^{M}| +  2 \ln | \ln |z_i^{M}||
			\big)
		\bigg) 
		\\
		= 
		\bigg[
			\frac{\ln |z_i^{M}| 
			\psi''
			\big(
				\ln |z_i^{M}| / \ln \theta
			\big)	}{2 |z_i^{M} \ln \theta|^{2}}
			+
			\frac{\psi'
			\big(
				\ln |z_i^{M}| / \ln \theta
			\big)	}{|z_i^{M}|^2 \ln \theta}			
			\\
			+
			O \bigg(
			\frac{\ln | \ln |z_i^{M}||}{|z_i^{M} \ln |z_i^{M}||^{2} }			
			\bigg)
		\bigg]
		\frac{- \imun d z_i^{M} d \overline{z}_i^{M}}{2\pi},
	\end{multline}
	\vspace*{-0.5cm}
	\begin{multline}\label{est_an_2}
		c_1(\omega_M(D), (\, \norm{\cdot}_{M}^{\rm{f}, \theta})^2) = - \frac{\imun \partial \dbar}{2 \pi} 
		\Big( 
			\psi
			\big(
				\ln |z_i^{M}| / \ln \theta
			\big)
			\cdot
			\big(
				2 \ln | \ln |z_i||
			\big)
		\Big) 
		\\
		= 
			O \bigg(
			\frac{\ln | \ln |z_i^{M}||}{|z_i^{M} \ln |z_i^{M}||^{2} }			
			\bigg)
		d z_i^{M} d \overline{z}_i^{M}.
	\end{multline}
	By $dz_i^{0, M} = (h_i^{\phi})'(z_i^{M}) \cdot d z_i^{M}$ and $\ln | \ln |z_i^{0, M}| | = \ln | \ln |z_i^{M}| | + O(1/|\ln |z_i^{0, M}| |)$, we deduce
	\begin{multline}\label{est_an_3}
		\ln \big( 
		\norm{\cdot}_{\rm{f}, \theta, M}^{\omega, 0} / \norm{\cdot}_{\rm{f}, \theta, M}^{\omega}		
		\big)		
		 = \,  
		 	\psi
			\big(
				\ln |z_i^{0, M}| / \ln \theta
			\big)
			\Big( 
				\ln |z_i^{0, M}| + \ln | \ln |z_i^{0, M}| |
			\Big) 
			 \\ 
			-
			\psi
			\big(
				\ln |z_i^{M}| / \ln \theta
			\big)
			\Big( 
				\ln |z_i^{M}| + \ln | \ln |z_i^{M}| |
			\Big) 
			-
			\ln |(h_i^{\phi})'(z_i^{M})|				
		\\
		= 
		\ln |(h_i^{\phi})'(0)| 
		\Big(
			- 1 + 					
			\psi
			\big(
				\ln |z_i^{M}| / \ln \theta
			\big)
			+
			\psi'
			\big(
				\ln |z_i^{M}| / \ln \theta
			\big)
			\frac{\ln |z_i^{M}|}{\ln \theta}
		\Big)				
		\\ 
		+ O \bigg(
		\frac{\ln | \ln |z_i^{M}||}{|\ln |z_i^{M}|| }			
		\bigg),
	\end{multline}
	\vspace*{-0.5cm}
	\begin{multline} \label{est_an_5}
		\ln \big( 
		\norm{\cdot}_{0, M}^{\rm{f}, \theta} / \norm{\cdot}_{M}^{\rm{f}, \theta}		
		\big)		
		 =   
		 	\psi
			\big(
				\ln |z_i^{0, M}| / \ln \theta
			\big)
				\ln | \ln |z_i^{0, M}| | 
			\\
			-
			\psi
			\big(
				\ln |z_i^{M}| / \ln \theta
			\big)
			\ln | \ln |z_i^{M}| |
		= 
		 O \bigg(
		\frac{\ln | \ln |z_i^{M}||}{|\ln |z_i^{M}|| }			
		\bigg).
	\end{multline}
	Finally, from (\ref{est_an_1}) and the analogical statement for $\norm{\cdot}_{\rm{f}, \theta, M}^{\omega, 0}$, we easily get
	\begin{flalign} \label{est_an_4}
		\quad \partial \dbar \ln \big( 
		\norm{\cdot}_{\rm{f}, \theta, M}^{\omega, 0} / \norm{\cdot}_{\rm{f}, \theta, M}^{\omega}		
		\big)		
		 =  O \bigg(
		\frac{\ln | \ln |z_i^{M}||}{| z_i^{M} \ln |z_i^{M}| |^2 }			
		\bigg) d z_i^{M} d \overline{z}_i^{M} . &&
	\end{flalign}
	From Theorem \ref{thm_anomaly_BGS}, (\ref{ch_bc_0}), (\ref{ch_bc_2}) and (\ref{est_an_1}) - (\ref{est_an_4}), we get
	\begin{multline}\label{eqn_di_identity}
		\int_{V_i^{M}(\theta) \cup V_i^{0, M}(\theta)} \Phi(\theta)  =
		-\frac{\rk{\xi}}{3}
		\int_{V_i^{M}(\theta) \cup V_i^{0, M}(\theta)}
		\bigg[
		 c_1 \big( \omega_M, (\, \norm{\cdot}_{\rm{f}, \theta, M}^{\omega})^2 \big) 
		\ln \bigg( 
		\frac{\norm{\cdot}_{\rm{f}, \theta, M}^{\omega, 0}}{\norm{\cdot}_{\rm{f}, \theta, M}^{\omega}}
		\bigg)	
		\\ 
		+ 
		O \bigg(
		\frac{\ln | \ln |z_i^{M}||}{| z_i^{M} \ln |z_i^{M}| |^2 }			
		\bigg) d z_i^{M} d \overline{z}_i^{M}
		\bigg].
	\end{multline}
	From (\ref{est_an_1}),  (\ref{est_an_3}) and (\ref{eqn_di_identity}), we get
	\begin{align}
		\lim_{\theta \to 0} & \int_{V_i^{M}(\theta) \cup V_i^{0, M}(\theta)} \Phi(\theta) =   \frac{2  \ln |(h_i^{\phi})'(0)|}{3}
		\cdot		
		\rk{\xi}
		\\
		& \cdot
		\lim_{\theta \to 0}
		\int_{\theta}^{\theta^{1/2}} 
		\frac{1}{r} 
		\Big(
			\psi''
			\Big(
				\frac{\ln r}{\ln \theta}
			\Big) 
			\frac{\ln r}{2 (\ln \theta)^2} 
			+
			\psi'
			\Big(
				\frac{\ln r}{\ln \theta}
			\Big) 
			\frac{1}{\ln \theta}
		\Big)
		\Big(
			-1
			+
			\psi
			\Big(
				\frac{\ln r}{\ln \theta}
			\Big) 
			+
			\psi'
			\Big(
				\frac{\ln r}{\ln \theta}
			\Big) 
			\frac{\ln r}{\ln \theta} 
		\Big)
		dr
		\nonumber		
		\\
	 	 & =   -\frac{2 \ln |(h_i^{\phi})'(0)|}{3}
			\cdot		
			\rk{\xi}
			\cdot 
			\int_{1/2}^{1} \Big( - \psi'(u) + \psi'(u) \psi(u) + u \psi'(u)^2
			\nonumber			
			\\ \label{lim_d_i}
			&  
			\phantom{\frac{2 \ln |h_j'(0)|}{3} 
			\int_1^{2} \Big( }
			\qquad \qquad
			\quad 
			 - u \psi''(u)/2 + u \psi''(u) \psi(u)/2 + u^2 \psi'(u) \psi''(u)/2 \Big) du,
	\end{align}
	where in the last identity we used the change of variables $u := \ln r / \ln \theta$.
	By the integration by parts and (\ref{eq_defn_psi}), we have
	\begin{equation}\label{int_cal_1}
	\begin{aligned}
		&  \int_{1/2}^{1} \psi'(u) du = -1,
		&&
		\int_{1/2}^{1} u \psi''(u) \psi(u) du = \frac{1}{2} - \int_{1/2}^{1} u \psi'(u)^2 du,
		\\
		& \int_{1/2}^{1} u \psi''(u) du = 1,   
		&&
		\int_{1/2}^{1} u^2 \psi'(u) \psi''(u) du = - \int_{1/2}^{1} u \psi'(u)^2 du, 
		\\
		&  \int_{1/2}^{1} \psi'(u) \psi(u) du = -\frac{1}{2}.
		&& 
	\end{aligned}
	\end{equation}
	We get (\ref{eqn_lim_h_j}) from (\ref{lim_d_i}), (\ref{int_cal_1}).
	\end{proof}
	
\appendix

\section{Elliptic estimates}
	Here we collected the results related to elliptic estimates of Kodaira Laplacians. More precisely, in Section \ref{app_1} we prove Lemma \ref{lem_ell_est}, and in Section \ref{app_2} we prove the existence of $n$-tight flattenings, which essentially boils down to proving that a certain family of metrics satisfy the uniform elliptic estimate (\ref{lem_ell_est_eqn_A}). For $0 < a < b$, $i = 1, \ldots, m$, we denote (see (\ref{defn_v_i}))
	\begin{equation}\label{notat_ci}
		C_i^{M}(a, b) := V_i^{M}(b) \setminus V_i^{M}(a).
	\end{equation}
\subsection{A proof of Lemma \ref{lem_ell_est}}\label{app_1}
	\begin{proof}[Proof of Lemma \ref{lem_ell_est}.]
		First of all, for $\epsilon > 0$ small enough, (\ref{lem_ell_est_eqn}) holds for $x \in M \setminus (\cup_i V_i^{M}(\epsilon))$ by \cite[Lemma 1.6.2]{MaHol}, thus, we may concentrate on the proof of (\ref{lem_ell_est_eqn}) for $x \in V_i^{M}(\epsilon)$. 
		We will prove (\ref{lem_ell_est_eqn})  for metrics, which are scaled by a factor $1/4$ in the neighbourhood of cusps. In other words, instead of the constant scalar curvature $-1$, it will be $-4$ around the cusps. We choose to do so, since part of the calculation is already done with this normalization in \cite{Auvr}.
		\par
		We fix a holomorphic frame $e_1, \ldots, e_{\rk{\xi}}$ of $\xi$ over $V_i^{M}(\epsilon)$.
		We note that since $(\xi, h^{\xi})$ is defined over $\overline{M}$, by (\ref{eq_h_xi_appr_1}), for any $k \in \nat$, there is $C > 0$ such that for any $i=1, \ldots, \rk{\xi}$, we have $|\nabla^k e_i|_h \leq C$ over $V_i^{M}(\epsilon)$.
		By this and (\ref{eq_comp_e_lapl}), we see that we may suppose that $(\xi, h^{\xi})$ is trivial.
		We do so and remove $\xi$ from further notation.
		For simplicity, we prove (\ref{lem_ell_est_eqn}) for $\alpha = {\rm{diam}}(V_i^{M}(2^{-4}, 2^{-1}))$, where ${\rm{diam}}$ is the diameter of a set with respect to $g^{TM}$. We write $s_{D_M}$ for the canonical section of $\mathscr{O}(D_M)$:
		\begin{equation}\label{defn_sigma}
			\sigma = f \cdot ( dz_i^{M} \otimes s_{D_M} / z_i^{M} )^n,
		\end{equation}
		where $f : V_i^{M}(\epsilon) \to \comp$ is some smooth function.
		\par
		\textbf{Let's prove Lemma \ref{lem_ell_est} for $n = 0$, $k \in \nat$.} 	
		Let $g : \real \to [0, 1]$ be a smooth extension of
		\begin{equation}
			g(u) = 
			\begin{cases} 
      			\hfill 1 & \text{ for }u \in [2,3], \\
      			\hfill 0  & \text{ for }u \in ]- \infty,1] \cup [4, + \infty[. \\
 			\end{cases}
		\end{equation}
		We introduce the function $\psi_{\theta} : \overline{M} \to [0, 1]$ for $\theta \in ]0, \epsilon]$ by
		\begin{equation}\label{defn_psiteta}
			\psi_{\theta}(x) = 
			\begin{cases} 
      			\hfill g(\ln |z_i^{M}(x)|/ \ln \theta) & \text{for }x \in V_{i}^{M}(\theta), \\
      			\hfill 0  & \text{otherwise}. \\
 			\end{cases}
		\end{equation}
		Then ${\rm{supp}}(\psi_{\theta}) \subset V_i^{M}(\theta^4, \theta)$. Since $\psi_{\theta}(x) = 1$ for $x \in V_i^{M}(\theta^3, \theta^2)$, for any $\tau \in \ccal^{\infty}_{c}(M)$, we have
		\begin{equation}\label{ell_eqn_1}
			\norm{\tau}_{\mathscr{C}^{0}(V_i^{M}(\theta^3, \theta^2))} \leq \norm{ \psi_{\theta} \tau}_{\mathscr{C}^{0}}.
		\end{equation}
		We recall that the function $\rho_M$ was defined in (\ref{defn_rho}). For $k \in \nat$ and $q \geq 1$, the \textit{weighted Sobolev space} $\textbf{L}_{{\rm{wtd}}}^{k,q}(\overline{M}, g^{TM})$ is defined as the space of $\textbf{L}_{{\rm{wtd}}}^{k,q}$-functions $\tau$ on $\overline{M}$, where
		\begin{equation}
			\norm{\tau}_{\textbf{L}_{{\rm{wtd}}}^{k,q}}^{q} := \int_M \rho_M(x)^2 \Big( \big| \tau(x) \big|^q + \ldots + \big| \nabla^k \tau(x) \big|^q_{h} \Big) d \vol_{M}(x) < \infty.
		\end{equation}
		We define the \textit{Sobolev norm} $\norm{\cdot}_{H^{k}}$ by
		\begin{equation}
			\norm{\tau}_{H^{k}}^{2} := \int \Big( \big| \tau(x) \big|^2 + \ldots + \big| \nabla^k \tau(x) \big|^2_{h} \Big) d \vol_{M}(x).
		\end{equation} 
		From \cite[\S 4A and Lemme 4.5]{Biq} (cf. also \cite[Lemma 2.6]{Auvr} and \cite[Lemma 4.2]{AuvThese}), we get
		\begin{equation}\label{ell_eqn_2323}
			\norm{\tau}_{\mathscr{C}^{0}} \leq C \norm{\tau}_{\textbf{L}_{{\rm{wtd}}}^{3,1}},	
		\end{equation}
		for some $C > 0$ and for all $\tau \in \ccal^{\infty}_c(M)$.
		We fix $k \in \nat$.
		By induction from (\ref{ell_eqn_2323}), similarly to the induction step in \cite[Proposition 4.1]{Auvr} (cf. \cite[(4.11)]{Auvr}), there is $C > 0$ such that
		\begin{equation}\label{ell_eqn_2}
			\big \| \nabla^k \tau \big\|_{\mathscr{C}^{0}} \leq C \norm{\tau}_{\textbf{L}_{{\rm{wtd}}}^{3 + k,1}}.	
		\end{equation}
		We remark that by (\ref{reqr_poincare}), for any $\theta \in ]0,1/2]$, we have
		\begin{equation}\label{bound_hold}
			\int_{V_i^{M}(\theta^4, \theta)} \rho_M(x)^4 dv_M(x) = \int_{\theta^4 < |z| < \theta} \frac{\imun dz d \overline{z}}{4 |z|^2} = 3 \pi |\ln \theta|. 
		\end{equation}
		Now, by Hölder's inequality, (\ref{ell_eqn_2}) and (\ref{bound_hold}), we have
		\begin{multline}\label{ell_eqn_3}
			\big \| \nabla^k \tau \big \|_{\mathscr{C}^{0}} 
			\leq 
			\Big( 
				\int_{V_i^{M}(\theta^4, \theta)} \rho_M(x)^4 dv_M(x)		
			\Big)^{1/2}
			\Big(
				\int \Big( \big| \tau(x) \big| + \ldots + \big| \nabla^{3 + k} \tau(x) \big|_{h} \Big)^2
			\Big)^{1/2}
			\\
			\leq
			16 C \sqrt{| \ln \theta|} \norm{h}_{H^{3 + k}}.
		\end{multline}
		The hyperbolic Laplacian $\laplcomp$ associated with $g^{T\dd^*}$ is given by
		\begin{equation}\label{defn_hyp_lapl}
			\laplcomp := \Big| z_i^{M} \ln |z_i^{M}|^2 \Big|^2 \frac{\partial^2}{\partial z_i^{M} \partial \overline{z}_i^{M}}.
		\end{equation}
		By repeating the proof of \cite[(4.29)]{Auvr} for $p=0$, we get for some $C > 0$ and for any $\tau \in \ccal^{\infty}_{c}(M)$:
		\begin{equation}\label{ell_eqn_3705}
			\textstyle \norm{\tau}_{H^{2}}^{2} \leq C \Big( \norm{\tau}_{L^2}^{2} + \norm{\laplcomp \tau}_{L^2}^{2} \Big).
		\end{equation}
		As it is explained in \cite[Proposition 4.2]{Auvr}, by induction, there is $C > 0$ such that for any $\tau \in \ccal^{\infty}_{c}(M)$: 
		\begin{equation}\label{ell_eqn_37}
			\textstyle \norm{\tau}_{H^{2 + k}}^{2} \leq C \Big( \norm{\tau}_{L^2}^{2} + \norm{\laplcomp \tau}_{L^2}^{2} + \cdots + \norm{\laplcomp^{k + 1} \tau}_{L^2}^{2}  \Big).
		\end{equation}
		By (\ref{ell_eqn_3}) and (\ref{ell_eqn_37}), there is $C > 0$, such that for any  $\theta \in ]0,1/2]$, we have
		\begin{equation}\label{ell_eqn_3.7}
		\textstyle \norm{ \nabla^k (\psi_{\theta} f)}_{\mathscr{C}^{0}} \leq C \sqrt{ |\ln \theta |} \cdot \Big( \norm{\psi_{\theta} f }_{L^2}^{2} + \norm{\laplcomp (\psi_{\theta} f) }_{L^2}^{2} + \cdots + \norm{\laplcomp^{k + 1} (\psi_{\theta} f) }_{L^2}^{2}  \Big).
		\end{equation}
		By the fact that ${\rm{supp}}(\psi_{\theta}) \subset V_i^{M}(\theta^4, \theta)$, there is $C > 0$ such that for $\theta \in ]0, 1/2]$, we have 
		\begin{multline}\label{ell_eqn_4}
			\textstyle \norm{\psi_{\theta} f }_{L^2}^{2} + \norm{\laplcomp (\psi_{\theta} f) }_{L^2}^{2} +  \cdots + \norm{\laplcomp^{k + 1} (\psi_{\theta} f) }_{L^2}^{2} \\
			\textstyle
			\leq 
			C C_{\theta, k} 
			\Big( 
				\norm{f }_{L^2(V_i^{M}(\theta^4, \theta))}^{2} 
				+ \norm{\laplcomp f }_{L^2(V_i^{M}(\theta^4, \theta))}^{2}
				+ \cdots 			
				+ \norm{\laplcomp^{k + 1} f }_{L^2(V_i^{M}(\theta^4, \theta))}^{2}
			\Big),
		\end{multline}
		where 
		\begin{equation}\label{ell_eqn_444}
			\textstyle C_{\theta, k} := 1 + \max \Big\{  
				\norm{\partial \psi_{\theta}}_{\mathscr{C}^{0}},  
				\norm{\laplcomp \psi_{\theta}}_{\mathscr{C}^{0}}, 
				\norm{\partial \laplcomp \psi_{\theta}}_{\mathscr{C}^{0}},  
				\cdots,
				\norm{\laplcomp^{k + 1} \psi_{\theta}}_{\mathscr{C}^{0}}
			\Big\}.
		\end{equation}
		Now, by (\ref{defn_hyp_lapl}), we see that there is $C > 0$, such that for any $\theta \in ]0,1/2]$, we have
		\begin{equation}\label{eqn_der_psi_thet}
			\big \| \laplcomp \psi_{\theta} \big \|_{\mathscr{C}^{0}} \leq 16 \norm{ g'' ( \ln |z_i^{M}|/\ln \theta ) }_{\mathscr{C}^{0}} \leq C.
		\end{equation}
		Similarly, for any $k \in \nat$, there is $C  > 0$, such that for any $\theta \in ]0, 1/2]$, we have $C_{\theta, k} \leq C$. Now, for $\theta = \sqrt{|z_i^{M}(x)|}$, we have $V_i^{M}(\theta^4, \theta) \subset B^M(x, \alpha)$ and 
		\begin{equation}\label{eqn_triv_est_ell}
			\big| \nabla^k f(x) \big|_h \leq \big\| \nabla^k(\psi_{\theta} f) \big\|_{\ccal^0}.
		\end{equation}				
		Thus, by (\ref{ell_eqn_1}), (\ref{ell_eqn_3.7}), (\ref{ell_eqn_4}), (\ref{eqn_triv_est_ell}), we get (\ref{lem_ell_est_eqn}) for $n = 0$ and any $k \in \nat$.
		\par \textbf{Now we work for any $n \in \integ$, $k \in \nat$.}
		We prove the general result by reducing it to $n = 0$.
		We denote $z := z_i^{M}$. By (\ref{defn_sigma}), we have (cf. \cite[(4.30)]{Auvr}):
		\begin{equation}\label{lapl_nontriv}
			\laplcomp^{E_M^{\xi, n}} \sigma = \laplcomp f \cdot \Big( \frac{dz \otimes s_{D_M}}{z} \Big)^n - n \overline{z} \ln |z|^2 \frac{\partial f}{\partial \overline{z}} \Big( \frac{dz \otimes s_{D_M}}{z} \Big)^n.
		\end{equation}
		As $\norm{dz \otimes s_{D_M}/z}_M = |\ln |z|^2|$ over $V_i^{M}(\epsilon_0)$, we have
		\begin{equation}\label{eq_app_111}
			\norm{ \laplcomp^{E_M^{\xi, n}} \sigma}_{L^2} \geq 
				\Big\lVert \big| \ln |z|^2 \big|^n \laplcomp f \Big\rVert_{L^2}
				 - 
				|n| \Big\lVert \overline{z} \big| \ln |z|^2 \big|^{n+1} \frac{\partial f}{\partial \overline{z}} \Big\rVert_{L^2}.
		\end{equation}
		By (\ref{reqr_poincare}), (\ref{defn_hyp_lapl}) and integration by parts, we get
		\begin{multline}\label{eq_app_113}
			\Big\| \overline{z} \big| \ln |z|^2 \big|^{n+1} \frac{\partial f}{\partial \overline{z}} \Big\|_{L^2}^{2} 
			\\
			= 
			\Big\langle \big| \ln |z|^2 \big|^n \laplcomp f, \big| \ln |z|^2 \big|^n f \Big\rangle
			+
			2n \Big\langle \overline{z} \big| \ln |z|^2 \big|^{n+1} \frac{\partial f}{\partial \overline{z}} , \overline{z} \big| \ln |z|^2 \big|^n f \Big\rangle.
		\end{multline}
		By (\ref{eq_app_113}) and Cauchy inequality, for $\epsilon > 0$, we get
		\begin{equation}\label{eq_app_112}
			\Big\| \overline{z} \big| \ln |z|^2 \big|^{n+1} \frac{\partial f}{\partial \overline{z}} \Big\|_{L^2}^{2} 
			 \leq 
			\epsilon \norm{\big| \ln |z|^2 \big|^{n} \laplcomp f}_{L^2}^{2} + 
			\frac{n^2 + 1}{\epsilon}
			\norm{\big| \ln |z|^2 \big|^{n} f}_{L^2}^{2}.
		\end{equation}
		By (\ref{eq_app_111}) and (\ref{eq_app_112}), applied for $\epsilon = 1/(2n)$, we conclude that there is $C > 0$ such that (cf. \cite[(4.33)]{Auvr})
		\begin{equation}\label{eqn_ell_2_1}
			\textstyle \big\| \laplcomp^{E_M^{\xi, n}} \sigma \big\|_{L^2} 
			\geq
			C \Big( 
				\big\| \big| \ln |z|^2 \big|^{n} \laplcomp f \big\|_{L^2}
				-   (n^4 + 1) \big\| \big| \ln |z|^2 \big|^{n} f \big\|_{L^2}
			\Big).
		\end{equation}
		\par \textbf{We would like to prove a similar statement, that for some $C > 0$, the following holds: }
		\begin{multline}\label{eqn_ell_2_2}
			\big\| ( \laplcomp^{E_M^{\xi, n}})^2 \sigma \big\|_{L^2} 
			\geq
			C \Big( 
				\big\| \big| \ln |z|^2 \big|^{n} \laplcomp^2  f \big\|_{L^2} \\
				 - (n^4 + 1) \big\| \big| \ln |z|^2 \big|^{n} \laplcomp f \big\|_{L^2} 
				 - (n^8 + 1) \big\| \big| \ln |z|^2 \big|^{n} f \big\|_{L^2}
			\Big).
		\end{multline}
		The estimation (\ref{eqn_ell_2_2}) would follow from (\ref{lapl_nontriv}) if we will be able to prove that there is $C > 0$ such that for any $\epsilon \in ]0,1]$ and $\tau \in \ccal^{\infty}(M)$, we have
		\begin{align}			
			&
			\Big\| \overline{z} \big| \ln |z|^2 \big|^{n+1} \frac{\partial (\laplcomp \tau)}{\partial \overline{z}} \Big\|_{L^2}
			\leq \epsilon \norm{ \big| \ln |z|^2 \big|^{n} \cdot \laplcomp^2 \tau }_{L^2}  
			\nonumber
			\\
			&
			\qquad  \qquad  \qquad  \qquad 	\qquad  \qquad  \qquad  \qquad   \qquad  \quad 	 \label{ell_est_2} 	
			+
			\frac{n^2 + 1}{\epsilon} \norm{ \big| \ln |z|^2 \big|^{n} \cdot \laplcomp \tau }_{L^2},
			\\ 
			& \Big\lVert \overline{z} \big| \ln |z|^2 \big|^{n + 1} \frac{\partial}{\partial \overline{z}} \Big( \overline{z} \ln |z|^2 \frac{\partial \tau}{\partial \overline{z}} \Big) \Big\rVert_{L^2}^{2}		
			\leq \epsilon \norm{  \big| \ln |z|^2 \big|^{n} \cdot \laplcomp^2 \tau }_{L^2} 
			\nonumber
			\\ 			\label{ell_est_1} 
			 & \qquad  \qquad  \qquad  \qquad \qquad 
			 + \frac{C(n^2 + 1)}{\epsilon}  \Big(  \norm{  \big| \ln |z|^2 \big|^{n} \cdot \laplcomp \tau }_{L^2} + \norm{ \big| \ln |z|^2 \big|^{n} \tau}_{L^2}  \Big),
			\\ \nonumber
			 &
			 \Big\|  \big| \ln |z|^2 \big|^{n} \laplcomp \Big( \overline{z} \ln |z|^2 \frac{\partial \tau}{\partial \overline{z}} \Big) \Big\|_{L^2} 
			\leq \epsilon \norm{  \big| \ln |z|^2 \big|^{n} \laplcomp^2 \tau }_{L^2}  
			\\ \label{ell_est_3} 
			&						
			  \qquad  \qquad  \qquad  \qquad \qquad + \frac{C(n^2 + 1)}{\epsilon}  \Big( \norm{  \big| \ln |z|^2 \big|^{n} \laplcomp \tau }_{L^2} +  \norm{ \big| \ln |z|^2 \big|^{n} \tau}_{L^2}  \Big).
		\end{align}
		\par \textbf{Let's prove (\ref{ell_est_2}).} It simply follows from (\ref{eq_app_112}) by making a substitution $\tau \mapsto \laplcomp \tau$.
		\par \textbf{Let's prove (\ref{ell_est_1}) and (\ref{ell_est_3}).}
		First of all, by (\ref{eq_app_112}), we have
		\begin{multline}\label{ell_est_3_suppl}
			\Big\lVert \overline{z} \big| \ln |z|^2 \big|^{n+1} \frac{\partial}{\partial \overline{z}} \Big( \overline{z} \ln |z|^2 \frac{\partial \tau}{\partial \overline{z}} \Big) \Big\rVert_{L^2}^{2}
			\\			 
			 \leq \frac{1}{2} \Big\lVert \big| \ln |z|^2 \big|^{n} \laplcomp
				\Big( 
					\overline{z} \ln |z|^2 \frac{\partial \tau}{\partial \overline{z}}
				\Big) \Big\rVert_{L^2}
			+
			\frac{2n^2  + 1}{2} \Big\lVert
					\overline{z} \big| \ln |z|^2 \big|^{n + 1} \frac{\partial \tau}{\partial \overline{z}}
				\Big\rVert_{L^2}.
		\end{multline}
		By (\ref{defn_hyp_lapl}), we have
		\begin{equation}\label{eqn_ell_aux_25705}
			\laplcomp \Big( \overline{z} \ln |z|^2 \frac{\partial \tau}{\partial \overline{z}} \Big)
			=
			\Big| z \ln |z|^2 \Big|^2 \frac{\partial}{\partial \overline{z}} 
				\Big(
				\frac{\overline{z}}{z} \cdot \frac{\partial \tau}{\partial \overline{z}} + \overline{z} \ln |z|^2 \frac{\partial^2 \tau}{\partial z \partial \overline{z}}
				\Big)
			\end{equation}
			However, we have
			\begin{equation}\label{eqn_ell_aux_25710}
				z \ln |z|^2 \frac{\partial \tau}{\partial \overline{z}} = \frac{\partial}{\partial \overline{z}} \Big( z \ln |z|^2 \tau \Big) - \frac{z \tau}{\overline{z}}.
			\end{equation}
			Now, from (\ref{eqn_ell_aux_25705}) and (\ref{eqn_ell_aux_25710}), we have
			\begin{equation}\label{eqn_ell_aux_257}	
			\laplcomp \Big( \overline{z} \ln |z|^2 \frac{\partial \tau}{\partial \overline{z}} \Big)
			=
			\overline{z} \ln |z|^2 \frac{\partial}{\partial \overline{z}}
			\Big(
				\laplcomp \tau + \overline{z} \ln |z|^2 \frac{\partial \tau}{\partial \overline{z}}
			\Big)
			-
			\laplcomp \tau
			- 
			\overline{z} \ln |z|^2 \frac{\partial \tau}{\partial \overline{z}}.
		\end{equation}
		By (\ref{eq_app_112}), (\ref{ell_est_2}) and (\ref{eqn_ell_aux_257}), for any $\epsilon \in ]0,1]$, we get
		\begin{multline}\label{eq_fin_a20}
			\Big\lVert \big |\ln |z|^2 \big|^{n} \laplcomp \Big( \overline{z} \ln |z|^2 \frac{\partial \tau}{\partial \overline{z}} \Big) \Big\rVert_{L^2} 
			\leq 
			\epsilon
			\norm{ \big| \ln |z|^2 \big|^{n} \cdot \laplcomp^2 \tau}_{L^2}
			+
			\frac{C (n^2 + 1)}{\epsilon}
			\norm{ \big| \ln |z|^2 \big|^{n} \cdot \laplcomp \tau}_{L^2}
			\\
			+
			(n^2 + 1) \norm{ \big| \ln |z|^2 \big|^{n} \tau}_{L^2}
			 +
			\Big\lVert \overline{z} \big| \ln |z|^2 \big|^{n+1} \frac{\partial}{\partial \overline{z}} \Big( \overline{z} \ln |z|^2 \frac{\partial \tau}{\partial \overline{z}} \Big)\Big\rVert_{L^2}.
		\end{multline}
		By  (\ref{eq_app_112}), (\ref{ell_est_3_suppl}) and (\ref{eq_fin_a20}), for $\epsilon \in ]0,1]$ we have
		\begin{multline}\label{ell_est_3_supplasd}
			\Big\lVert \overline{z} \big| \ln |z|^2 \big|^{n+1} \frac{\partial}{\partial \overline{z}} \Big( \overline{z} \ln |z|^2 \frac{\partial \tau}{\partial \overline{z}} \Big) \Big\rVert_{L^2}
			\\
			 \leq 
			 \frac{1}{2} \Big\lVert \overline{z} \big| \ln |z|^2 \big|^{n+1} \frac{\partial}{\partial \overline{z}} \Big( \overline{z} \ln |z|^2 \frac{\partial \tau}{\partial \overline{z}} \Big) \Big\rVert_{L^2}
			 +
			\epsilon \norm{  \big| \ln |z|^2 \big|^{n} \cdot \laplcomp^2 \tau}_{L^2}  
			\\
			 + \frac{C (n^2 + 1)}{\epsilon} \Big(   \norm{  \big| \ln |z|^2 \big|^{n} \cdot \laplcomp \tau}_{L^2} + \norm{ \big| \ln |z|^2 \big|^{n} \tau}_{L^2}  \Big).
		\end{multline}		
		By (\ref{ell_est_3_supplasd}), we get (\ref{ell_est_1}).
		By (\ref{ell_est_1}) and (\ref{eq_fin_a20}), we get (\ref{ell_est_3}). 
		\par 
		By  (\ref{defn_sigma}), (\ref{eqn_ell_2_1}), (\ref{eqn_ell_2_2}), there is $C > 0$ such that for any $n \in \integ$, $j = 0,1,2$, we have
		\begin{equation}\label{ell_est_sigma_log}
			\sum_{i=0}^{j}  (n^{4j - 4i}+1) \norm{( \laplcomp^{E_M^{\xi, n}})^i \sigma}_{L^2} 
			\geq 
			C \sum_{i=0}^{j} (n^{4j - 4i}+1) \norm{ \big| \ln |z|^2 \big|^{n} \cdot \laplcomp^i f}_{L^2}.
		\end{equation} 
		By induction, using (\ref{lapl_nontriv}), (\ref{eq_app_112}), (\ref{ell_est_3_suppl}) and (\ref{eqn_ell_aux_257}), we get (\ref{ell_est_sigma_log}) for any $j \in \nat$ in the same way as we did it for $j = 2$.
		\par By (\ref{defn_hyp_lapl}), we remark that for any $j \in \nat$, there is $C > 0$ such that we have
		\begin{equation}\label{eqn_major_der_log}
			\begin{aligned}
				& | \partial \laplcomp^j ((\ln |z|^2)^n) / (\ln |z|^2)^n |_h \leq C |n|^{2j + 1}, 
				&& | \laplcomp^{j + 1} ( (\ln |z|^2)^n) / (\ln |z|^2)^n | \leq C n^{2j + 2}.
			\end{aligned}
		\end{equation}
		By (\ref{eqn_major_der_log}), for any $j \in \nat$, there is $C > 0$, such that we have
		\begin{equation}\label{eqn_comm_log}
			\sum_{i=0}^{j} (n^{4j - 4i}+1) \norm{ \big| \ln |z|^2 \big|^{n} \laplcomp^i f}_{L^2} \geq C \sum_{i=0}^{j} (n^{4j - 4i}+1) \norm{ \laplcomp^i ( \big| \ln |z|^2 \big|^{n}  f)}_{L^2}.
		\end{equation}
		However, if we apply (\ref{lem_ell_est_eqn}) for functions, i.e. $n = 0$, there is  $C>0$ such that
		\begin{equation}\label{eq_appl_n0}
			\Big| \nabla^k \big(|\ln |z|^2|^{n}  f\big) \Big|_{h} \leq C \rho_M(x)
			\sum_{i=0}^{k} \norm{ \laplcomp^i \big(|\ln |z|^2|^{n} f\big)}_{L^2}
		\end{equation}
		Now, by \cite[(4.16)]{Auvr}, we see that for $\tau \in \ccal^{\infty}(M)$,
		\begin{equation}
			\sigma' = \tau \otimes \Big( \frac{dz \otimes s_D}{z} \Big)^n \otimes (dz)^{k} \otimes (d \overline{z})^{l},
		\end{equation}				
		we have
		\begin{equation}\label{eqn_nabla_full}
			\nabla \sigma' = 
			\Big( 
			d \tau
			+
			(n+k+l) \tau
				\Big( \frac{dz}{z \ln |z|^2} + \frac{d \overline{z}}{\overline{z} \ln |z|^2} \Big)
			\Big)
			\otimes 
			\Big( \frac{dz \otimes s_D}{z} \Big)^n \otimes (dz)^{k} \otimes (d \overline{z})^{l}.
		\end{equation}
		From  (\ref{defn_sigma}), (\ref{eqn_nabla_full}), for any $k \in \nat$, there is $C > 0$ such that
		\begin{equation}\label{eqn_nabla_bounds000}
			\big| \nabla^k \sigma(x) \big|_{h} \leq C \sum_{i = 0}^{k} (1 + |n|^{k - i}) \Big|  (\ln |z|^2)^n \nabla^i \big( f(x) \big) \Big|_{h}.
		\end{equation}
		By (\ref{eqn_major_der_log}), we have
		\begin{equation}\label{eqn_nabla_bounds00}
			\Big|  (\ln |z|^2)^n \nabla^i \big( f(x) \big) \Big|_{h} \leq C \sum_{j = 0}^{i} (1 + |n|^{i - j}) \Big| \nabla^j \big( (\ln |z|^2)^n f(x) \big) \Big|_{h}.
		\end{equation}
		Thus, by (\ref{eqn_nabla_bounds000}) and (\ref{eqn_nabla_bounds00}), we deduce
		\begin{equation}\label{eqn_nabla_bounds}
			\big| \nabla^k \sigma(x) \big|_{h} \leq C \sum_{i = 0}^{k} (1 + |n|^{k - i}) \Big| \nabla^i \big( (\ln |z|^2)^n f(x) \big) \Big|_{h}.
		\end{equation}
		From (\ref{ell_est_sigma_log}), (\ref{eqn_comm_log}), (\ref{eq_appl_n0}) and (\ref{eqn_nabla_bounds}), we conclude (\ref{lem_ell_est_eqn}) for any $n \in \integ$, $k \in \nat$.
	\end{proof}
	\begin{rem}\label{rem_a1}
		Let's define the function
		\begin{equation}
			\psi_{\theta, 0}(x) = 
			\begin{cases} 
      			\hfill g(|z(x)| / \theta) & \text{for }x \in V_{i}^{M}(\theta), \\
      			\hfill 0  & \text{otherwise}. \\
 			\end{cases}
		\end{equation}
		We can do the same analysis as in Lemma \ref{lem_ell_est_eqn} with function $\psi_{\theta, 0}$ instead of $\psi_{\theta}$. 
		We have ${\rm{supp}} (\psi_{\theta, 0}) \subset V_i^{M}(\theta, 4 \theta)$.
		Similarly to (\ref{bound_hold}), for some $C>0$, we have
		\begin{equation}\label{bound_hold2}
			\int_{V_i^{M}(\theta, 4\theta)} \rho_M(x)^4 dv_M(x) \leq C. 
		\end{equation}
		Thus, instead of (\ref{ell_eqn_3}), we have
		\begin{equation}\label{ell_eqn_32222}
			\norm{ \psi_{\theta, 0} f}_{\mathscr{C}^{0}} \leq C \norm{\psi_{\theta, 0} f }_{H^3}.
		\end{equation}
		Also, like in (\ref{ell_eqn_444}), we denote 
		\begin{equation}\label{ell_eqn_4445}
			\textstyle C_{\theta, 0} := 1 + \max \Big\{  
				\norm{\partial \psi_{\theta, 0}}_{\mathscr{C}^{0}},  
				\norm{\laplcomp \psi_{\theta, 0}}_{\mathscr{C}^{0}}, 
				\norm{\partial \laplcomp \psi_{\theta, 0}}_{\mathscr{C}^{0}},  
				\norm{\laplcomp^2 \psi_{\theta, 0}}_{\mathscr{C}^{0}}  
			\Big\}.
		\end{equation}
		Like in (\ref{eqn_der_psi_thet}), there is $C>0$ such that for any $\theta \in ]0,1/2]$, we have 
		\begin{equation}
			\norm{\laplcomp \psi_{\theta, 0}}_{\mathscr{C}^{0}} 
			= 
			\Big\| \big| \ln|z|^2 \big|^2 |z|^2 \Big( \frac{|z|^2}{\theta^4} g''(|z_i(x)| / \theta) + \frac{g'(|z(x)| / \theta)}{\theta^2} \Big) \Big\|_{\mathscr{C}^{0}} \leq C |\ln \theta|^2.
		\end{equation}				
		Similarly, there is $C>0$ such that for any $\theta \in ]0,1/2]$, we have $C_{\theta, 0} \leq C (\ln \theta)^4$.
		By imitating the proof of Lemma \ref{lem_ell_est_eqn} with $\psi_{\theta}$ replaced by $\psi_{\theta, 0}$, we deduce that for any $\beta > 1$, there are $\epsilon, C >0$ such that for any $n \in \integ$, $\sigma \in \ccal^{\infty}(\overline{M}, E_M^{\xi, n})$, $x \in V_i^{M}(\epsilon)$, we have
		\begin{equation}\label{eqn_est_ell_dill}
			\big| \sigma(x) \big|_{h} \leq C \rho_M(x)^8
			\sum_{i=0}^{2} (n^{8 - 4i} + 1) \big\lVert (\laplcomp^{E_M^{\xi, n}})^i \sigma \big\rVert_{L^2(V_i^{M}(\beta^{-1} |z_i^{M}(x)|, \beta |z_i^{M}(x)|))}.
		\end{equation}
		We will use this estimation in Section \ref{app_2} to prove tightness of a certain family of flattenings.
	\end{rem}

\subsection{Existence of tight families of flattenings}\label{app_2}
	\par In this section, we prove by an explicit construction that for any $n \in \integ$, there are $n$-tight families of flattenings  $g^{TM}_{\rm{f}, n, \theta}$, $\norm{\cdot}_{M}^{\rm{f}, \theta}$ (see Definition \ref{defn_tight}).
	Before giving the formulas, let's describe in words this construction. 
	In this section, as in Section \ref{sect_tight}, we suppose that $(\xi, h^{\xi})$ is trivial near the cusps.
	\par
	The metrics $g^{TM}_{\rm{f}, n, \theta}$, $\norm{\cdot}_{M}^{\rm{f}, \theta}$ are equal to $g^{TM}$, $\norm{\cdot}_{M}$ over  $M \setminus (\cup_i V_i^{M}(\theta))$. The metric $\norm{\cdot}_{M}^{\rm{f}, \theta}$ gets “flattened" over the set $V_i^{M}(\theta^2, \theta)$ (see (\ref{notat_ci})), so that it differs from  $\norm{\cdot}_{M}$ by a multiplication by a function, which is bounded by a constant independent of $\theta$. Over $V_i^{M}(\theta^2)$, $\norm{\cdot}_{M}^{\rm{f}, \theta}$ is flat with a normalization as in \textit{property 2} of tight families. 
	The metric $g^{TM}_{\rm{f}, n, \theta}$ verifies $g^{TM}_{\rm{f}, n, \theta} \otimes (\, \norm{\cdot}_{M}^{\rm{f}, \theta})^{2n} = g^{TM} \otimes \, \norm{\cdot}_{M}^{2n}$  over $V_i^{M}(\theta^4, \theta)$, i.e. \textit{property 3} of tight families is trivially satisfied  over $V_i^{M}(\theta^4, \theta)$. It gets “flattened" over the set $V_i^{M}(\theta^4/4, \theta^4)$, so that it differs from  $g^{TM}$ by a bounded function. Finally, over $V_i^{M}(\theta^4/4)$ it is flat with a normalization constant such that the Riemannian manifolds $(V_i^{M}(\theta^4/4), g^{TM}_{\rm{f}, n, \theta})$ and $(D(2), (dx^2 + dy^2)/(\ln \theta)^2)$ are isometric up to a multiplication by a constant independent of $\theta$.
	\par Now let's make this description more precise by giving explicit formulas. 
	Let $\phi: [1, + \infty[ \to [0,1]$ be some smooth decreasing function satisfying
	\begin{equation}
		\phi(u) = 
		\begin{cases} 
      		\hfill 1 & \text{ for }u \in [1,5/4], \\
      		\hfill 0  & \text{ for } u \in [7/4,+ \infty[. \\
 		\end{cases}
	\end{equation}
	Let  $\chi: [0, 1] \to [0, 1]$ be a some smooth increasing function satisfying
	\begin{equation}
		\chi(u) = 
		\begin{cases} 
      		\hfill 0 & \text{ for }u \in [0,1/2], \\
      		\hfill 1  & \text{ for } u \in [3/4,1]. \\
 		\end{cases}
	\end{equation}
	\par Fix $\theta \in ]0, 1/2]$. We denote by $\norm{\cdot}_{M}^{\rm{f}, \theta}$ the Hermitian norm on $\omega_M(D)$ such that $\norm{\cdot}_{M}^{\rm{f}, \theta}$ coincides with $\norm{\cdot}_{M}$ over $M \setminus (\cup_i V_i^{M}(\theta))$, and 
	\begin{equation}\label{defn_tight_1}
      \norm{ dz_i^{M} \otimes s_{D_M} / z_i^{M}}_{M}^{\rm{f}, \theta}(x) = 	| \ln \theta | \cdot \Big( \frac{\ln |z_i^{M}(x)|}{\ln \theta} \Big)^{\phi(\ln |z_i^{M}(x)|/ \ln \theta)}  \qquad \text{over }  x \in V_i^{M}(\theta).
	\end{equation}
	 Let the metric $g^{TM}_{\rm{f}, n, \theta}$ coincide with $g^{TM}$ over $M  \setminus (\cup_i V_i^{M}(\theta))$, and over  $V_i^{M}(\theta)$ be induced by
	\begin{equation}\label{defn_tight_2}		
	\begin{aligned}
      	& 
      	\Big( \frac{\ln |z_i^{M}(x)|}{\ln \theta} \Big)^{2n(1-\phi(\ln |z_i^{M}(x)|/ \ln \theta))} \cdot
      	\frac{ \imun dz_{i}^{M} d\overline{z}_{i}^{M}}{|z_{i}^{M}|^2 | \ln |z_{i}^{M}| |^{2}}
      	&& \text{over }  x \in V_i^{M}(\theta^4, \theta), 
      	\\ 
      	& 
      	\bigg[
      		\Big( \frac{\ln |z_i^{M}(x)|}{\ln \theta} \Big)^{2n}
      			\cdot
      		\frac{\theta^8 |\ln \theta^{4}|^2}{|z_{i}^{M} \ln |z_i^{M}||^2}
      	\bigg]^{\chi(|z_i^{M}|^2 / \theta^8)}
      	\cdot \frac{ \imun dz_{i}^{M} d\overline{z}_{i}^{M}}{\theta^8 |\ln \theta^{4}|^2}
      	&& \text{over }  x \in  V_i^{M}(\theta^4).
    \end{aligned}
	\end{equation}
	Then we see that the metrics $g^{TM}_{\rm{f}, n, \theta}$, $\norm{\cdot}_{M}^{\rm{f}, \theta}$  verify the description given in the beginning of the section.
	We note by $\laplcomp^{E_M^{\xi, n}}_{\rm{f}, \theta}$, $\norm{\cdot}_{L^2, \theta}$ the Kodaira Laplacian and the $L^2$-norm induced by $g^{TM}_{\rm{f}, n, \theta}$, $h^{\xi}$, $\norm{\cdot}_M^{\rm{f}, \theta}$.
	\begin{thm}\label{tight_fl_const}
		The flattenings $g^{TM}_{\rm{f}, n, \theta}, \norm{\cdot}_{M}^{\rm{f}, \theta}$, $\theta \in ]0, 1/2]$ are $n$-tight.
	\end{thm}
	\begin{proof}
		We see directly from (\ref{defn_tight_1}) and (\ref{defn_tight_2}) that all the requirements for tightness are trivially satisfied with only one exception - the estimate (\ref{lem_ell_est_eqn_A}). By Lemma \ref{lem_ell_est}, (\ref{defn_tight_1}) and (\ref{defn_tight_2}), it is enough to prove (\ref{lem_ell_est_eqn_A}) for $x, x' \in V_i^{M}(\theta^{1/2})$, $i = 1, \ldots, m$.
		\par 
		\textbf{Let's prove (\ref{lem_ell_est_eqn_A}) for $x, x' \in M \setminus (\cup_i V_i^{M}(\theta^3))$.}
		By (\ref{dist_hyp_comp}), for $x \in M \setminus (\cup_i V_i^{M}(\theta^3))$, we have $B(x, \ln(4/3)) \subset M \setminus (\cup_i V_i^{M}(\theta^4))$. Thus, it is enough to prove that for all $n \in \integ$, there is $C>0$ such that for any $\sigma \in \ccal^{\infty}(\overline{M}, E_M^{\xi, n})$, $x \in M \setminus (\cup_i V_i^{M}(\theta^3))$, we have
		\begin{equation}\label{tight_1010}
			\big| \sigma(x) \big|_{h, \theta} \leq C \rho_{M, \theta}(x) \sum_{i=0}^{2}
			\big\lVert (\laplcomp^{E_M^{\xi, n}}_{{\rm{f}}, \theta})^{i} \sigma \big\lVert_{L^2(B^M(x, \ln(4/3))), \theta}.
		\end{equation}
		\par
		\textbf{Let's prove (\ref{tight_1010}) for $n=0$.}
		By Lemma \ref{lem_ell_est}, there is $C > 0$ such that for any $\sigma \in \ccal^{\infty}(\overline{M}, E_M^{\xi, n})$, $x \in M \setminus (\cup_i V_i^{M}(\theta^3))$, we have
		\begin{equation}\label{eq_ex_fl_1}
			\big| \sigma(x) \big|_{h} \leq C \rho_M(x)
			\sum_{i = 0}^{2}
			\big\lVert (\laplcomp^{E_M^{\xi, n}})^{i} \sigma \big\lVert_{L^2(B^M(x, \ln(4/3)))}.
		\end{equation}
		Now, by (\ref{defn_tight_2}), there is $C > 0$ such that for all $\theta \in ]0, 1/2]$, we have
		\begin{equation}\label{eq_ex_fl_2}
			1 \leq ( g^{TM} / g^{TM}_{{\rm{f}}, \theta} ) \leq C \qquad \text{over} \qquad M \setminus (\cup_i V_i^{M}(\theta^4/4)).
		\end{equation}
		We recall that $\rho_M(x) \leq \rho_{M, \theta}(x)$ for $x \in M \setminus (\cup_i V_i^{M}(\theta^4/4))$, and  since $n = 0$, we have
		\begin{equation}\label{eq_ex_fl_3}
			\laplcomp^{E_M^{\xi, n}}_{{\rm{f}}, \theta} = ( g^{TM} / g^{TM}_{{\rm{f}}, \theta} ) \cdot \laplcomp^{E_M^{\xi, n}}.
		\end{equation}
		Now, the following identity holds
		\begin{multline}\label{eq_ex_fl_332}
			\Big(
				\frac{\ln |z|}{\ln \theta}
			 \Big)^{-(1-\phi(\ln |z|/ \ln \theta))}
			 \frac{\partial}{\partial z} \Big(
				\frac{\ln |z|}{\ln \theta}
			 \Big)^{(1-\phi(\ln |z|/ \ln \theta))} 
			\\ = 
			 \Big(
			 	\frac{1-\phi(\ln |z|/ \ln \theta)}{2 z \ln |z|} -
			 	\ln \Big( \frac{\ln |z|}{ \ln \theta} \Big) \frac{\phi'(\ln |z|/ \ln \theta))}{2 z \ln \theta} 
			 \Big).
		\end{multline}		
		By (\ref{eq_ex_fl_332}), we deduce that there such $C > 0$ such that for any $x \in M \setminus (\cup_i V_i^{M}(\theta^3))$, we have
		\begin{equation}\label{eq_ex_fl_3332}
			\big| \partial ( g^{TM} / g^{TM}_{{\rm{f}}, \theta})(x) \big|_{h, \theta} < \theta.
		\end{equation}
		Similarly (\ref{defn_tight_2}), (\ref{eq_ex_fl_3}), (\ref{eq_ex_fl_332}), we see that there is $C > 0$ such that for any $\theta \in ]0, 1/2]$, $x \in M \setminus (\cup_i V_i^{M}(\theta^3))$, we have
		\begin{equation}\label{eq_ex_fl_333}
			\big| \laplcomp ( g^{TM} / g^{TM}_{{\rm{f}}, \theta} ) \big| < C, 
			\quad
			\big| g^{TM} / g^{TM}_{{\rm{f}}, \theta} \big| < C,
		\end{equation}
		where $\laplcomp$ is the hyperbolic Laplacian defined in (\ref{defn_hyp_lapl}).
		By (\ref{eq_ex_fl_1}), (\ref{eq_ex_fl_2}), (\ref{eq_ex_fl_3}), (\ref{eq_ex_fl_3332}) and (\ref{eq_ex_fl_333}), we get  (\ref{tight_1010}) for $n = 0$.
		\par 
		\textbf{Now let's prove (\ref{tight_1010}) for $n \in \integ$.}
		The idea is to reduce (\ref{tight_1010}) to the case $n=0$ by using the same trick as we did in the proof of Lemma \ref{lem_ell_est}, i.e. identity of type (\ref{lapl_nontriv}).
		By (\ref{defn_tight_1}), (\ref{eq_ex_fl_332}), there is $C > 0$ such that for $i=0,1,2$:
		\begin{equation}\label{eq_ex_fl_6}
			\big| \laplcomp^i \big( \ln \big\lVert (dz_i^{M} \otimes s_{D_M} ) / z_i^{M} \big\rVert_{M}^{\rm{f}, \theta} \big) \big|, \quad
			\big| \partial \laplcomp^i \big( \ln \big\lVert (dz_i^{M} \otimes s_{D_M} ) / z_i^{M} \big\rVert_{M}^{\rm{f}, \theta} \big) \big|_{h, \theta} \leq C,
		\end{equation}
		over $V_i^{M}(\theta^4, \theta)$.
		Also, by (\ref{defn_tight_1}) we conclude that there is $C > 0$ such that
		\begin{equation}
			1 \leq (\, \norm{\cdot}_{M}^{\rm{f}, \theta}  /  \norm{\cdot}_{M}) \leq C \qquad \text{over} \qquad M \setminus (\cup_i V_i^{M}(\theta^4/4)).
		\end{equation}
		Then by (\ref{eq_ex_fl_6}), similarly to (\ref{eq_app_111}) and (\ref{eqn_ell_2_2}), for any $n \in \integ$, there is $C > 0$ such that for any $\theta' \in ]0,1/2]$ and $\phi_{\theta'}$, defined in (\ref{defn_psiteta}), $f$ defined in (\ref{defn_sigma}) we have
		\begin{align}\label{eq_ex_fl_7}
			& 
			\big \lVert \laplcomp^{E_M^{\xi, n}}_{{\rm{f},  \theta}} (\psi_{\theta'} \sigma) \big \rVert_{L^2, \theta} 
			\geq C |\ln \max(\theta, \theta')|^{n} \Big(
				\norm{ \laplcomp_{\rm{f}, \theta} (\psi_{\theta'} f)}_{L^2, \theta}
				 - 
				\norm{\psi_{\theta'} f}_{L^2, \theta}
			\Big),
			\\
			& 
			\big \lVert ( \laplcomp^{E_M^{\xi, n}}_{{\rm{f}, \theta}})^2 (\psi_{\theta'} \sigma) \big \rVert_{L^2, \theta} 
			\geq
			C |\ln \max(\theta, \theta')|^n \Big( 
				\big \lVert \laplcomp_{{\rm{f}, \theta}}^{2} (\psi_{\theta'} f) \big \rVert_{L^2, \theta}	 
				 - \norm{\laplcomp_{{\rm{f}, \theta}} (\psi_{\theta'} f)}_{L^2, \theta} 
				 - \norm{\psi_{\theta'} f}_{L^2, \theta}
			\Big),
			\nonumber
		\end{align}
		where $\laplcomp_{{\rm{f}, \theta}}$ is is the Kodaira Laplacian on functions on $\dd$ induced by the metric $((z_i^{M})^{-1})^* g^{TM}_{{\rm{f}, n, \theta}}$.
		Thus, by (\ref{eq_ex_fl_7}) there is $C > 0$ such that for any $\theta' \in ]0,1/2]$ and $\theta \in ]0,1/2]$, we have
		\begin{equation}\label{eq_ex_fl_711}
			\textstyle \sum_{i=0}^{2} \big \lVert ( \laplcomp^{E_M^{\xi, n}}_{{\rm{f}, \theta}})^i (\psi_{\theta'} \sigma) \big \rVert_{L^2, \theta} \geq
			C |\ln  \max(\theta, \theta')|^n \sum_{i=0}^{2} \big \lVert \laplcomp_{{\rm{f}, \theta}}^{i} (\psi_{\theta'} f) \big \rVert_{L^2, \theta}.
		\end{equation}
		By (\ref{tight_1010}) for $n=0$ and (\ref{eq_ex_fl_711}), we get for any $n \in \integ$ and $x \in M \setminus (\cup_i V_i^{M}(\theta^3))$:
		\begin{equation}\label{eq_ex_fl_712}
			\textstyle \sum_{i=0}^{2} \big \lVert ( \laplcomp^{E_M^{\xi, n}}_{{\rm{f}, \theta}})^i (\psi_{\theta'} \sigma) \big \rVert_{L^2, \theta} \geq
			C   \rho_{M, \theta}(x)  |\ln  \max(\theta, \theta')|^n  \big| \psi_{\theta'} f(x) \big|.
		\end{equation}
		For any $n \in \integ$ and $x \in M \setminus (\cup_i V_i^{M}(\theta^3))$, we have
		\begin{equation}\label{eq_ex_fl_713}
			|\ln \theta|^n  \big| \psi_{\theta} f(x) \big| \geq |\sigma(x)|_{h, \theta}.
		\end{equation}
		By (\ref{eq_ex_fl_712}) and (\ref{eq_ex_fl_713}), we get (\ref{tight_1010}) for any $x \in M \setminus (\cup_i V_i^{M}(\theta^3))$ and $n \in \integ$. This implies (\ref{lem_ell_est_eqn_A}) for $x, x' \in M \setminus (\cup_i V_i^{M}(\theta^3))$ and $n \in \integ$.
		\par \textbf{We suppose $x = x' = z, z \in V_i^{M}(\theta^4, \theta^3)$.}
		Since the Hermitian line bundle $(\omega_M(D), \norm{\cdot}_M^{{\rm{f},  \theta}})$ is trivial over $V_i^{M}(\theta^2)$, without loss of generality we may and we will suppose $n = 0$.
		By (\ref{eqn_est_ell_dill}):
		\begin{equation}\label{lem_ell_est_eqnaa}
			\big| \sigma(z, z) \big| \leq C (\ln \theta)^{8}
			\sum_{i,j = 0}^{2}
			\big\lVert (\laplcomp^{E_M^{\xi, n}}_{z})^{i}(\laplcomp^{E_M^{\xi, n}}_{z'})^{j} \sigma \big\lVert_{L^2(V_i^{M}(|z|/2, 2 |z|) \times V_i^{M}(|z|/2, 2 |z|))}.
		\end{equation}
		Similarly to (\ref{eq_ex_fl_332}), we have
		\begin{multline}\label{lem_ell_est_eqne}
			\Big(
				\frac{|z \ln |z||^2}{\theta^8 |\ln \theta^4|^2}
			 \Big)^{-\psi(|z|^2/ \theta^8)}
			\frac{\partial}{\partial z} \Big(
				\frac{|z \ln |z||^2}{\theta^8 |\ln \theta^4|^2}
			 \Big)^{\psi(|z|^2/ \theta^8)}  
			\\
			= 
			 \Big(
			 	\frac{\psi(|z|^2/ \theta^8)(\ln |z| + 1/2)}{z \ln |z|}
			 	+
			 	\ln \Big(
				\frac{|z \ln |z||^2}{\theta^8 |\ln \theta^4|^2}
			 \Big) \psi'(|z|^2/ \theta^2)) \overline{z} \theta^{-8} 
			 \Big),
		\end{multline}	
		By (\ref{defn_tight_2}), (\ref{lem_ell_est_eqne}) and the fact that $n = 0$, similarly to (\ref{eq_ex_fl_333}), there is $C > 0$ such that for any $\theta \in ]0, 1/2]$, we have
		\begin{equation}\label{lem_ell_est_eqnaabbcc}
			| \laplcomp ( g^{TM} / g^{TM}_{{\rm{f}}, \theta} ) | < C (\ln \theta)^2, \qquad
			| \partial ( g^{TM} / g^{TM}_{{\rm{f}}, \theta} ) |_{h, \theta} < C |\ln \theta|, 
		\end{equation}
		over $M \setminus (\cup_i V_i^{M}(\theta^4/4))$.
		Since $V_i^{M}(|z|/2, 2 |z|) \subset M \setminus (\cup_i V_i^{M}(\theta^4/4))$, by (\ref{eq_ex_fl_2}), (\ref{eq_ex_fl_3}) and (\ref{lem_ell_est_eqnaabbcc}), similarly to (\ref{tight_1010}), from (\ref{lem_ell_est_eqnaa}), we get (\ref{lem_ell_est_eqn_A}) for $x = x',x  \in V_i^{M}(\theta^4, \theta^3)$.
		\par 
		\textbf{Now, we suppose $x = x', x \in V_i^{M}(\theta^4)$.}
		First of all, we recall that by Sobolev inequality and standard elliptic estimates, we have for some $C > 0$ and any $h \in \ccal^{\infty}(D(2))$, $x \in D(1)$:
		\begin{equation}\label{eq_ex_fl_4}
			\textstyle \big| h(x) \big| \leq C \sum_{i=0}^{2} \big\lVert \laplcomp_{{\rm{st}}}^{i} h \big\rVert_{L^2_{{\rm{st}}}},
		\end{equation}
		where $\norm{\cdot}_{L^2_{{\rm{st}}}}$ is the $L^2$-norm induced by the standard Euclidean metric $g_{\rm{st}}$ over $D(2)$, and $\laplcomp_{{\rm{st}}}$ is the Kodaira Laplacian induced by $g_{\rm{st}}$.
		Analogically to (\ref{eq_ex_fl_3}), the estimation (\ref{eq_ex_fl_4}) implies that 
		\begin{equation}\label{eq_ex_fl_5}
			\textstyle \big| h(x) \big| \leq C (\ln \theta)^4  \sum_{i=0}^{2} \big\lVert \laplcomp_{{\rm{st}}, \theta}^{i} h \big\rVert_{L^2_{{\rm{st}}}, \theta},
		\end{equation}
		where $\norm{\cdot}_{L^2_{{\rm{st}}}, \theta}$ is the $L^2$-norm induced by the rescaled Euclidean metric $g_{{\rm{st}}, \theta}$, and $\laplcomp_{{\rm{st}}, \theta}$ is the Kodaira Laplacian induced by $g_{{\rm{st}}, \theta}$ for
		\begin{equation}
			g_{{\rm{st}}, \theta} := \frac{dx^2 + dy^2}{(\ln \theta)^2}.
		\end{equation}
		By (\ref{defn_tight_2}), the spaces $(D(2), g_{{\rm{st}}, \theta})$ and $(V_i^{M}(\theta^4), g^{TM}_{\rm{f}, n, \theta})$ are isometric up to a constant independent of $\theta$. Thus, by (\ref{eq_ex_fl_5}), we deduce (\ref{lem_ell_est_eqn_A}) for $x = x', x \in V_i^{M}(\theta^4)$. 
		\par Now, all the cases have been considered, thus, the proof of Theorem \ref{tight_fl_const} is finished.
	\end{proof}

		\bibliographystyle{abbrv}

\Addresses

\end{document}